\documentclass[12pt]{amsart}
 \usepackage{mathtools}
\mathtoolsset{showonlyrefs,showmanualtags}

\usepackage[backrefs]{amsrefs}

 \usepackage[top=1.5in, bottom=1.5in, left=1.25in, right=1.25in]	{geometry}
\usepackage[all]{xy}
\usepackage{amsmath}
\usepackage{amssymb}
\usepackage{amsthm}
\usepackage{amscd}

\numberwithin{equation}{section}

\newtheorem{theorem}[equation]{Theorem}
\newtheorem{definition}[equation]{Definition}
\newtheorem{proposition}[equation]{Proposition}
\newtheorem{corollary}[equation]{Corollary}
\newtheorem{lemma}[equation]{Lemma}

\newtheorem{remark}[equation]{Remark}

\usepackage{hyperref} 
\hypersetup{
    colorlinks=true,       
    linkcolor=blue,          
    citecolor=magenta,        
    filecolor=magenta,      
    urlcolor=cyan           
}

\begin{document}
\title[An Approach to Pointwise Ergodic Theorems]
{A Unified Approach to Two Pointwise Ergodic Theorems: Double Recurrence and Return Times}

\author{Ben Krause}
\address{
Department of Mathematics,
University of Bristol\\
Bristol BS8 1QU}
\email{ben.krause@bristol.ac.uk}
\date{\today}


\maketitle

\begin{abstract}
We present a unified approach to extensions of Bourgain's Double Recurrence Theorem and Bourgain's Return Times Theorem to integer parts of the Kronecker sequence, emphasizing stopping times and metric entropy. Specifically, we prove the following two results for each $\alpha \in \mathbb{R}$:
\begin{itemize}
    \item For each $\sigma$-finite measure-preserving system, $(X,\mu,T)$, and each $f,g \in L^{\infty}(X)$, for each $\gamma \in \mathbb{Q}$ the bilinear ergodic averages
    \[ \frac{1}{N} \sum_{n \leq N} T^{\lfloor \alpha n \rfloor } f \cdot T^{\lfloor \gamma n \rfloor} g \]
    converge $\mu$-a.e.; and
    \item For each aperiodic and countably generated measure-preserving system, $(Y,\nu,S)$, and each $g \in L^{\infty}(Y)$, there exists a subset $Y_{g} \subset Y$ with $\nu(Y_{g})= 1$ so that for all $\gamma \in \mathbb{Q}$ and $\omega \in Y_{g}$, for any auxiliary $\sigma$-finite measure-preserving system $(X,\mu,T)$, and any $f \in L^{\infty}(X)$, the ``return-times" averages
\[ \frac{1}{N} \sum_{n \leq N} T^{\lfloor \alpha n \rfloor} f \cdot S^{\lfloor \gamma n \rfloor } g(\omega) \]
converge $\mu$-a.e.
\end{itemize}
Moreover, in both cases the sets of convergence are identical for all $\gamma \in \mathbb{Q}$.
\end{abstract}

 \setcounter{tocdepth}{1}
\tableofcontents 

\section{Introduction}

The study of pointwise convergence of ergodic averages dates back to Birkhoff's Pointwise Ergodic Theorem \cite{BI}, with the field dramatically advanced by Bourgain in the late 80s-early 90s \cite{B0,B1,B00,B2,B3}, in the 2010s by Mirek, Stein, and collaborators \cite{MST1, MST2, MSZK}, and most recently by the author, in collaboration Mirek and Tao \cite{KMT}, respectively Mousavi, Ter\"{a}v\"{a}inen and Tao \cite{KMTT}, Ter\"{a}v\"{a}inen \cite{Joni}, and Kosz et.\ al.\ \cite{KMPW}. The goal of this paper is to provide a unified approach to two types of theorems of Bourgain, for which special cases already exist; one of which is proven using techniques from ``hard analysis," Theorem \ref{t:DR}, and the other of which, Theorem \ref{t:RT}, is proven using ``softer" methods. We will adapt the ``hard approach" of \cite{B3} to handle both results using a single argument.

\medskip

Here is the set-up:

\medskip

A \emph{measure-preserving system}, $(X,\mu,T)$ is a $\sigma$-finite measure space, equipped with an invertible measure-preserving transformation; we use the notation
\[ T f(x) := f(T x) \]
throughout.



Our focus will be two classes of pointwise ergodic theorems. 

\begin{theorem}[Bourgain, $\alpha \in \mathbb{Q}$ Case]\label{t:DR}
    Let $(X,\mu,T)$ be a measure-preserving system.
    Then for all $\alpha \in \mathbb{R}$, all $\gamma \in \mathbb{Q}$, and all $f,g \in L^{\infty}(X)$
    \begin{align}\label{e:avg} \frac{1}{N} \sum_{n \leq N } T^{\lfloor \alpha n \rfloor } f \cdot T^{\lfloor \gamma n \rfloor } g
    \end{align}
    converges $\mu$-a.e.
\end{theorem}

In what follows, as per Bourgain, we focus on the case $\gamma = -1$; the only changes to handle the general case are notational, after splitting into residue classes and exploiting periodicity. As additional motivation for our work here, by a Rokhlin tower argument, one can reduce the issue of pointwise convergence of two commuting shifts
\begin{align}\label{e:commute} \frac{1}{N} \sum_{n \leq N} T^n f \cdot S^n g, \; \; \; f,g \in L^{\infty}(X)
\end{align}
to the case where $T$ generates a flow, and $S = T^{\alpha}$ for irrational $\alpha$, so one is interested in understanding
\begin{align}\label{e:commute} \frac{1}{N} \sum_{n \leq N} T^n f \cdot T^{ \alpha n} g, \; \; \; f,g \in L^{\infty}(X);
\end{align}
but this reduction does not seem to be viable at present.

Our second theorem is a result on return times:

\begin{theorem}[Bourgain, $\alpha \in \mathbb{Q}$ Case]\label{t:RT}
    Let $(Y,\nu,S)$ be a countably generated probability space equipped with an aperiodic measure-preserving transformation, and let $\alpha \in \mathbb{R}$. Then for for all $g \in L^{\infty}(Y)$ there exists $Y_{g} \subset Y$ with $\nu(Y_{g}) = 1$ so that the following holds for every $\gamma \in \mathbb{Q}$ and $\omega \in Y_{g}$: for every measure-preserving system, $(X,\mu,T)$, and every $f \in L^{\infty}(X)$
    \[\frac{1}{N} \sum_{ n\leq N }   S^{\lfloor \gamma n \rfloor}  g(\omega) \cdot T^{\lfloor \alpha n \rfloor} f \]
    converges $\mu$-a.e.

    In particular, if $g = \mathbf{1}_E$ for non-trivial $E \subset Y$ and
    \[ E_N(\omega) := \{ n \leq N : S^n \omega \in E \} ,\]
then whenever $S$ is ergodic with respect to $\nu$
    \[ \frac{1}{|E_N(\omega)|} \sum_{n \in E_N(\omega)}  T^{\lfloor \alpha n \rfloor} f \]
    converges $\mu$-a.e. for every measure-preserving system $(X,\mu,T)$, and every $f \in L^{\infty}(X)$.\footnote{By the pointwise ergodic theorem, $\omega$-almost surely $|E_N(\omega)| = N \nu(E) \cdot ( 1 + o_{N \to \infty;\omega}(1))$; see \S \ref{sss:O} below for a review of ``little-Oh" notation. }
\end{theorem}


\subsection{History}
Following Furstenberg's ergodic theoretic proof of Szemer\'{e}di's theorem \cite{Sz}, the issue of multiple recurrence of ergodic averages came into being, namely the limiting behavior of
\begin{align}\label{e:mult} \frac{1}{N} \sum_{n \leq N} T^n f_1 \cdot \dots \cdot T^{kn} f_k
\end{align}
for $f_i \in L^{\infty}(X)$; almost $30$ years after Szemer\'{e}di proved his theorem by establishing an appropriate lower bound on the $\liminf$ of \eqref{e:mult}, Host-Kra \cite{HK} proved that the multiple averages \eqref{e:mult} converged in $L^2(X)$-norm, with the $k=3$ case first appearing in \cite{HK0}.

The issue of pointwise convergence of \eqref{e:mult} is much more delicate, with only the $k=2$ case currently available \cite{B3}, although the case where $f,g$ are in more general $L^p(X)$-spaces
was later established by Lacey \cite{L}.

Prior to establishing his so-called ``double recurrence" result \cite{B3}, Bourgain began working on the issue of pointwise convergence of weighted ergodic averages
\[ \frac{1}{N} \sum_{n \leq N} w(n) \cdot T^n f\]
for certain classes of weights, $w$; this interest dates at least as far back as \cite{Bp}, in which case the role of the weight $w$ was played by the von Mangoldt function. But, this line of inquiry much precedes Bourgain, beginning with the Wiener-Wintner Theorem \cite{WW}, which dictates that for each $g \in L^1(X)$, $\nu$-a.e.\ the following averages converge for all $\theta \in \mathbb{T}$,
\[ \omega \mapsto \frac{1}{N} \sum_{n \leq N} g(S^n \omega) \cdot e^{2 \pi i n \theta}, \]
this result follows directly from Theorem \ref{t:RT} by specializing $(X,\mu,\mathbb{T})$ to the systems
\[ \{ (\mathbb{T},dx,x \mapsto x- \theta) \}_{\theta \in \mathbb{T}}\]
and choosing each time $f(x) = e^{2 \pi i x}$. In particular, Bourgain was interested in the case where the weights are given by
\[ w(n) = g(S^n\omega), \; \; \; g \in L^{\infty}(Y), \]
initially in the case where $g = \mathbf{1}_E$ and $S : Y \to Y$ is ergodic. In unpublished work, \cite{B00}, Bourgain first essentially established the $\alpha \in \mathbb{Q}$ case of Theorem \ref{t:RT} using methods of harmonic analysis, with a complete proof and subsequent simplification using ergodic-theoretic methods later appearing as joint work with Furstenberg-Katznelson-Ornstein in the appendix to \cite{B2}; a later proof using joinings is due to Rudolph \cite{Rud}. One feature of our hard analytic approach is that we extend all arguments involved to the case of real $\alpha \in \mathbb{R}$, with the case where $g$ is constant, namely, the linear averages
\[ \frac{1}{N} \sum_{n \leq N} T^{\lfloor \alpha n \rfloor} f, \]
first addressed in \cite{B2}. One theme of this paper is to emphasize the connection between Theorem \ref{t:DR} and Theorem \ref{t:RT}, in which the presence of the second function $g$ introduces a ``half-degree" of multi-linearity.

Our proof uses many of the ingredients of \cite{B00,B3}, which emphasize the role of harmonic analysis on $\mathbb{T}$ and entropic considerations deriving initially from martingale theory, and in particular allows us to \emph{quantify} convergence more precisely than in the treatments of \cite{B00,B3, Rud}; there is a direct analogy to Bourgain's work on polynomial ergodic theorems \cite{B2}, which will be partially explored in \S \ref{s:ex} below. We refer the reader to \cite{Tao} for a fuller discussion of these techniques in the work of Bourgain, and to \cite{BK} for a more focused discussion of their use in pointwise ergodic theory.

\subsection{High-Level Approach}
The mechanism for proving pointwise convergence without the presence of a dense subclass of examples derives from Bourgain's work, \cite{B0,B1,B2}, in which one seeks to show that a certain ``oscillation operator"
\[ O_K^2 := \frac{1}{K} \sum_{k \leq K} \max_{N_k \leq N \leq  N_{k+1}} |A_N - A_{N_k}|^2\]
satisfies
\[ \| O_K \|_{L^2(X)} = o_{K \to \infty}(1) \]
for any given increasing subsequence $\{ N_k\}$; above $\{ A_M \}$ are a sequence of $1$-bounded functions defined on a probability space, $(X,\mu)$.

But, establishing a norm estimate can itself be challenging, as one is forced to encounter the situation where 
\[ \max_{N_k \leq N \leq  N_{k+1}} |A_N(x) - A_{N_k}(x)| \]
is large for many $k \leq K$ at the same time, yielding large contributions to $O_K(x)$; we instead choose to quantify convergence as weakly as possible, using the following elementary lemma; see \S \ref{sss:O} for relevant notation.

\begin{lemma}\label{l:convex}
    Suppose that $\{ A_M \}$ are a sequence of $1$-bounded functions defined on a probability space, $(X,\mu)$, which satisfy the smoothness condition
    \[ \|A_M - A_N\|_{L^{\infty}(X)} \lesssim \frac{|N-M|}{N}.\]
Consider the quantity
\begin{align}  C_{\tau,B,H}^{(X,\mu)} &:= \sup \Big\{ K : \text{ there exists } M_0 < M_1 < \dots < M_K \leq H/100, \\
& \text{ so that } \mu(\{ x : \max_{M_{k-1} \leq M \leq M_k} |A_M - A_{M_k}| \gg \tau \}) \gg \tau^B \Big\}, \end{align}
and we restrict all times to be of the form $\{ 2^{n/B_\tau} \}, \ B_\tau = O(\tau^{-1})$.
If for each $\tau,B,H$
\[ C_{\tau,B,H}^{(X,\mu)} \lesssim_{\tau,B} 1,\]
with the implicit constant independent of $H$, then $\{ A_M \}$ converge $\mu$-a.e.
\end{lemma}
\begin{proof}
    The proof is by contradiction, where we begin by assuming that $\{ A_M \}$ fail to converge, and then violate the condition $C_{\tau,B,H}^{(X,\mu)} \lesssim_{\tau,B} 1$.

Since $\{ A_M \}$ fail to converge, there exists some $\tau > 0$
so that    \[ \mu(\{ x : \limsup_{M,N} |A_M - A_N| \gg \tau\}) \gg \tau. \]
In this case, by a brief measure-theoretic argument, we could extract a finite subsequence of arbitrary length $K$ so that for each $1 \leq k \leq K$
\[ \mu(\{ x : \max_{M_{k-1} \leq M \leq M_k} |A_M - A_{M_k}| \gg \tau\}) \gg \tau;\]
there is no loss of generality here in assuming that $M_k,M \in \{ 2^{n/B_\tau} \}$ by the smoothness condition on $\{ A_M\}$. So, let $K$ be an unbounded but fixed integer, and select $H = H(K) \gg M_K$; summing over $k \leq K$ yields a strict upper,
\[ K \tau \ll \sum_{k \leq K} \mu(\{ x : \max_{M_{k-1} \leq M \leq M_k} |A_M - A_{M_k}| \gg \tau\}) \leq C_{\tau,B,H}^{(X,\mu)} \lesssim_{\tau, B} 1, \]
for the desired contradiction.
\end{proof}

By applying Calder\'{o}n's transference principle \cite{C1} (and additional measure-theoretic arguments in the case of Return Times) one is able to pass, essentially, to a discrete system
\[ ([H],\frac{d\#}{H}, x \mapsto x-1), \; \; \; [H] := \{ 1,\dots,H\}, \]
where the task reduces to estimating the oscillation of discrete bilinear averaging operators on the integers. The analysis here derives from an appropriate understanding of the single-scale situation. Using martingale-based bilinear entropy techniques, whose linear counterparts were crucially used in \cite{B2}, one aggregates the bilinear averages into scales and locations where the single-scale example becomes typical; borrowing from the time-frequency literature, we refer to these collections as ``branches." The argument is concluded by showing that -- up to exceptional set considerations -- there are only a ``few" branches which present. This is accomplished by way of an energy pigeonholing argument adapted from time frequency techniques, largely deriving from \cite{D00}.

\medskip

The structure of the paper is as follows:
\begin{itemize}
\item We begin in \S \ref{ss:not} with some preliminary material;
\item In \S \ref{s:ex}, we present a modern take on Bourgain's approach, in the context of Double Recurrence, Theorem \ref{t:DR}, as this case admits a simpler transference argument to pass to the integers;
\item In the next section we quantify convergence, which allows us to pass to the measure-preserving system of $(\mathbb{Z},d\#,x \mapsto x-1)$; we largely accomplish these reductions in \S \ref{s:QC} and \S \ref{s:transtoZ}, with an additional measure-theoretic argument deferred to the Appendix, \S \ref{s:appendix};
  \item In \S \ref{ss:tools} we collect some tools from harmonic analysis that we will appeal to: estimates for trigonometric polynomials, martingale-based entropy estimates, orthogonality tools from time frequency analysis, and some dyadic grid techniques;
  \item In the next section, \S \ref{s:floor}, we ``resolve" the floor function, simplifying the analysis to follow;
    \item In \S \ref{ss:fol} we introduce a selection algorithm which will eventually allow us to reduce our analysis to (essentially) that considered in the examples of \S \ref{s:ex};
    \item In \S \ref{s:oscmax} we present our proof strategy;
    \item In the remaining sections, \S \ref{s:DRmax}, \S \ref{s:RTmax}, \S \ref{s:DRosc}, and \S \ref{s:RTosc}, we complete the program outlined in the previous \S \ref{s:oscmax}, thereby concluding the proof.
\end{itemize}

\subsection{Acknowledgements}
I would like to thank Tim Austin, Vitaly Bergelson, and M\'{a}t\'{e} Wierdl for their help with the Appendix, $\S \ref{s:appendix}$, Thomas Jordan for useful discussions on topological entropy, Joel Moreira for useful discussions about the Wiener-Wintner Theorem, and  Tim Austin and Joni Ter\"{a}v\"{a}inen for their helpful editorial comments  which improved the exposition throughout. I am especially grateful to Ciprian Demeter for generously sharing \cite{D00}, from which I first learned Theorem \ref{t:DR}, and for more than a decade of encouraging me to pursue the topic of pointwise ergodic theory.

\section{Preliminaries}\label{ss:not}
\subsection{Notation}
We use
\[ e(t) := e^{2 \pi i t} \]
throughout to denote the complex exponential, and for $\theta \in \mathbb{T}$, we use the abbreviation
\begin{align}\label{e:mod} \text{Mod}_\theta g(n) := e(\theta n) \cdot g(n).\end{align}

We let 
\[ B(r) := \{ |x| \leq r \} \subset \mathbb{R} \]
denote Euclidean balls, and use the Minkowski sum
\begin{align}\label{e:Minksum} \Lambda + B(r) \end{align}
to denote $\{ x : \text{dist}(x,\Lambda) \leq r \}$; we use the notation 
\[ [M] := \{1,\dots,M\}.\]
Throughout we use $\Lambda$ (sometimes with subscripts, etc.) to denote finite subsets of frequencies
\[ \Lambda \subset \mathbb{T}, \; \; \; |\Lambda| < \infty.\]

Sometimes, we will need to measure distance with respect to different Banach spaces: if ($X, \| \cdot \|_X$) is such a Banach space, we use
\begin{align}\label{e:normball} B_{\| \cdot \|_X}(r) := \{ \| v \|_X \leq r \} \subset X 
\end{align}
to denote the ball inside of $X$.

We use
\[ M_{\text{HL}}f(x) := \sup_{r > 0} \frac{1}{2r} \int_{|t| \leq r} |f(x-t)| \ dt, \; \; \; := \sup_{N \geq 0} \frac{1}{2N+1} \sum_{|n| \leq N} |f(x-n)|  \]
to denote the Hardy-Littlewood Maximal Function, with context determining whether we use the real-variable/integer formulation.

For intervals $I \subset \mathbb{R}$ (with length $\geq 1$), we use $|I|$ to denote side length, thus
\[ |I| = \# \{ n \in I \cap \mathbb{Z} \} + O(1)\]
so that whenever $\{ a_n \}$ are $1$-bounded
\[ \frac{1}{\#(I \cap \mathbb{Z})} \sum_{n \in I \cap \mathbb{Z}} a_n = \frac{1}{|I|} \sum_{n \in I \cap \mathbb{Z}} a_n + O(|I|^{-1}) =: \frac{1}{|I|} \sum_{n \in I} a_n + O(|I|^{-1}),\]
and thus from the perspective of pointwise convergence there is no loss of generality in using this normalizing convention; see Subsection \ref{sss:O} for a review of big-Oh notation. We will let
\[ C \cdot I \]
denote the interval that is concentric to $I$ but with $C$ times the length.

In \S \ref{s:ex} we will use the heuristic notation
\[ f \; \; \; ``= " \; \; \;  g \]
to denote moral equivalence: up to tolerable errors, $f$ and $g$ exhibit the same type of behavior.

Finally, throughout the remainder of the paper, $\alpha \in \mathbb{R}$ will be arbitrary but fixed, and we will allow all implicit constants to depend on $\alpha$. We will set $D = D_\alpha \in \mathbb{Z}$ satisfying
\begin{align}\label{e:Da} 1/20 \leq |\alpha/D| \leq 1/10, \; |\alpha| \geq 1/10, \; \; \; D =1 \text{ otherwise.}\end{align}

\subsection{Relevant Parameters}\label{ss:param}
Throughout the remainder of the paper, $0 < \tau, \delta \ll 1$ will denote small parameters, and $A_0 \gg 1$ will be a large constant which we are free to increase finitely many times throughout the argument; we set $\alpha_0 := A_0^{-1}$

Beginning in \S \ref{ss:fol} we will introduce two moderately large parameters $R \ll W = R^{10}$, with 
\begin{align}\label{e:RW} R = \begin{cases} 2^{\tau^{-A_0}} & \text{ if } \log^{-1}(1/\delta) \geq \tau^4 \\
\delta^{-\alpha_0} & \text{ otherwise.} \end{cases}  \end{align}
We will also need a somewhat small parameter
$t$, defined
\begin{align}\label{e:t} t = \begin{cases} \tau^{A_0} & \text{ if } \log^{-1}(1/\delta) \geq \tau^4 \\
\log^{-10}(1/\delta) & \text{ otherwise,} \end{cases} \end{align}
and finally set
\begin{align}\label{e:rho}
\rho = (\alpha_0 \cdot t)^2.
\end{align}

We will also select an odd integer $\Delta \in 2 \mathbb{N} + 1$ of size 
\begin{align}\label{e:Delta} \Delta \approx t^{-10}.\end{align}

\subsection{Asymptotic Notation}\label{sss:O}
We will make use of the modified Vinogradov notation. We use $X \lesssim Y$ or $Y \gtrsim X$ to denote
the estimate $X \leq CY$ for an absolute constant $C$ and $X, Y \geq 0.$  If we need $C$ to depend on a
parameter, we shall indicate this by subscripts, thus for instance $X \lesssim_p Y$ denotes the estimate $X \leq C_p Y$ for some $C_p$ depending on $p$. We use $X \approx Y$ as shorthand for $Y \lesssim X \lesssim Y$. We use the notation $X \ll Y$ or $Y \gg X$ to denote that the implicit constant in the $\lesssim$ notation is extremely large, and analogously $X \ll_p Y$ and $Y \gg_p X$.

We also make use of big-Oh and little-Oh notation: we let $O(Y)$  denote a quantity that is $\lesssim Y$ , and similarly
$O_p(Y )$ will denote a quantity that is $\lesssim_p Y$; we let $o_{t \to a}(Y)$
denote a quantity whose quotient with $Y$ tends to zero as $t \to a$ (possibly $\infty$), and
$o_{t \to a;p}(Y)$
denote a quantity whose quotient with $Y$ tends to zero as $t \to a$ at a rate depending on $p$.

\subsection{Entropic Considerations}\label{ss:ent}
Below we develop three similar -- and linked -- notions of the size of a set, used in the study of topological entropy; a nice reference for the below is \cite[\S 7.2]{W}, see also \cite[\S 1.15]{Tbook}.

Given a finite collection of vectors inside of some (complete) normed vector space 
\[ V:=\{ \vec{v_i} \} \subset X \]
we define
\begin{align}\label{e:Ent}
    \text{Ent}(V; \epsilon)
\end{align}
to be the cardinality of the maximal $\epsilon$-separated subset of $\{ \vec{v_i}\} \subset V$, namely the cardinality of an \emph{$\epsilon$-net} inside $V$; this quantity is bounded by
\begin{align}\label{e:jump} N_\epsilon(V),
\end{align}
the \emph{$\epsilon$-jump-counting number of $V$}, the maximal $D$ so that there exists $\{ \vec{v_k} : 0 \leq k \leq D\}$ so that
\[ \| \vec{v_k} - \vec{v_{k-1}} \|_X > \epsilon, \; \; \; 1 \leq k \leq D.\]

We need two further notions of entropy:

First, the external $\epsilon$-covering number,
\begin{align}\label{e:net}
    \text{Ext}(V;\epsilon) := \min\{ D : V \subset \bigcup_{i \leq D}  \vec{w_i}  + B_{\| \cdot \|_X}(\epsilon) \},
\end{align}
where we use the Minkowski sum, see \eqref{e:normball}, and we do not insist that $\{ \vec{w_i} \} \subset V$; this quantity is bounded above by the internal $\epsilon$-covering number
\begin{align}\label{e:c-net}
    \text{Int}(V;\epsilon) := \min\{ D : V \subset \bigcup_{i \leq D}  \vec{v_i}  + B_{\| \cdot \|_X}(\epsilon) \},
\end{align}
where we here restrict the centers of our balls to actually live inside $V$. But, we have the useful comparison:
\begin{align}\label{e:comp}
\text{Int}(V;2\epsilon) \leq \text{Ext}(V;\epsilon); \; \; \;\text{Ent}(V;\epsilon) \leq \text{Int}(V;\epsilon/2); \; \;
\text{Int}(V;\epsilon) \leq \text{Ent}(V;\epsilon).
\end{align}


\section{Bourgain's Approach}\label{s:ex}
Bougain's argument begins at level of dynamical systems, by employing the \emph{uniform Wiener-Wintner Theorem}, to reduce to the case where $g \in \mathcal{K}(T)^{\perp}$ for $T$ ergodic
\begin{align}
    \sup_{\beta} |\frac{1}{N} \sum_{n \leq N}    T^n g \cdot e(-n \beta) | \to 0, \; \; \; \mu-\text{a.e.};
\end{align}
this can be exploited to reduce to the studying the integer model of the problem where $g$ satisfies
\[ \sup_{\beta, \text{ ``almost all" }|I| \geq N} \big| \frac{1}{|I|} \sum_{n \in I} g(n) \cdot e(-n \beta) \big| = o_{N \to \infty}(1);\]
see \cite[\S 2]{D00} for a rigorous formulation. 

In the more modern language of time frequency analysis, see \cite{L, D0, DOP}, this allows one to transfer to a problem in Euclidean space where the ``wave packet coefficients" are very small,
\begin{align}\label{e:wavepacket}
\langle \psi_s, g \rangle = o(|I_s|^{1/2}), \; \; \; s = I_s \times \omega_s
\end{align}
whenever $\psi_s$ is an $L^2$-normalized wave-packet adapted to the ``tile" $s := I_s \times \omega_s \subset \mathbb{R}^2$, see \S \ref{ss:tools} below. On the other hand, the time frequency approach exploits Carleson estimates on the collections
\[ \{ \langle \psi_s, g \rangle : s \} \]
via Bessel-type orthogonality estimates, rather than pointwise bounds as in \eqref{e:wavepacket}. The time-frequency approach is incredibly powerful, and allows one to prove extremely strong norm estimates: one is able to control oscillation in tandem with symmetry, by passing to a (maximally truncated) singular integral model, see \cite{L}. Bourgain's approach is poorly suited to proving norm estimates; on the other hand, it is able to reduce proving convergence of bilinear ergodic averages to proving a weak maximal estimate whenever $g \in \mathcal{K}(T)^{\perp}$; the passage to the integer model is straightforward in this case, where the modulation invariance
\begin{align}\label{e:modinv0} \big| \frac{1}{N} \sum_{n \leq N} (\text{Mod}_{\theta} f)(x-n) \cdot (\text{Mod}_{\theta} g)(x+n) \big| \; \; \; \text{ is independent of $\theta$},   \end{align}
see \eqref{e:mod},
dictates that the class of examples must be rich enough to accommodate 
\[ \{ \text{Mod}_\theta g : \theta \in \mathbb{T} \} \]
uniformly in $\theta$.

Taken together, we see that we are interested in studying
\begin{align} B_N(f,g)(x) := \frac{1}{N} \sum_{n \leq N} f(x-n) g(x+n), \; \; \; g = \sum_{\theta \in \Lambda} c_\theta \cdot  e(\theta \cdot) \cdot \varphi(\cdot/N)
\end{align}
with the aim to extract gain in
\[ \| c_\theta \|_{\ell^{\infty}(\Lambda)}, \]
as per the additive combinatorial inequality
\[ \| \mathbb{E}_{n \in \mathbb{Z}/N} f(x-n) g(x+n) \|_{\ell^1(\mathbb{Z}/N)} \leq \| g\|_{U^2(\mathbb{Z}/N)} \leq \| \hat{g} \|_{\ell^\infty}^{1/2}: \]
the only obstruction to the average-case estimate
\[ |\mathbb{E}_{n \in \mathbb{Z}/N} f(x-n) g(x+n)| \ll \delta \]
is if there exists some $\beta \in \mathbb{Z}/N$ so that $|\hat{g}(\beta)| \gtrsim \delta^2$.

But, for $g$ of the above form, one computes
\begin{align}\label{e:approx1}
    B_N(f,g)(x) \; \; \; ``=" \; \; \; \sum_{\theta \in \Lambda} c_\theta \cdot e(2 \theta x) \cdot \eta_N * (\text{Mod}_{-\theta} f)(x),
\end{align} 
where $\eta_N := N^{-1} \eta( \cdot /N)$ for smooth $\eta$, so that 
$\widehat{\eta_N}$ is essentially a smooth approximation to $\mathbf{1}_{B(1/N)}$. Above, we think of 
\[ \min_{\theta \neq \theta' \in \Lambda} |\theta - \theta'| \gg 1/N, \]
as 
\[ \| \theta - \theta' \|_{\mathbb{T}} \ll 1/N, \; \;\;|x| \leq N \; \; \; \Rightarrow e(\theta x) = e(\theta' x) + o(1). \]
We continue
\begin{align}
    &\eqref{e:approx1} \; \; \; ``=" \; \; \; \sum_{\theta \in \Lambda} c_\theta \cdot e(2 \theta x) \cdot \eta_N * f_\theta(x), \; \; \;  f_\theta := \phi_N * (\text{Mod}_{-\theta} f) \end{align}
for 
\[ \mathbf{1}_{|t| \leq 1/4N} \leq \widehat{ \phi_N} \leq \mathbf{1}_{|t| \leq 1/2N} \]
like $\eta_N$, so that the following vector-valued inequalities persist, see \eqref{l:rsum} for the second point
\begin{align}
    \| (\sum_{\theta \in \Lambda} |f_\theta|^2)^{1/2} \|_{\ell^2(\mathbb{Z})}^2 \lesssim \| f \|_{\ell^2(\mathbb{Z})}^2, \; \; \; 
    \| (\sum_{\theta \in \Lambda} |f_\theta|^2)^{1/2} \|_{\ell^\infty(\mathbb{Z})}^2 \lesssim \| f \|_{\ell^{\infty}(\mathbb{Z})}^2.
  \end{align}
This vector valued approach allows us to bound 
\begin{align*}
    \| B_N(f,g) \|_{\ell^2(|x| \leq N)}^2 & ``\lesssim" \; \; \;  N \cdot \min_{|x_0| \leq N} \ \sum_{\theta \in \Lambda} |c_\theta|^2 \cdot |\eta_N*f_\theta(x_0)|^2  \\
    & \qquad \leq \| c_\theta \|_{\ell^{\infty}(\Lambda)}^2 \cdot N \cdot \min_{|x_0| \leq N} \sum_{\theta \in \Lambda} |\eta_N*f_\theta(x_0)|^2 \\
    & \qquad \qquad \lesssim \| c_\theta \|_{\ell^{\infty}(\Lambda)}^2 \cdot N \cdot \| f \|_{\ell^{\infty}(\mathbb{Z})}^2.
\end{align*} 
In particular,
whenever $|f|,|g| \leq 1$,
\[ \| B_N(f,g) \|_{\ell^2(|x| \leq N)}^2 \lesssim \| c_\theta \|_{\ell^{\infty}(\Lambda)}^2 \cdot N.\]
A preliminary take-away from this computation is that we should be organizing 
\[ \{ c_\theta : \theta \in \Lambda\} \]
into level-sets
\begin{align}\label{e:level} \{ c_\theta : \theta \in \Lambda, |c_\theta| \approx \delta \} , \end{align}
and seeking gain in $\delta$, which suggests to us that the type of $\Lambda$ we expect are $1/N$-nets inside the \emph{large spectra} of our secondary functions, $g$, appropriately localized and normalized:
\begin{align}\label{e:spec}
 \text{Spec}_{\delta}(g_I) := \{ \beta : |\frac{1}{|I|} \sum_{n \in I} g(n) e(-n \beta)| \geq \delta \};   
\end{align}
in what follows, we will implicitly assume that all $c_\theta$ are of this magnitude. Note further that if
\[ \Lambda_{\delta}(I) \subset \text{Spec}_{\delta}(g_I)\]
is a $1/N$-net, then we have an absolute bound
\begin{align}\label{e:specsize}
|\Lambda_{\delta}(I)| \lesssim \delta^{-2}
\end{align}
by Lemma \ref{l:rsum}.

Moreover, to the extent that multi-frequency issues as in the polynomial ergodic theorem present, we expect the issue of combining scales to be entropy-driven, using the tools from $\ell^2(\Lambda)$-based entropy and variation, summarized for our purposes in \S \ref{ss:tools} below.

At the level of proof strategy, if we expect the analogy between the pointwise ergodic theorem along polynomial orbits \cite{B0,B1,B2} to persist, we further expect a splitting of the problem into an oscillation-based estimate (according to whether the parameter $\delta$ is ``large") and a maximal estimate (when $\delta$ is ``small"), with the aim of using orthogonality-based techniques in the former case, and gain in the parameter $\delta$ in the latter, see \S \ref{s:oscmax} below. The maximal estimate is the more fundamental, and so we accordingly focus our attention.

\subsection{Combining Scales}
With the single-scale case understood, we develop a tool kit that will allow us to enjoy multi-scale control of the averages in certain model cases; the argument proper will then derive from trying to reduce to these cases.

\subsubsection{Simple Case: One Parental Scale}
Suppose that $g$ is as above, and we are interested in showing the preliminary bound on $I_0$
\begin{align*}
    \sup_{\delta N \leq M \leq N} |B_M(f,g)(x)| = o_{\delta \to 0}(1)
\end{align*}
in a quantifiable sense.

At this point, we see that we need to control
\[ \sup_{\delta N \leq M \leq N} \big| \sum_{\theta \in \Lambda} e(2 \theta x) \cdot \eta_M*(c_\theta \cdot f_\theta)(x) \big|, \]
where we will impose the non-concentration condition
\[ \min_{\theta \neq \theta' \in \Lambda} |\theta - \theta'| \geq \frac{1}{\delta N} \]
to be able to maintain our vector-valued perspective. 

By Bourgain's multi-frequency maximal inequality \cite{B2}, used crucially in his polynomial ergodic theorem, we know that
\begin{align*}
    \| \sup_{\delta N \leq M \leq N} \big| \sum_{\theta \in \Lambda} e(2 \theta x) \eta_M*(c_\theta \cdot f_\theta)(x) \big| \|_{\ell^2(\mathbb{Z})} &\lesssim \log^2 |\Lambda| \cdot \| \big( \sum_{\theta \in \Lambda} |c_\theta \cdot f_\theta|^2 \big)^{1/2} \|_{\ell^2(\mathbb{Z})} \\
    & \qquad \lesssim \delta \log^2 |\Lambda| \cdot \| f \|_{\ell^2(\mathbb{Z})}.
\end{align*}

On the other hand, specializing our second function $g$ to be of this simplified form removes a fundamental layer of generality, in that $g$'s behavior is essentially uniform. In particular, whenever $|I| = |J| = M$ are both contained in $I_0$
\[ \text{Spec}_\delta( g \cdot \mathbf{1}_I) = \text{Spec}_{\delta}(g \cdot \mathbf{1}_J),\]
see \eqref{e:spec},
which leads to the case of constant coefficients $\{ c_\theta \}$, when in fact we expect
\begin{align}
B_M(f,g)(x)
\end{align}
to be determined by the large spectrum of $g$ at the appropriate scale and location. In particular, generalizing
\[ c_\theta \longrightarrow \frac{1}{|J|} \sum_{ n \in J } g(n) e(-n \theta) \]
according to scale and location, we have the heuristic that when $x \in I, \ |I| = M$,
\begin{align*}
     B_M(f,g)(x)  
\; \; \; ``&=" \; \; \; \sum_{\theta \in \Lambda_\delta(I)} e(2\theta x) \cdot \big( \frac{1}{|I|} \sum_{n \in I} g(n) e(-n \theta)  \big) \cdot  \eta_M*(\text{Mod}_{-\theta} f)(x) \\
    & \qquad ``=" \; \; \; \sum_{\theta \in \Lambda_\delta(I)} e(2\theta x) \cdot \big( \frac{1}{|I|} \sum_{n \in I} g(n) e(-n \theta)  \big) \cdot  \big( \frac{1}{|I|} \sum_{n \in I} f(n) e(-n \theta) \big) \\
    & \qquad \qquad ``=" \; \; \; 
    \sum_{\theta \in \Lambda_\delta(I)} e(2\theta x) \cdot \eta_M*(\text{Mod}_{-\theta}g)(x) \cdot  \eta_M*(\text{Mod}_{-\theta} f)(x)
\end{align*}
where, as above, $\Lambda_\delta(I)$ is a $1/|I|$-net inside $\text{Spec}_{\delta}(g_I)$; at this point the necessity of a bilinear perspective on entropy becomes clear.

\subsubsection{Combining Scales, Take Two}
To facilitate this argument, it is helpful to make use of the tree-structure of the class of dyadic intervals; henceforth, every interval considered will be dyadic, and we will rely on harmonic-analytic technicality to reduce to this case, see \S \ref{ss:tools} below. Regarding $\delta > 0$ as fixed, we adopt the temporary definition
\begin{align*}
B_I(f,g)(x) := \sum_{\theta \in \Lambda_\delta(I)} e(2\theta x) \cdot \big( \mathbb{E}_I \text{Mod}_{-\theta} g \big) \cdot \big( \mathbb{E}_I \text{Mod}_{-\theta} f \big) \cdot \mathbf{1}_I(x)
\end{align*}
where $\mathbb{E}_I g := \frac{1}{|I|} \sum_{n \in I} g(n),$
and seek to study the related maximal operator
\begin{align*}
    \sup_{I} |B_I(f,g)|; 
\end{align*}
the goal will be to exhibit a collection of scales and locations
\[  \mathcal{B} := \{ J \} \]
so that we can show that
\begin{align}\label{e:BB}
\sup_{I \in \mathcal{B}} |B_I(f,g)| = o_{\delta \to 0}(1)
\end{align}
on all but an $o_{\delta \to 0}(1)$ percentage of $I_0$. 

The first issue which immediately presents is that of aggregating the sets
\[ \{ \Lambda_\delta(I) : I \} \]
efficiently. To do so, let us say that two collections of functions 
\[ \{ A_I,B_I : I \} \]
are \emph{close}, if
\[ |\bigcup_{I \subset I_0} \{ I : |A_I-B_I| \geq o_{\delta \to 0}(1)\}| = o_{\delta \to 0}(|I_0|). \]
In order to proceed, we must assume that there exists a single collection $\Lambda = \Lambda_{\delta}(\mathcal{B})$, so that
\begin{align}\label{e:AI} A_I(f,g)(x) := \sum_{\theta \in \Lambda} 
e(2\theta x) \cdot \big( \mathbb{E}_I \text{Mod}_{-\theta} g \big) \cdot \big( \mathbb{E}_I \text{Mod}_{-\theta} f \big) \cdot \mathbf{1}_I(x)
\end{align}
are close to $B_I$ for each $I \in \mathcal{B}$, while maintaining the normalizing bound
\[ |\mathbb{E}_I \text{Mod}_{-\theta} g| \approx \delta; \]
we of course need to maintain the separation condition: 
\[ \min_{\theta \neq \theta' \in \Lambda} |\theta - \theta'| \geq L^{-1} \gg \max_{J \in \mathcal{B}} |J|^{-1}, \; \; \; L \in 2^{\mathbb{N}}\]
to ensure that our vector-valued perspective is valid, as
\[ \sup_{|g| \leq 1, \ |I| \geq L}  (\sum_{\theta \in \Lambda} |\mathbb{E}_I \text{Mod}_{-\theta} g|^2 )^{1/2}  \lesssim 1 \]
by Lemma \ref{l:rsum} below. Finally, we hope that $\Lambda$ is not too much larger than any individual $|\Lambda_\delta(I)| \lesssim \delta^{-2}$, see \eqref{e:specsize}.

With these moves in mind, we work locally on (dyadic) intervals of length $L$ and seek to control
\begin{align}
    \sup_{I \in \mathcal{B}} |A_I(f,g)(x)| \cdot \mathbf{1}_P, \; \; \; |P| = L
\end{align}
Motivated as above, we apply the entropic method: we view 
\[ X(\mathcal{B}) := \{ \{ (\mathbb{E}_I \text{Mod}_{-\theta} g ) \cdot (\mathbb{E}_I \text{Mod}_{-\theta} f ) : \theta \in \Lambda \} : P \subset I \in \mathcal{B} \} \subset \ell^1({\Lambda}) \]
as a subset of $\ell^1(\Lambda)$, 
and let
\[ \mathcal{B}_{\epsilon} := \{ I_i : 1 \leq i \leq D_\epsilon(P) \} \]
be such that
\[ X(\mathcal{B}_\epsilon) \subset \ell^1(\Lambda)\]
is minimal subject to the constraint that
\[ X(\mathcal{B}) \subset X(\mathcal{B}_\epsilon) + B_{\ell^1(\Lambda)}(\epsilon)\]
where 
\[ B_{\ell^1(\Lambda)}(\epsilon) := \{ \{ v(\theta) : \theta \in \Lambda \} : \| v \|_{\ell^1(\Lambda)} < \epsilon \} \]
is with respect to the $\ell^1(\Lambda)$-norm, see \eqref{e:normball}; in other words
\[ D_\epsilon(P) = \text{Int}(X(\mathcal{B});\epsilon)\]
with respect to the $\ell^1(\Lambda)$-norm, see \eqref{e:c-net}.

We bound 
\[ D_\epsilon(P) \leq D_{\epsilon/2}'(P) := \text{Ext}( X(\mathcal{B});\epsilon/2), \]
see \eqref{e:net}, 
and bound
\begin{align}\label{e:ent00}
    \sup_{P \subset I \in \mathcal{B}} |A_I(f,g)| \leq O(\epsilon) + (\sum_{\mathcal{B}_\epsilon} |A_I(f,g)|^2)^{1/2} 
\end{align}
so that
\begin{align}\label{e:ent01}
    \| \sup_{P \subset I \in \mathcal{B}} |A_I(f,g)| \|_{\ell^2(P)}^2 &\lesssim \epsilon^2  \cdot |P| + D_\epsilon(P) \cdot \max_{I \in X(\mathcal{B})} \| A_I(f,g) \|_{\ell^2(P)}^2 \notag \\
    & \qquad \lesssim \epsilon^2 \cdot |P| + D_\epsilon(P) \cdot \delta^2 \cdot |P| \notag \\
    & \qquad \qquad \leq\epsilon^2 \cdot  |P| + D_{\epsilon/2}'(P) \cdot \delta^2 \cdot  |P| \end{align}
as 
\[ \sum_{\theta \in \Lambda} |\mathbb{E}_I g(n) e(-n\theta)|^2 \cdot |\mathbb{E}_I \text{Mod}_{-\theta} f(n) |^2 \lesssim \delta^2 \cdot \sum_{\theta \in \Lambda} |\mathbb{E}_I \text{Mod}_{-\theta} f(n) |^2 \lesssim \delta^2 \]
by Lemma \ref{l:rsum} below.

But, if we let 
\begin{align*}
\mathcal{A}_\epsilon(P) &\subset \big\{ \{ \mathbb{E}_I \text{Mod}_{-\theta} f(n) : \theta \in \Lambda\} : P \subset I \in \mathcal{B} \big\}, \; \; \; \text{ and } \\
\mathcal{C}_\epsilon(P) &\subset \big\{ \{ \mathbb{E}_I \text{Mod}_{-\theta} g(n)  : \theta \in \Lambda\} : P \subset I \in \mathcal{B} \big\} \end{align*}
denote minimal subsets so that
\[ \big\{ \{ \mathbb{E}_I \text{Mod}_{-\theta}f(n)  : \theta \in \Lambda\} : P \subset I \in \mathcal{B} \big\} \subset \mathcal{A}_\epsilon(P) + B_{\ell^2(\Lambda)}(\epsilon) \]
and 
\[ \big\{ \{ \mathbb{E}_I \text{Mod}_{-\theta}g(n) : \theta \in \Lambda\} : P \subset I \in \mathcal{B} \big\} \subset \mathcal{C}_\epsilon(P) + B_{\ell^2(\Lambda)}(\epsilon), \]
namely the the sets which realize $\text{Int}(\cdot;\epsilon)$ with respect to the $\ell^2(\Lambda)$-balls, see \eqref{e:normball},
then
\begin{align}\label{e:cover}
    D_{\epsilon/2}'(P) &\leq |\mathcal{A}_{c_0 \epsilon}(P)| \cdot |\mathcal{C}_{c_0 \epsilon}(P)| \notag \\
    & \qquad \leq N_{c_1 \epsilon}( \{ \mathbb{E}_I \text{Mod}_{-\theta} f(n)  : \theta \in \Lambda\} : I \supset P) \cdot N_{c_1 \epsilon}( \{ \mathbb{E}_I \text{Mod}_{-\theta} g(n)  : \theta \in \Lambda\} : I \supset P) \\
    & \qquad \qquad =: N_\epsilon^1(P) \cdot N_\epsilon^2(P) \notag
\end{align}
for some absolute $0 <c_1 := c_0/10 < c_0 < 1$; the first inequality follow from Cauchy-Schwartz estimates, using the bounds
\[ \sum_{\theta \in \Lambda} |\mathbb{E}_I g(n) e(-n\theta)|^2, \ \sum_{\theta \in \Lambda} |\mathbb{E}_I f(n) e(-n \theta)|^2 \lesssim 1, \]
by Lemma \ref{l:rsum} below.

In particular
\begin{align}\label{e:key}
\| \sup_{P \subset I \in \mathcal{B}} |A_I(f,g)| \|_{\ell^2(P)}^2 \lesssim \epsilon^2  \cdot |P| + \delta^2 \cdot  N_{\epsilon}^1(P) \cdot N_{\epsilon}^2(P) \cdot |P|.
\end{align}
By Corollary \ref{c:sep}, we (essentially) have
\[ \sum_P N_{\epsilon}^i(P) \cdot  |P| \lesssim \epsilon^{-2} \cdot |I_0|, \; \; \; i=1,2
\]
so that, morally speaking
\begin{align}
\| \sup_{x \in I \in \mathcal{B}} |A_I(f,g)| \|_{\ell^2(I_0)}^2 = \sum_{P \subset I_0} \| \sup_{P \subset I \in \mathcal{B}} |A_I(f,g)| \|_{\ell^2(P)}^2 \lesssim (\epsilon^2 + \delta^2  \cdot \epsilon^{-4}) \cdot |I_0| = O(\delta^{2/3} \cdot |I_0|),
\end{align}
after optimizing $\epsilon = \delta^{1/3}$.

\subsection{Conclusion}
After implementing the steps described in the introduction above, we apply tools from harmonic analysis to apply the analysis introduced here to establish Theorem \ref{t:DR}, with the inequality \eqref{e:key} proving fundamental; Theorem \ref{t:RT} follows using similar techniques.

Accordingly, the major analytical effort of this paper will be in selecting appropriate collections of intervals (``branches") $\mathcal{B}$, and successfully approximating \[ B_I \longrightarrow A_I \]
by an appropriate ``close" sequence of operators.

This approximation is quite involved, and begins by efficiently grouping $\{ J \}$ into sub-collections known as \emph{trees}, $\{ \mathcal{D}_j \}$, which resemble branches in that one may find a single $\Lambda \subset \mathbb{T}$ which will approximates
\[ \{ \Lambda_{\delta}(I) : I \in \mathcal{D}_j \} \]
but possibly does \emph{not} enjoy the non-concentration condition
\[ \min_{\theta \neq \theta' \in \Lambda} |\theta - \theta'| \gg \max_{J \in \mathcal{B}} |J|^{-1};\]
up to some exceptional set considerations, the trees are subsequently decomposed into the appropriate branches. This approximation of the $\{ \Lambda_\delta(I) \}$ -- and in particular, showing that only an appropriately small number of trees (a ``forest") presents -- follows from orthogonality considerations by way of an energy pigeon-holing argument.

With these ideas in mind, we turn to the details.

\section{Quantifying Convergence}\label{s:QC}
The goal of this section is to reduce Theorems \ref{t:DR} and \ref{t:RT} to certain quantitative statements; in the following section, we will reduce these statements to certain Fourier-analytic estimates involving appropriate families of Fourier multipliers on the torus, $\mathbb{T}$.

We begin with some notation.

For $x \in X$ and $\omega \in Y$, define the averages
\[ \Phi_M(f,g)(x) := \Phi^T_M(f,g)(x) := \frac{1}{M} \sum_{m \in [M]} T^{\lfloor \alpha m \rfloor} f(x) \cdot T^{-m} g(x)\]
and
\[ \Phi_M(f,g)(x;\omega) := \Phi^{T,S}_M(f,g)(x;\omega) := \frac{1}{M} \sum_{m \in [M]} T^{\lfloor \alpha m \rfloor} f(x) \cdot S^{m} g(\omega) \]
where we are restricting to $f,g$ to be $1$-bounded. Below, each function introduce will be implicitly assumed to be $1$ bounded on the relevant domain.

For each $\tau \ll 1$ sufficiently small, choose
\begin{align}\label{e:Btau}
\mathbb{N} \ni B_\tau \approx A_0 \cdot \tau^{-1}
\end{align}
and two much larger parameters, $B \ll H$, satisfying (say)
\[ B_\tau^{10} \ll B \ll B^{10} \ll H.\]

\begin{definition}[Quantifying Oscillation]
We introduce the following quantities, designed to precisely quantify the oscillation of the averages $\{ \Phi_M, \ \Phi_M(\cdot;\omega) \}$; throughout this definition we restrict all times $M_0,M_1,M,\dots$ to be of the form
\[ \{ 2^{n/B_\tau} \}.\]
\begin{align}
C_{\tau,B,H}^{(X,\mu,T)} &:= \sup \Big\{ K : \text{ there exists $1$-bounded } f,g , \ M_0 < M_1 < \dots < M_K \leq H/100, \\
& \text{ so that } \mu(\{ x : \max_{M_{k-1} \leq M \leq M_k} |\Phi_M(f,g)(x) - \Phi_{M_k}(f,g)(x)| \geq \tau \}) \geq \tau^B \Big\} \notag
\end{align} 
and, regarding $g : Y \to \{|z| \leq 1\}$ (and $S:Y \to Y$) as fixed, define
\begin{align}
&    C_{\tau,B,H}^{(X,\mu,T)}(\omega) := \sup \Big\{ K : \text{there exists $1$-bounded } f , \ M_0(\omega) < M_1(\omega) < \dots < M_K(\omega) \leq H/100, \\
&  \qquad \text{ so that } \mu(\{ x : \max_{M_{k-1}(\omega) \leq M \leq M_k(\omega)} |\Phi_M(f,g)(x;\omega) - \Phi_{M_k}(f,g)(x;\omega)| \geq \tau \}) \geq \tau^B \Big\}. \notag
\end{align} 
\end{definition}

The main effort of this paper will go into proving the following estimates.

\begin{proposition}[Quantifying Convergence]\label{p:QC}
    For each $\tau,B,H$, 
    \[ C_{\tau,B,H}^{(X,\mu,T)} \lesssim_{\tau,B} 1 \]
uniformly in $(X,\mu,T)$ and $H$. And,
\[ \nu(\{ \omega \in Y : \sup_{(X,\mu,T)} C_{\tau,B,H}^{(X,\mu,T)}(\omega) = +\infty \}) = 0 \]
uniformly in $H$.
\end{proposition}

Before proceeding with the argument, let us see how Proposition \ref{p:QC} implies Theorems \ref{t:DR} and \ref{t:RT}.

\begin{proof}[Reduction to Proposition \ref{p:QC}]
 Note that the above hypotheses extend to
    \[ C_{\tau,B}^{(X,\mu,T)} := C_{\tau,B,\infty}^{(X,\mu,T)}, \; \; \; C_{\tau,B}^{(X,\mu,T)}(\omega) := C_{\tau,B,\infty}^{(X,\mu,T)}(\omega) \]
by a monotone convergence argument, so we will concern ourselves with the infinitary statements. 

Both proofs are by contradiction; we begin with the case of $C_{\tau,B}^{(X,\mu,T)}$.

So, suppose that for some measure-preserving system, $(X,\mu,T)$, and some functions $f, g: X \to \{|z| \leq 1\},$
\[ \mu(\{ x : \limsup_{M,N} |\Phi_M(f,g)(x) - \Phi_N(f,g)(x)| \gg \tau\}) \gg \tau \]
for some $1 \gg \tau > 0$. In this case, by a brief measure-theoretic argument, we could extract a finite subsequence of arbitrary length $K$ so that for each $1 \leq k \leq K$
\[ \mu(\{ x : \max_{M_{k-1} \leq M \leq M_k} |\Phi_M(f,g)(x) - \Phi_{M_k}(f,g)(x)| \gg \tau\}) \gg \tau;\]
there is no loss of generality here in assuming that 
\[ M_k,M \in \{ 2^{n/B_\tau} \}, \]
as whenever
\[ M = 2^{n/B_\tau} \leq M' < 2^{(n+1)/B_\tau} \]
are close,
\[ \Phi_M(f,g)(x) = \Phi_{M'}(f,g)(x) + O(B_\tau^{-1}) = \Phi_{M'}(f,g)(x) + o_{A_0 \to 0}(\tau), \]
see \eqref{e:Btau}; we will implicitly restrict all times to this lacunary sequence for the rest of the proof.

Summing over $k \leq K$, we bound
\[ K \tau \ll \sum_{k \leq K} \mu(\{ x : \max_{M_{k-1} \leq M \leq M_k} |\Phi_M(f,g)(x) - \Phi_{M_k}(f,g)(x)| \gg \tau\}) \leq C_{\tau,B}^{(X,\mu,T)};\]
this handles Theorem \ref{t:DR}.

The second statement is similar; we have already shown that the statement that
\[ C_{\tau,B}^{(X,\mu,T)}(\omega) < \infty \]
for each $B$ implies pointwise convergence of $\{ \Phi_M(f,g)(\cdot \, ;\omega) \}$, so the statement that
\[ \sup_{(X,\mu,T)} C_{\tau,B}^X(\omega) < \infty \]
implies pointwise convergence of $\{ \Phi_M(f,g)(\cdot \, ;\omega)  \}$ uniformly in all measure-preserving systems $(X,\mu,T)$, which happens $\nu$-almost surely.
\end{proof}

So far, we have not addressed the $\nu$-measurability of the sets
\begin{align}\label{e:bad} \{ \omega \in Y : \sup_{(X,\mu,T)} C_{\tau,B,H}^{(X,\mu,T)}(\omega) = + \infty \} = \bigcap_{L \geq 1} \{ \omega \in Y : \sup_{(X,\mu,T)} C_{\tau,B,H}^{(X,\mu,T)}(\omega) \geq L \} 
\end{align}
this is slightly delicate, so we defer the proof to our below Appendix, \S \ref{s:appendix}, along with a reduction to the case where $S$ is \emph{ergodic}; we will proceed under this assumption of ergodicity for the rest of the argument.

\section{Transference to The Integers}\label{s:transtoZ}
The goal of this section is to reduce Proposition \ref{p:QC} to a pair of discrete-harmonic-analytic estimates; we begin with notation.

For $x, a \in \mathbb{Z}$, define
\[ {A}_M(f,g)(x;a) := \frac{1}{M} \sum_{m \in [M]} f(x-{\lfloor \alpha m \rfloor}) \cdot g(a+m)\]
and set
\[ {A}_M(f,g)(x) := {A}_M(f,g)(x;x).\]

With $\tau,B,H$ as above, maintaining the constraint that all times $M_0,M_1,M,\dots$ are of the from $\{ 2^{n/B_\tau} \}$ with $B_\tau$ as above, see \eqref{e:Btau}, define

\begin{align}
    {C}_{\tau,B,H} &:= \sup\Big\{ K : \text{ there exist $1$-bounded } f,g , M_0 < M_1 < \dots < M_K \leq H/100,  \\
&  \text{ so that } |\{ x \in [H] : \max_{M_{k-1} \leq M \leq M_k} |{A}_M(f,g)(x) - {A}_{M_k}(f,g)(x)| \geq \tau \}| \geq \tau^B H \Big\} \notag
\end{align} 
and for a $1$-bounded function $g$
\begin{align}
    {C}_{\tau,B,H}(a) &:= \sup\Big\{ K : \text{ there exist a $1$-bounded } f, \ M_0 < M_1 < \dots < M_K \leq H/100  \\
&  \text{ so that } |\{ x \in [H] : \max_{M_{k-1} \leq M \leq M_k} |{A}_M(f,g)(x;a) - {A}_{M_k}(f,g)(x;a)| \geq \tau \}| \geq \tau^B H \Big\}. \notag
\end{align} 

It will be convenient to smooth out our averages: choose a smooth
\begin{align}\label{e:phitau} \phi = \phi_\tau \end{align}
so that
\[ \mathbf{1}_{[\tau^{A_0},1-\tau^{A_0}]} \leq \phi \leq \mathbf{1}_{[0,1]} \]
so that
\[ |\partial^{j} \phi| \lesssim \tau^{- A_0 j} \]
for sufficiently many $j$; set
\[ \phi_M(m) := 1/M \phi(m/M),\]
and
for $x, a \in \mathbb{Z}$, define
\[ \tilde{A}_M(f,g)(x;a) := \sum_{m} \phi_M(m) f(x-{\lfloor \alpha m \rfloor}) \cdot g(a+m)\]
and set
\[ \tilde{A}_M(f,g)(x) := \tilde{A}_M(f,g)(x;x).\]

Maintaining the constraint that all times $M_0,M_1,M,\dots$ are of the from $\{ 2^{n/B_\tau} \}$ with $B_\tau$ as above, see \eqref{e:Btau}, we introduce 

\begin{align}\label{e:tilde}
 \tilde{C}_{\tau,B,H} &:= \sup\Big\{ K : \text{ there exist $1$-bounded } f,g , \ M_0 < M_1 < \dots < M_K \\ 
&  \leq H/100,  \text{ so that } |\{ x \in [H] : \max_{M_{k-1} \leq M \leq M_k} |\tilde{A}_M(f,g)(x) - \tilde{A}_{M_k}(f,g)(x)| \geq \tau \}| \geq \tau^B H \Big\} \notag
\end{align} 
and for a $1$-bounded $g$
\begin{align}\label{e:tilde0}
    \tilde{C}_{\tau,B,H}(a) &:= \sup\Big\{ K : \text{ there exist a $1$-bounded } f, \ M_0 < M_1 < \dots < M_K \leq H/100  \\
&  \text{ so that } |\{ x \in [H] : \max_{M_{k-1} \leq M \leq M_k} |\tilde{A}_M(f,g)(x;a) - \tilde{A}_{M_k}(f,g)(x;a)| \geq \tau \}| \geq \tau^B H \Big\}. \notag
\end{align} 
Note that
\[ {C}_{\tau,B,H} \leq \tilde{C}_{\tau/2,B,H} \leq {C}_{\tau/4 ,B,H} \]
and similarly
\[ {C}_{\tau,B,H}(a) \leq \tilde{C}_{\tau/2,B,H}(a) \leq {C}_{\tau/4,B,H}(a);\]
in what follows, we will therefore make use of the ability to pass back between smooth and rough cut-offs as we see fit, and will use the notation $A_M$ to refer to either (though typically we will use the smooth cut-off).

\begin{remark}
    In what follows, we will often find it convenient to sparsify our set of times into subsets of the form $\{ 2^{(n+kA)/B_\tau} : k \} $ where $0 \leq n <A $ for $A \gg_\tau 1$ very large; this move only effects the final bound on $C_{\tau,B,H}, \ C_{\tau,B,H}(a)$ by at most a  multiplicative factor depending on $\tau$, to which the argument is indifferent.
\end{remark}

The relationship between the quantities introduced above and those in the previous section are captured by the following Proposition.

\begin{proposition}[Transference to $\mathbb{Z}$, I]\label{p:trans}
For each $\tau,B,H$, 
\[ C_{\tau,B,H}^X \lesssim \tau^{-B} \cdot C_{\tau,B+1,H};\]
and, if for each $N \gg 1$ large,
    \[ |\{ a \in [N] : C_{\tau,B,H}(a) \geq L \}|  = o_{L \to \infty;\tau,B}(N), \]
    independent of $H$, then the second statement in Proposition \ref{p:QC} holds.
\end{proposition}
\begin{proof}
    We begin with the first statement. By assumption, for any $|f|,|g| \leq 1$, and any $M_0<M_1< \dots < M_K$ if we set
    \begin{align}
        Z_k &:= Z_k(f,g,M_0,\dots,M_K,\tau,B) \\
        & \qquad := \{ r \in [H] : \max_{M_{k-1} \leq M \leq M_k} | {A}_M(f,g)(r) - {A}_{M_k}(f,g)(r)| \geq \tau\} \subset [H], \notag
    \end{align} 
    where all times $M \in \{  2^{n/B_\tau}  \}$,
    then
    \begin{align}\label{e:Zptwise}
        \sum_{k \leq K} \frac{1}{H} \sum_{r \in [H]} \mathbf{1}_{Z_k}(r) &\leq \sum_{k \leq K \text{ good}} \frac{1}{H} \sum_{r \in [H]} \mathbf{1}_{Z_k}(r) + \sum_{k \leq K \text{ bad}} \frac{1}{H} \sum_{r \in [H]} \mathbf{1}_{Z_k}(r) \\
        & \qquad \leq \tau^{B+1} K + C_{\tau,B+1,H}; \notag
    \end{align}
above an index is \emph{bad} if 
\begin{align}\label{e:badcrit} |Z_k| \geq \tau^{B+1} H \end{align}
and \emph{good} otherwise.
    
Let $C_{\tau,B,H}^X \lesssim_H 1$ be as above; our job is to bound \[ C_{\tau,B,H}^X \lesssim \tau^{-B} \cdot C_{\tau,B+1,H}\]
independent of $H$. If we set, for an appropriate  $f,g : X \to \{0,1\}$,
\begin{align}
    U_k &:= U_k(f,g,M_0,\dots,M_K, \tau,B) \\
    & \qquad := \{ x : \max_{M_{k-1} \leq M \leq M_k} |\Phi_M(f,g)(x) - \Phi_{M_k}(f,g)(x)| \geq \tau \}, \; \; \; \mu(U_k) \geq \tau^B \notag
\end{align}  
for an appropriate realization of $C_{\tau,B,H}^X = K \lesssim_H 1$, then using the measure-preserving nature of $T$, we may bound
    \begin{align}\label{e:restI}
        \tau^B \cdot C_{\tau,B,H}^X = \tau^B K \leq \int_X \big( \sum_{k \leq K} \frac{1}{|I|} \sum_{r \in I} \mathbf{1}_{U_k}(T^r x) \big) \ d\mu(x), \; \; \; I := [H/100,H - H/100].
    \end{align}
We claim that $\mu$-a.e., we may dominate the integrand by a constant multiple of
\[ 
\tau^{B+1} K + C_{\tau,B+1,H}
\]
which leads to the bound
\[ C_{\tau,B,H}^X \lesssim \tau^{-B} \cdot C_{\tau,B+1,H}.\]
To see this, let $x \in X$ be arbitrary, and define
\[ F(n) := T^n f(x) \cdot \mathbf{1}_{n \in [H]} \]
and
\[ G(n) := T^n g(x) \cdot \mathbf{1}_{n \in [H]}.\]
The key observation is that for all $r \in I$ and $M \leq M_K \leq H/100$,
\[ \Phi_M(f,g)(T^r x) = \frac{1}{M} \sum_{m \in [M]} T^{r+m} f( x) \cdot T^{r-m} g( x) \]
is precisely given by
\[ A_M(f,g)(r) = \frac{1}{M} \sum_{m \in [M]} F(r-m) \cdot G(r+m),\]
so we may pointwise bound
\[ \sum_{k\leq K} \frac{1}{|I|} \sum_{r \in I} \mathbf{1}_{U_k}(T^r x) \lesssim \sum_{k \leq K} \frac{1}{H} \sum_{r \in [H]} \mathbf{1}_{Z_k}(r) \leq \tau^{B+1} K + C_{\tau,B+1,H}, \]
for an appropriate choice of $\{ Z_k : k \leq K\}$. This establishes the first statement.

For the second, it suffices to restrict attention to a fixed ergodic measurable system, $(X_0,\mu_0,T_0)$ by Lemma \ref{l:II}, and to assume that $(Y,\nu,S)$ is ergodic by Lemma \ref{l:serg}; set
\[ V_L := V_{L,\tau,B,H} := \{ \omega \in Y : C_{\tau,B,H}^{(X_0,\mu_0,T_0)}(\omega) \geq L \}) \subset Y \]
and apply the Pointwise Ergodic Theorem. In particular, for any $\epsilon > 0$, for $\nu$-a.e.\ $\omega \in Y$, there exists $N(\omega,\epsilon) \gg H$ so large that for all $N \geq N(\omega,\epsilon)$,
\[ \nu(V_L) = \frac{1}{N} \sum_{ a\in [N]} \mathbf{1}_{V_L}(S^a \omega) + o(\epsilon); \]
fix such an $\omega$, and note that the statement that $S^a \omega \in V_L$ is equivalent to the statement that
\[ C_{\tau,B,H}^{(X_0,\mu_0,T_0)}(S^a \omega) \geq L.\]
So, with $N \geq N(\omega,\epsilon) \gg H$, set 
\[ G(n) := g(S^n \omega) \cdot \mathbf{1}_{n \in [N]}\]
so that the condition
\[  C_{\tau,B,H}^{(X_0,\mu_0,T_0)}(S^a \omega) \geq L \]
means that there exists an indicator $f_a : X_0 \to \{0,1\}$, and a sequence
\[ M_0(a)<M_1(a) < \dots < M_{L}(a) \leq H/100,\]
so that for each $k \leq L$, if we set
\begin{align*}
    U_k(a) &:= U_k(a;\tau,B) \\
    &  \qquad := \Big\{ x \in X_0 : \max_{M_{k-1}(a) \leq M \leq M_k(a)} |\frac{1}{|M|} \sum_{m \in [M]} G(a+m) T_0^n f_a(x) \notag \\
    & \qquad \qquad \qquad \qquad \qquad \qquad  - \frac{1}{M_k(a)} \sum_{m \in [M_k(a)]} G(a+m) T_0^n f_a(x)| \geq \tau \Big\} \notag 
\end{align*} 
then $\mu_0(U_k(a)) \geq \tau^B$ for each $k \leq L$. But now, for each such $a$, set $I= [H/100,H-H/100]$, see \eqref{e:restI}, and bound
\begin{align*}
    \tau^B L &\leq \sum_{k \leq L} \int_{X_0} \frac{1}{|I|} \sum_{r \in I} \mathbf{1}_{U_k(a)}(T_0^r x) \ d\mu_0(x) \\
    & \qquad = \int_{X_0} \big( \sum_{k \leq L
} \frac{1}{|I|} \sum_{r \in I} \mathbf{1}_{U_k(a)}(T_0^r x) \big) \ d\mu_0(x).
\end{align*} 
We once again bound the integrand pointwise, fixing $x \in X_0$, and defining 
\[ F_a(r) := T_0^r f_a(x) \cdot \mathbf{1}_{r \in [H]}.\]
Then, for
\begin{align*}
        Z_k^a &:= Z_k^a(\tau,B) \\
        & \qquad := \{ r \in [H] : \max_{M_{k-1}(a) \leq M \leq M_k(a)} | {A}_M(F_a,G)(r;a) - {A}_{M_k(a)}(F_a,G)(r;a)| \geq \tau\} \subset [H],
    \end{align*}
    using the good/bad dichotomy of \eqref{e:badcrit},
we may bound
\begin{align*}
    \sum_{k \leq L} \frac{1}{|I|} \sum_{r \in I} \mathbf{1}_{U_k(a)}(T_0^r x) \lesssim \sum_{k \leq L} \frac{1}{|I|} \sum_{r \in I} \mathbf{1}_{Z_k^a}(r) \lesssim \tau^{B+1} L + C_{\tau,B+1,H}(a)
\end{align*} 
to obtain the string of inequalities
\[ \tau^B L \lesssim C_{\tau,B+1,H}(a). \]
But, the set of all $a \in [N]$ so that
\[ \tau^B L \lesssim C_{\tau,B+1,H}(a)\]
has size $o_{L \to \infty;\tau,B}(N)$, so that
\[ \nu(V_L) = \frac{1}{N} \sum_{ a\in [N]} \mathbf{1}_{V_L}(S^a \omega) +o(\epsilon) = o_{L \to \infty;\tau,B}(1) + o(\epsilon) \]
for $\nu$-a.e. $\omega \in Y$.
\end{proof}

\section{Tools from Harmonic Analysis}\label{ss:tools}
The goal of this section is to collect some tools from harmonic analysis which we will appeal to in the main argument below.
\subsection{Fourier Analysis on the Torus}
  \begin{lemma}[Fourier Statistics Control Pointwise Data: $H^{1/2}(\mathbb{T})$-Sobolev Embedding]\label{l:H12}
Suppose $|f| \leq 1$, and $m : [0,1] \to \mathbb{C}$ is a $\mathcal{C}^1(\mathbb{T})$-multiplier. Then for any $k \in \mathbb{Z}$, we may bound
\[ \aligned |(\hat{f} m)^{\vee}| \lesssim  \| m \|_{L^2(\mathbb{T})} + \| m \|_{L^2(\mathbb{T})}^{1/2} \cdot \| \partial_\beta \big( e(k \cdot) m \big) \|_{L^2(\mathbb{T})}^{1/2}  
\endaligned \]
and for intervals $J \subset \mathbb{R}$,
\begin{align}\label{e:notJ} \sum_{n \notin 3J} |m^{\vee}(n)| \lesssim |J|^{-1/2} \cdot \| \partial_\beta \big( e(x_J \beta) m(\beta) \big) \|_{L^2(\mathbb{T})}, \; \; \; x_J \in J. 
\end{align}
\end{lemma}

Our next inequality will allow us to control the square-sum of low-degree polynomials over widely-separated frequencies by its $L^2(\mathbb{T})$-norm.

\begin{lemma}[Sampling Low-Degree Polynomials]\label{l:rsum}
    Suppose that $P(\beta)$ is a polynomial of degree $N$. Then
    \[ \sum_{\theta \in \Lambda} |P(\theta)|^2  \lesssim N \cdot \| P \|_{L^2(\mathbb{T})}^2 \]
        whenever $\Lambda \subset \mathbb{T}$ is a collection of frequencies satisfying the separation condition
    \[ \min_{\theta \neq \theta' \in \Lambda} |\theta - \theta'| \geq N^{-1}. \]
\end{lemma}

\subsection{Orthogonality in Phase Space: Time Frequency Analysis}
We state an orthogonality result of the Bessel type; we begin by introducing the language of \emph{tiles}. Throughout the remainder of this section, we introduce a non-negative, sufficiently smooth $\phi \in \mathcal{C}_c([0,1])$ satisfying
\[ 0 \leq \phi \lesssim \mathbf{1}_{[0,1]}. \]
Suppose we are given a family of intervals
\[ \mathcal{D} := \{ I \} \subset \mathbb{R} \]
so that $|I| \in \lambda 2^{\mathbb{N}}$ for some $1 \leq \lambda = 2^{n/B} < 2, B = O_t(1)$, satisfying the grid property, 
\[ I \cap J \in \{ \emptyset, I,J\}\]
for $I,J \in \mathcal{D}$; grids will be further explored below. For each subset of $\mathbb{T}$, let $\{ \omega \} \subset \mathcal{D}'$ denote dyadic sub-intervals of $\mathbb{T}$, a ``dual grid," so that $|\omega| |I| = \lambda$ for some $I \in \mathcal{D}$. For a collection of $\xi_\omega \in \omega$, define
\[ \psi_s(n) := \frac{1}{|I|^{1/2}} \phi(\frac{n-c_I}{|I|}) e( n \xi_\omega), \; \; \; s := I \times \omega\]
where $c_I$ is the left end-point of $I$.

There is a natural partial ordering on the set of ``tiles" 
\[ \{ s = I \times \omega \in \mathcal{D} \times \mathcal{D}'\} \]
where we say that
\[ s = I \times \omega \prec s' = I' \times \omega' \iff I \subset I', \; \omega' \subset \omega.\]
We will refer to $I$ as the ``time" interval, and $\omega$ as the ``frequency" interval associated to $s = I \times \omega$.

The collection $\{ \psi_s \}$ satisfy a weak Bessel inequality, which we will appeal to below.

\begin{proposition}\label{p:tfa}
    Suppose that $\mathbf{S} := \{ s = I_s \times \omega_s\}$ are a set of pairwise incomparable tiles, whose times intervals are contained in a parental interval $I_{\mathbf{S}}$, with $|\omega_s| \ll 1$ sufficiently small. Suppose that $|g| \leq 1$, and that for some $\delta > 0$,
\[ |\langle \psi_s, g \rangle_{\mathbb{Z}}| \approx \delta |I_s|^{1/2}\]
Then, for any $\epsilon > 0$,
\[ \sum_{s \in \mathbf{S}} |\langle \psi_s,g\rangle|^2 \lesssim_\epsilon \delta^{-\epsilon} \| g \mathbf{1}_{I_{\mathbf{S}}} \|_{\ell^2}^2 \]
and in particular
\[ \sum_{s \in \mathbf{S}, \ I_s \subset I_0} |I_s| \lesssim_\epsilon \delta^{-2-\epsilon} |I_0|. \]
\end{proposition}

The proof of Proposition \ref{p:tfa} is identical to \cite[Lemma 5.1]{THIELE}, after decomposing the function
\[ \phi = \phi_\delta + \mathcal{E}_\delta \]
where 
\[ \delta^{o(1)} \sum_{j \leq 10} |\partial^j \phi_\delta(x)| + \delta^{-10} \sum_{j \leq 10} |\partial^j \mathcal{E}_j(x)| \lesssim (1 + |x|)^{-10} \] and $\phi_\delta$ has small Fourier support
\[ \text{supp }\widehat{\phi_{\delta}} \subset [-\delta^{-o(1)},\delta^{-o(1)}];\]
$\mathcal{E}_\delta$ is an error term.

\subsection{Tools: Jump-Counting Estimates}\label{s:jump}
In this section, we recall necessary jump-counting estimates in this section for (finite dimensional) Hilbert-space-valued functions: throughout, $\mathcal{H}$ will denote a finite-dimensional Hilbert space.

For a sequence of functions 
\[ \{ \vec{f_k} \} \subset \ell^2(\mathcal{H}) \]
define
\[ 
\vec{N}_\epsilon( \vec{f_k}(x) : k )
\]
to be the jump-counting function at altitude $\epsilon > 0$ with respect to $\mathcal{H}$: 

the supremum over all $K$ so that there exists $k_0 < k_1 < \dots < k_K$ so that
\[ \| \vec{f_{k_i}}(x) - \vec{f_{k_{i-1}}}(x) \|_{\mathcal{H}} > \epsilon \]
for all $1 \leq i \leq K$; see Subsection \ref{ss:ent} for a discussion of the relationship between $\vec{N}_\epsilon$ and other measurements of $\epsilon$-entropy.

We will make frequent use of the dyadic (reverse) martingale for Hilbert-valued functions:

We set
    \[ \{ \mathbb{E}_k \vec{f} := \sum_{|I| = 2^k \text{ dyadic}} \mathbb{E}_I \vec{f} \cdot \mathbf{1}_I : k \geq 0 \} \]
    where
\begin{align}\label{e:EI}
    \mathbb{E}_I \vec{f} := \frac{1}{|I|} \sum_{n \in I} \vec{f}(n)
    \end{align}
denotes the average value of $\vec{f} \in \ell^2(\mathcal{H})$ over $I$; we will for the most part be interested in the case where $I$ is a dyadic interval.

We recall classical \emph{jump-counting} estimates for the dyadic martingale; for a similar proof in the scalar-valued setting, see \cite[Theorem 3.10]{BOOK}.

\begin{lemma}\label{l:martjump}
 The following estimate holds, independent of $\mathcal{H}$:
Then
\begin{align*}
    \sup_{\epsilon > 0}\| \epsilon \cdot \vec{N}_\epsilon( \mathbb{E}_k \vec{f}(x) : k )^{1/2} \|_{\ell^2(\mathbb{Z})} \lesssim \| \vec{f} \|_{\ell^2(\mathcal{H})}.
\end{align*}    
\end{lemma}

The following relatively straightforward corollary will suffice for most purposes.

\begin{corollary}\label{c:sep}
    Suppose that $\Lambda = \{ \theta \}$ are $\kappa-$separated, 
    \[ \min_{\theta \neq \theta' \in \Lambda} \|\theta - \theta'\|_{\mathbb{T}} \geq \kappa,\]
    and that we are in the regime where $2^k \geq A/\kappa$.

For dyadic intervals, define
\[ A_I^\theta f(x) := \mathbb{E}_I \text{Mod}_{-\theta}f  \cdot \mathbf{1}_I, \]
see \eqref{e:EI}
and let
\begin{align*}
\vec{N}_\epsilon( A_I \vec{f}(x) : I )
\end{align*}
be the jump-counting function at altitude $\epsilon > 0$ with respect to the norm $\ell^2(\Lambda)$. Then
\begin{align*}
\| \vec{N}_\epsilon( A_I \vec{f}) \|_{\ell^2(\mathbb{Z})} \lesssim (1 + \frac{|\Lambda|^{1/2}}{A}) \cdot \| f\|_{\ell^2(\mathbb{Z})}.
\end{align*}
\end{corollary}

\begin{proof}
    If we replace $A_I^\theta f$ with
    \begin{align*}
        B_I^\theta f(x) &:= \mathbb{E}_I \big( \sum \varphi_\kappa(x-u) f(u) e(-\theta u) \big) \\
        & \qquad = \sum_{u \in I} \big( \mathbb{E}_{x \in I} \varphi(x-u) \big) \cdot f(u) e(-\theta u) \big) \\
        & \qquad \qquad + \sum_{k \geq 1} \Big( \sum_{u \in 2^k I \smallsetminus 2^{k-1} I}  \big( \mathbb{E}_{x \in I} \varphi(x-u) \big) \cdot f(u) e(-\theta u) \Big)
        \end{align*} 
     for  $\widehat{ \varphi}$ which smoothly approximates $\mathbf{1}_{[-\kappa/2,\kappa/2]}$, then the result follows from vector-valued jump-counting estimates with respect to the norm $\ell^2(\Lambda)$, namely Lemma \ref{l:jump0}. But with $u \in I$, using that $\widehat{\varphi}(0) = 1$, we have the asymptotic
\begin{align*}
    \mathbb{E}_I \varphi(x-u) = \frac{1}{|I|} (1 + O( (\kappa |I|)^{-5} ))
\end{align*}
(say), so we can express
\begin{align*}
    \sum_{u \in I} \big( \mathbb{E}_{x \in I} \varphi(x-u) \big) f(u) e(-\theta u) \big) &= \frac{1}{|I|} \sum_{x \in I} \text{Mod}_{-\theta} f(x) \cdot (1 + O((\kappa |I|)^{-5})) \\
    & \qquad = A_I^\theta f(x) + O((\kappa |I|)^{-5}) \cdot M_{\text{HL}} f(x) 
\end{align*}
and just bound
\[ \sum_{u \in 2^k I \smallsetminus 2^{k-1} I}  \big| \big(\mathbb{E}_{x \in I} \varphi(x-u) \big) \cdot \text{Mod}_{-\theta} f(u) \big| \lesssim (2^k \kappa |I|)^{-5} \cdot M_{\text{HL}} f(x) \]
so that
\begin{align*}
    \vec{N}_\epsilon(A_I^\theta f) \leq  \big( \sum_k \sum_{\theta \in \Lambda} |A_I^\theta f - B_I^\theta f|^2 \big)^{1/2} + \vec{N}_\epsilon(B_I^\theta f)
\end{align*}
and
\begin{align*}
    \| \vec{N}_\epsilon(B_I^\theta f) \|_{L^2(\mathbb{R})} \lesssim \| f\|_{L^2(\mathbb{R})}
\end{align*}
by Lemma \ref{l:jump0}. But
\begin{align*}
    \| \big( \sum_{2^k \geq A/\kappa} \sum_{\theta \in \Lambda} \sum_{|I| = 2^k} |A_I^\theta f - B_I^\theta f|^2 \big)^{1/2} \|_{\ell^2(\mathbb{Z})}^2 &\leq \sum_{\theta \in \Lambda} \sum_{2^k \geq A/\kappa} (\kappa 2^k)^{-10} \| M_{\text{HL}} f \|_{\ell^2(\mathbb{Z})}^2 \\
    & \qquad \lesssim \frac{|\Lambda|}{A^{10}} \cdot \| f \|_{\ell^2(\mathbb{Z})}^2
\end{align*}
from which the result follows.
\end{proof}

We will need the following convolution-based modification, see \cite{J}.
\begin{lemma}\label{l:jump0}
Let 
\begin{align}\label{e:bumps} \mathcal{O} := \{ \phi \in \mathcal{C}^{1}(\mathbb{R}) : (1 + |t|)^{10} \cdot \sum_{j \leq 10} |\partial^{(j)} \phi(t)|  \leq A_0, \ \int_{\mathbb{R}} \phi = 1 \},
\end{align}
and for $\phi \in \mathcal{O}$, let
\begin{align*}
\vec{N}_\epsilon( \phi_k * \vec{f} : k ), \; \; \;
\end{align*}
 be the jump-counting function at altitude $\epsilon > 0$ with respect to $\mathcal{H}$, where $\phi_k := 2^{-k} \cdot \phi(2^{-k} \cdot)$ are the usual $L^1(\mathbb{R})$-normalized dilations.

Then there exists an absolute constant, $C(A_0) < \infty$,  independent of $\phi, \ \mathcal{H} $, so that
\[ \sup_{\epsilon > 0} \| \epsilon \cdot \vec{N}_\epsilon(\phi_k* \vec{f} : k)^{1/2} \|_{\ell^2(\mathbb{Z})} \leq C(A_0) \cdot \| \vec{f} \|_{\ell^2(\mathcal{H})}.\]
\end{lemma}

The following has a similar proof to Corollary \ref{c:sep}.

\begin{corollary}\label{c:jump0}
In the setting of Corollary \ref{c:sep}, one may replace
\[ A_k^\theta f(x) \longrightarrow A_k^{\phi,\theta} f(x) := \sum_n 2^{-k} \cdot \phi_k(n/2^k) \cdot f(x-n) \cdot e(-n \cdot \theta), \]
for $\phi \in \mathcal{O}$. Then, with
\begin{align*}
\vec{N}_\epsilon( A_k^{\phi} \vec{f}(x) : k )
\end{align*}
the analogous jump counting function with respect to $\ell^2(\Lambda)$, there exists an absolute constant, $C$, independent of $\Lambda, \phi$, so that
\[ \sup_{\epsilon > 0} \| \epsilon \cdot \vec{N}_\epsilon(A_k^{\phi } \vec{f} : k )^{1/2} \|_{\ell^2(\mathbb{Z})} \leq C \cdot (1 + \frac{|\Lambda|^{1/2}}{A}) \cdot \| f\|_{\ell^2(\mathbb{Z})}.\]
\end{corollary}

We will also need to make use of certain \emph{variational estimates}: for a sequence of vectors $\{ \vec{a_n}\} \subset \mathcal{H}$, define the \emph{$r$-variation}
\[ \mathcal{V}^r(\vec{a_n}) := \sup \big( \sum_i \|\vec{a_{n_i}} - \vec{a_{n_{i+1}}}\|_{\mathcal{H}}^r \big)^{1/r} + \sup \|\vec{a_n} \|_{\mathcal{H}} \]
where the supremum is over all finite increasing subsequences; note the elementary inequalities
\begin{align}\label{e:prod} \mathcal{V}^r(\vec{a_n} \cdot \vec{b_n}) \lesssim \mathcal{V}^r(\vec{a_n}) \cdot \mathcal{V}^r(\vec{b_n}),\end{align}
where $\vec{a_n} \cdot \vec{b_n}$ is given by
\[ \vec{a_n} \cdot \vec{b_n} = \{ a_n(i) \cdot b_n(i) : i \}, \; \; \; \vec{a_n} = (a_n(1),a_n(2),\dots), \ \vec{b_n} = (b_n(1),b_n(2),\dots)\]
and
\[ \mathcal{V}^r(\vec{a_n})^r \leq \sum_k 2^{-kr} \cdot \vec{N}_{2^{-k}}(\vec{a_n}) + \sup \| \vec{a_n} \|_{\mathcal{H}},\]
so that whenever $\sup_n \| \vec{a_n} \|_{\mathcal{H}} \leq 1/10$
\begin{align}
    \mathcal{V}^r(\vec{a_n})^r \leq \sum_{k \geq 0} 2^{-kr} \cdot \vec{N}_{2^{-k}}(\vec{a_n}) + \sup \| \vec{a_n} \|_{\mathcal{H}}.
\end{align}

Although significantly more is true, the following cheap estimate will suffice for our purposes.

\begin{lemma}\label{l:var}
Let $(X,\mu)$ be a measure space, and $2 < r < \infty$. Suppose that
\[ \sup_{\epsilon > 0} \| \epsilon \cdot \vec{N}_\epsilon(\vec{f_k})^{1/2} \|_{L^2(X)} +  \| \sup_k \|\vec{f_k}\|_{\mathcal{H}} \|_{L^2(X)} \leq A. \]
Then
\[ \| \mathcal{V}^r(\vec{f_k}) \|_{L^{2,\infty}(X)} \lesssim \frac{r}{r-2} \cdot A.
\]
\end{lemma}
\begin{proof}
    By homogeneity, it suffices to prove that
    \[ \mu( \{ \mathcal{V}^r(\vec{f_k}) \geq 1 \}) \lesssim (1 + (r-2)^{-2}) \cdot A^2.\]
We bound the left-hand side by
\[ \mu( \{ \mathcal{V}^r(\vec{f_k})^r \geq 1, \ \sup_k \| \vec{f_k} \|_{\mathcal{H}} \leq 1/10 \}) + O(A^2) \]
and then, with $\epsilon = \epsilon_r := \frac{r-2}{10}$
\begin{align*}
    &\mu( \{ \mathcal{V}^r(\vec{f_k})^r \geq 1, \ \sup_k \| \vec{f_k} \|_{\mathcal{H}} \leq 1/10 \}) \leq \mu( \{ \sum_{k \geq 0} 2^{-kr} \vec{N}_{2^{-k}}(\vec{f_k}) \geq 1/2 \}) \\
    & \qquad \leq \sum_{k \geq 0} \mu( \{ 2^{-rk} \vec{N}_{2^{-k}}(\vec{f_k}) \gtrsim \epsilon 2^{-\epsilon k} \}) \\
    & \qquad \qquad = \sum_{k \geq 0} \mu( \{ 2^{-2k} \vec{N}_{2^{-k}}(\vec{f_k}) \gtrsim (r-2) \cdot 2^{\frac{r-2}{2} \cdot k} \}) \\
    & \qquad \qquad \qquad \leq (r-2)^{-1} \sum_{k \geq 0} 2^{\frac{2-r}{2}  \cdot k} A^2 \lesssim (r-2)^{-2} A^2,
\end{align*} 
as desired.
\end{proof}


\subsection{Tools: Dyadic Multiplier Theory}\label{s:DG}
In this section introduce and use a useful tool in harmonic analysis: shifted dyadic grids. Below, every interval we introduce will have side length $\geq 1$.

A grid is a collection of intervals $\{ Q : Q \in \mathcal{Q} \}$ so that whenever $P,Q \in \mathcal{Q}$ 
\[ P \cap Q \in \{ P,Q,\emptyset\}, \]
(up to null sets). There will be no loss of generality in assuming that all side-lengths are of the form $\lambda 2^{\mathbb{N}}$ for $\leq \lambda = 2^{n/B} < 2 $ for some rational $0 \leq n/B < 1$, so we will restrict to this case.

The standard example, which will be the most important for us, is the usual dyadic grid
\[ \{ 2^k \cdot \big( n + [0,1) \big): k \geq 0, \ n \in \mathbb{Z} \}.\]

Throughout the argument, we will need the flexibility to work with many different dyadic grids, so we address the more general construction; below $\Delta \in 2\mathbb{N} + 1$ will be an odd integer which we will eventually specialize to 
\[ \Delta = t^{-10} + O(1),\]
see \eqref{e:Delta}.

With $t^{-1} \ll B_t \lesssim t^{-1}$ an integer, 
and $0 \leq b \leq B_t-1$, and define the \emph{shifted dyadic grids} to be
\[ \mathcal{D}_k^{\Delta,L,b} := \{ 2^{b/B_t} \cdot 2^k \cdot \big( n + L/\Delta + [0,1) \big)  : n\in \mathbb{Z}, \ L \in [\Delta]\} \]
and
\begin{align}\label{e:shiftgrid} \mathcal{D}_U^{\Delta,L,b} := \bigcup_{k \equiv U \mod \Delta-1} \mathcal{D}_k^{\Delta,L;b}, \end{align}
noting that $2^{\Delta-1} \equiv 1 \mod \Delta$ by Fermat's Little Theorem.

For each $L \in [\Delta]$, define
\begin{align}\label{e:grid0} 
\overline{\mathcal{D}_U^{\Delta,L,b}} &:= \bigcup_{k \equiv U \mod \Delta-1} 
\overline{\mathcal{D}_k^{\Delta,L,b}} \notag \\
& \qquad := \bigcup_{k \equiv U \mod \Delta-1} \{ 2^{b/B_t} \cdot 2^k \cdot \big( n + L/\Delta + [0,1/\Delta) \big)  : n \in \mathbb{Z}\}.
\end{align}

Note that for 
\[ x \in \overline{\mathcal{D}_U^{\Delta,L,b}}, \ x \in I \in \mathcal{D}_U^{\Delta,L,b} \] 
we have the smoothness property:
\[ \frac{|\{ I \triangle \big( x + [0, |I|) \big) \}|}{|I|} \lesssim \Delta^{-1}, \]
where we use $\triangle$ to denote symmetric difference
\[ A \triangle B := (A \smallsetminus B) \cup (B \smallsetminus A).\]

In what follows, this smoothness property will be crucial.
For notational ease, we will work with (a sparse subset of) the standard dyadic grid
\begin{align}\label{e:DD}
\mathcal{D} := \mathcal{D}_0^{\Delta,{0},0}, \; \; \; \overline{\mathcal{D}} := \overline{\mathcal{D}_0^{\Delta,{0},0}}; 
\end{align}
often, we will root the grid inside of a parental interval, $P \subset I_0$, thus
\[ \mathcal{D}(P) := \{ I \in \mathcal{D} : I \subset P \}, \; \; \; \overline{\mathcal{D}}(P) := \{ \overline{I} := I \cap \overline{\mathcal{D}} \} \cap P. \]
Our estimates, however, will be uniform in each grid, so our arguments be flexible enough to obtain a union bound at the close; we will indicate how this is done.

The central object of consideration in this paper are Fourier multipliers which are indexed by elements of $\mathcal{D}(I_0)$,
\begin{align}\label{e:OM}
    \Omega_J(\beta) := \frac{1}{|J|} \sum_{n \in J} \phi((n-c_J)/|J|) g(n) e(-n \beta) 
\end{align}
where $J = [c_J,c_J + |J|)$ is dyadic, $\phi \leq \mathbf{1}_{[0,1]}$ is a smooth approximate to the indicator of the unit interval, and $g$ is a fixed $1$-bounded function; note that for any $x_J \in J$, and any function
\[ \mathbf{1}_{[-3N,3N]} \leq \varphi_N^{\vee}(n), \; \; \; |J| = N\]
(say)
\begin{align}\label{e:rep} e(x_J \beta) \Omega_J(\beta) = \varphi_N * \big( e(x_J \beta) \Omega_J(\beta) \big).
\end{align}

With $\psi(t)$ as smooth approximation to $\mathbf{1}_{t \approx 1}$ satisfying
\[ \mathbf{1}_{(0,\delta_0]} \leq \sum_{0 < \delta \leq \delta_0} \psi(\delta^{-1} t) \leq \mathbf{1}_{(0,2\delta_0]} \]
define
\begin{align}\label{e:Psi} \Psi_\delta(t) := t \cdot \psi(\delta^{-1} |t|)
\end{align}
so that
\begin{align}\label{e:C1} \sup_\delta \| \Psi_\delta \|_{\mathcal{C}^1(\mathbb{R})} \lesssim 1 \end{align}
are uniformly $\mathcal{C}^1$ (and in particular, uniformly Lipschitz); set
\begin{align}\label{e:OM0} \Omega_{J,\delta} := \Psi_\delta( \Omega_J) \end{align}
to be the smooth restriction of $\Omega_J$ to its $\delta$-level set,
so that
\[ e(a \beta) \Omega_{J,\delta}(\beta) = \Psi_\delta\big( e(a \beta) \Omega_J(\beta) \big) \]
and thus
\[ \| \partial_\beta \big( e(a \beta) \Omega_{J,\delta}(\beta) \big) \|_{L^2(\mathbb{T})} \lesssim \| \partial_\beta \big( e(a \beta) \Omega_{J}(\beta) \big) \|_{L^2(\mathbb{T})} \lesssim N^{1/2}
\]
by \eqref{e:C1}, provided $a \in J$. 

The following approximation property will be used repeatedly.

\begin{lemma}\label{l:mult1}
    Suppose that $a \in \overline{\mathcal{D}}(I_0)$ and that 
    \[ J(a) = (a, a+N].\]
    If $0 \leq \eta \leq 1$ with $\| \eta' \|_{L^{\infty}(\mathbb{T})} \leq U N$, and $a \in J \in \mathcal{D}$ with $|J| = N$, then
    \[ \sup_\delta \| \big( \Omega_{J(a),\delta} - \Omega_{J,\delta} \big)  \eta \|_{A(\mathbb{T})} \lesssim (U/\sqrt{\Delta})^{1/2}.\]
And
\[ \sup_\delta \| (\Omega_{J(a),\delta} \eta)^{\vee} \|_{\ell^1( (VJ)^c)} + \sup_\delta \| (\Omega_{J,\delta} \eta)^{\vee} \|_{\ell^1( (VJ)^c)} \lesssim (U/\sqrt{V}).\]  
\end{lemma}
\begin{proof}
    Both points follow from Lemma \ref{l:H12}. Since 
    \[ \sup_\delta \| \Psi_\delta \|_{\mathcal{C}^1(\mathbb{R})} \lesssim 1,\]
    we may replace all instances of $\Omega_{J,\delta}$ with $\Omega_J$, see \eqref{e:OM}.
    
    Indeed
    \[ \|  \Omega_{J(a),\delta} - \Omega_{J,\delta} \|_{L^2(\mathbb{T})} \lesssim (\Delta N)^{-1/2}
        \]
as the composition with $\Psi_\delta$ decreases $L^2(\mathbb{T})$-norm
and
\begin{align*} \| \partial_\beta \Big( e(a \beta) \big( \Omega_{J(a),\delta}(\beta) - \Omega_{J,\delta}(\beta) \big)  \cdot \eta(\beta)\Big) \|_{L^2(\mathbb{T})} &\lesssim N^{1/2} + \big( \| \eta' \|_{L^{\infty}(\mathbb{T})} \cdot( \| \Omega_{J(a)} \|_{L^2(\mathbb{T})} + \| \Omega_{J} \|_{L^2(\mathbb{T})} ) \big) \\
& \qquad \lesssim U N^{1/2}
\end{align*}
as 
\[ \partial_\beta ( e(a \beta) \big( \Omega_{J(a),\delta} - \Omega_{J,\delta}(\beta) \big) ) = \partial_\beta \big( \Psi_\delta( e(a \beta) \Omega_{J(a)}(\beta) - \Psi_\delta( e(a \beta) \Omega_{J}(\beta) \big) \]
where composition again decreases the norm by \eqref{e:C1}.

The second point follows from Lemma \ref{l:H12}, namely \eqref{e:notJ}.     
\end{proof}

We finish this section by recording the elementary containment, which directly follows from Lemma \ref{l:mult1}.

\begin{lemma}\label{l:cont}
Suppose that $|I| = M$ is a fairly large (dyadic) interval, and that $N \leq M/2$. For each $a \in \overline{\mathcal{D}}(I_0)$, set $J_N(a) := (a, a + N]$ and choose $J(a) \in {\mathcal{D}}$ with $|J(a)| = N$ so that $a \in J(a)$.

If $\eta$ is as above, $|f| \leq 1$ with $f_P := f \cdot \mathbf{1}_P$ for $P$ intervals, and $\Delta \gg t^{-4} U^2$, then for any $0 < \delta \leq 1$
\begin{align*}
    &\{ x \in I : |\int \Omega_{J_N(x),\delta}(\alpha \beta) \eta(\alpha \beta) \hat{f}(\beta) e( \beta ( x  + \alpha x ) ) \ d\beta | \geq t \} \\
    & \qquad \subset 
\{ x \in I : |\int \Omega_{J(x),\delta}(\alpha \beta) \eta(\alpha \beta) \cdot \widehat{f_{J}}(\beta) e( \beta ( x + \alpha x) ) \ d\beta | \geq t/2  \}, \; \; \; J \supset D^2  t^{-4} U^2 \cdot J(x),
\end{align*}
see \eqref{e:Da}.
\end{lemma}

\section{Addressing the Floor Function}\label{s:floor}

To resolve the floor function, our first order of business is to smoothly localize our multipliers by restricting the Fourier support of our initial data.

To do so, let 
\[ \mathbf{1}_{[-1/4,1/4]} \leq \psi \leq \mathbf{1}_{[-1,1]}\] be a smooth bump function satisfying
\[ \sum_{l \in \mathbb{Z}} \psi(\xi - l) \equiv 1\]
for $\xi \in \mathbb{R}$, and let
\[ \mathbf{F}(x) := \big( \frac{\sin(\pi x)}{\pi x} \big)^2\]
denote the Fej\'{e}r kernel, so
\[ \widehat{\mathbf{F}}(\beta) = (1 - |\beta|)_+.\]
\begin{lemma}
For $D \in \mathbb{Z}$, see \eqref{e:Da}, define
    \[ \widehat{\mathbf{F}_d}(\beta) := \psi(D\beta -d) \cdot \widehat{\mathbf{F}}(\beta). \]
Then for any (compactly supported) $f \in \ell^{2}(\mathbb{Z})$
\[ f(n) = \sum_{|d| \leq 10 D} \mathbf{F}_d*f(n) = \sum_{|d| \leq 10 D} \Big( \sum_k f(n-k) \mathbf{F}_d(k) \Big).\]    
\end{lemma}
\begin{proof}
    By linearity, it suffices to specialize $f$ to be the pointmass at $0$, $f \longrightarrow \delta_{\{0\}}$. But, the result then just follows from the reproducing identity and Fourier inversion,
    \[ \sum_{|d| \leq 10 D} \psi(D \beta -d) \cdot \widehat{\mathbf{F}}(\beta) \equiv \widehat{\mathbf{F}}(\beta),\]
noting that
\[ \mathbf{F} \cdot \mathbf{1}_{\mathbb{Z}} \equiv \delta_{\{0\}}.\]
\end{proof}

\begin{lemma}\label{l:w}
    Suppose that $w \in \mathcal{C}_c^{\infty}([0,1])$ with
    \[ |\partial^{\gamma} w| \leq L^\gamma \mathbf{1}_{[0,1]} 
    \]
for all $\gamma \geq 0$.
Then for each $|f|, |g| \leq 1$
\begin{align}
    &\frac{1}{N} \sum_{n \leq N} \sum_m w(\alpha n - m) f(x-m) g(x+n) \\
    & \qquad = \sum_{\xi \in \mathbb{Z}} e(-\alpha x \xi) \cdot \int_{\mathbb{T}} \widehat{f_\xi}(\beta) \Big(\frac{1}{N} \sum_{n \leq N} g_\xi(x+n) e(-n \alpha \beta) \Big) e( \beta x ) \ d\beta \\
    & \qquad \qquad = \sum_{|\xi| \leq L/\tau} e(-\alpha x \xi) \cdot \int_{\mathbb{T}} \widehat{f_\xi}(\beta) \Big(\frac{1}{N} \sum_{n \leq N} g_\xi(x+n) e(-n \alpha \beta) \Big) e( \beta x ) \ d\beta + O_A(\tau^{A}),
\end{align}
where
\[ \widehat{f_\xi}(\beta) := \hat{f}(\beta) \cdot \hat{w}(\xi-\beta)\]
and
\[ g_\xi(n) := e(\alpha n \xi) \cdot g(n)\]
satisfy
\[ \| f_\xi \|_{\ell^{\infty}(\mathbb{Z})} \lesssim_A (1 + |\xi|/L)^{-A}, \; \; \; \|g_\xi\|_{\ell^{\infty}(\mathbb{Z})} \leq 1. \]

And, similarly with the rough cut-offs replaced by smooth ones, 
\[ \frac{1}{N} \mathbf{1}_{[1,N]} \longrightarrow \phi_N\]
see \eqref{e:tilde} and \eqref{e:tilde0}.
        \end{lemma}
\begin{proof}
    The argument is by Poisson summation. We express
    \begin{align} &\frac{1}{N} \sum_{n \leq N} \sum_m w(\alpha n - m) f(x-m) g(x+n) \\
    & \qquad = \int_{\mathbb{T}} \hat{f}(\beta) \frac{1}{N} \sum_{n \leq N} g(x+n) \Big( \sum_m w(\alpha n - m) e(-m \beta) \Big) e(\beta x) \ d\beta \\
    & \qquad \qquad = \sum_{\xi \in \mathbb{Z}} \Big( \int_{\mathbb{T}} \big( \hat{f}(\beta) w^{\vee}(\xi + \beta) \big) \frac{1}{N} \sum_{n \leq N} g(x+n) e(- \alpha n (\beta + \xi)) e(\beta x) \ d\beta \Big) \\
    & \qquad \qquad \qquad = \sum_{\xi \in \mathbb{Z}} e(-\alpha x \xi) \Big( \int_{\mathbb{T}} \widehat{f_\xi}(\beta) \frac{1}{N} \sum_{n \leq N} g_\xi(x+n) e(- \alpha n \beta) e(\beta x) \ d\beta \Big);
    \end{align}
by convexity, we may bound
\[ \| \Big( \int \widehat{f_\xi}(\beta) \frac{1}{N} \sum_{n \leq N} g_\xi(x+n) e(- \alpha n \beta) e(\beta x) \ d\beta \Big) \|_{\ell^{\infty}} \lesssim \| f_\xi \|_{\ell^{\infty}} \| g_\xi \|_{\ell^{\infty}},\]
    from which the result follows, upon using Lemma \ref{l:H12} to bound
    \begin{align}
        \| \hat{w}(\xi - \beta) \|_{A_\beta(\mathbb{T})} &\lesssim \| \hat{w}(\xi - \beta)\|_{L^2_\beta(\mathbb{T})} +  \| \hat{w}(\xi - \beta)\|_{L^2_\beta(\mathbb{T})}^{1/2} \| \partial_\beta \hat{w}(\xi - \beta)\|_{L^2_\beta(\mathbb{T})}^{1/2}\\ & \qquad \lesssim_A (1 + |\xi|/L)^{-A}. 
    \end{align} 
\end{proof}

In what follows, we will use reproducing to decompose each 
\[ f = \sum_{|d| \leq 10D} f*\mathbf{F}_d, \]
and, after conceding a constant factor, will assume that each $f$ has Fourier support in an interval of length $ \leq 1/D$, possibly after slightly decreasing $D$; by modulating $f$ and $g$ appropriately, there is no loss of generality in assuming $\hat{f}$ is supported in $[0,1/D]$, so we will use without comment the reproducing identity
\begin{align}\label{e:psiD} f \equiv f*\psi_D, \; \; \; \mathbf{1}_{[0,1/D]} \leq \widehat{\psi_D} \equiv \hat{\psi}(D \cdot) \leq \mathbf{1}_{[-2/D,2/D]}.\end{align}
In particular, this model no longer distinguishes between
\[ \sum_m \phi_M(m) f(x - \lfloor \alpha m \rfloor) g(x+m) \]
and the case (formally) connected to pointwise convergence of two commuting transformations,
\[ \sum_m \phi_M(m) f(x - \alpha m ) g(x+m), \]
see \eqref{e:commute}.

\section{Organizing Fourier Multipliers}\label{ss:fol}
Recalling the notation $\{ \Omega_{J,\delta} \}$, see \eqref{e:OM0}, we implement a selection procedure to organize our multipliers into sub-collections which can be efficiently studied, cf.\ \eqref{e:BB}.

The basic mechanism is to efficiently extract not-too-many subsets
\[ \{ \Lambda \} \subset \mathbb{T} \]
of size not much larger than $\delta^{-2}$, see \eqref{e:specsize}, so that for each $J$, there is some $\Lambda$ so that, essentially,
\[  \| \Omega_{J,\delta} \cdot \mathbf{1}_{\big(\Lambda + B(10/|J|)\big)^c} \|_{A(\mathbb{T})} \ll t;
\]
the task reduces to studying multipliers roughly of the form  
\[ \Big\{ \Omega_{J,\delta} \cdot \mathbf{1}_{\Lambda + B(10/|J|)} : \| \Omega_{J,\delta} \cdot \mathbf{1}_{\big( \Lambda + B(10/|J|) \big)^c} \|_{A(\mathbb{T})} \ll t \Big\}_\Lambda \]
as $\Lambda$ varies; but these types of multipliers will be essentially of the form \eqref{e:AI}.

In the below section, we will regard $t > 0$ as fixed.

\subsection{Separation into Trees}
With 
\[ M_0 \ll M_1 \ll \dots \ll M_L \]
a rapidly increasing, lacunary sequence, we collect
\[ \mathcal{Q}_l := \{ I \in \mathcal{D}(I_0) : M_{l-1} \leq |I| < M_{l} \} \] and choose arbitrary selectors
\[ \mathcal{I}_l : I_0 \to \mathcal{D}(I_0) \cap \mathcal{Q}_l. \]

The structures we introduce will depend on our sequence and selectors, but all estimates will be uniform (and, in particular, independent of $L$).

We will exhibit a partition of $\mathcal{D}(I_0)$ into $V$
many subsets, ``trees,"
\[ \mathcal{D}_j(I_0), \ 0 \leq j < V \]
and an exceptional collection of intervals, $\mathcal{D}_V(I_0)$, so that we can express
\[ \bigcup_{\mathcal{D}_j(I_0)} I = \bigcup_{T_j^{\text{max}}} I_j \]
as a disjoint union of maximal dyadic intervals, ``tree tops," so that for each $I_j \in T_j^{\text{max}}$ there exists a unique $I_{j-1} \in T_{j-1}^{\text{max}}$ with $I_j \subset I_{j-1}$; $I_0$ anchors the inductive construction. We use the notation
\[ T_j^{\text{max}}(I_{j-1}) := \{ I \in T_j^{\text{max}} : I \subset I_{j-1} \}. \]

The remaining intervals will be subsumed in 
\[ \mathcal{D}_V(I_0), \]
and will be very localized,
\begin{align}\label{e:loc} |\bigcup_{\mathcal{D}_V(I_0)} I| = \sum_{T_V^{\text{max}}} |I_V| \lesssim t^{-4} \delta^{-o(1)} V^{-1} |I_0|; 
\end{align}
 intervals contained in $\mathcal{D}_V(I_0)$ will only effect the argument minimally.

For finite $\Lambda \subset \mathbb{T}$ and $J \in \mathcal{D}(I_0)$ we will make use of a family of bump functions:
\[ \big\{ \chi_{\Lambda,J} : \Lambda \subset \mathbb{T} \text{ finite, } J \in \mathcal{D}(I_0) \big\} \]
satisfying
\begin{align}\label{e:chi0} |\chi_{\Lambda,J}| \leq 1, \; |\chi_{\Lambda,J}'| \lesssim |J| \end{align}
and
\begin{align}\label{e:chi} \mathbf{1}_{\big( \Lambda + B(25/J) \big)^c} \leq \chi_{\Lambda,J} \leq \mathbf{1}_{\big( \Lambda + B(20/J) \big)^c}.\end{align}

\begin{definition}
    An $(L,V)$-\emph{forest} consists of disjoint collections of intervals $\{ \mathcal{D}_j(I_0) : 0 \leq j\leq V\}$, equipped with nested ``tree tops" $T_j^{\text{max}} \subset \mathcal{D}_j(I_0)$ as above, so that each $I_j \in T_j^{\text{max}}$ comes equipped with a finite set of frequencies, $\Lambda_j(I_j) \subset \mathbb{T}$, satisfying the following properties:
    \begin{itemize}
    \item The forest exhausts most of $I_0$: \eqref{e:loc} holds;
    \item There are not too many frequencies involved: $|\Lambda_0(I_0)| \lesssim \delta^{-2}$, and in general
      \[ \sup_{j < V} \sup_{I_j \in T_{j}^{\text{max}}} |\Lambda_j(I_j)| \lesssim V \delta^{-2}; \]
      \item The frequencies respect nesting: if $I_j \in T_j^{\text{max}}(I_{j-1})$ then $\Lambda_{j-1}(I_{j-1}) \subset \Lambda_j(I_j)$;
    \item We may localize each $\{ \Omega_{J,\delta} : J \in \mathcal{D}_j(I_j), \ I_j \in T_j^{\text{max}} \}$ to a small neighborhood of $\Lambda_j(I_j)$
\[ \sup_{J \in \mathcal{D}_j(I_j), \ I_j \in T_j^{\text{max}}} \| \Omega_{J,\delta} \cdot \chi_{\Lambda_j(I_j),J} \|_{A(\mathbb{T})} \ll t, \]
see \eqref{e:chi}.
    \end{itemize}
    We refer to a $(1,V)$-forest as just a $V$-forest.
    
    \end{definition}
  
\begin{lemma}[Forest Lemma]\label{l:forest}
For each $L$, $V \leq R^{O(1)}$, we may organize $\mathcal{D}(I_0)$ into an $(L,V)$-forest.   \end{lemma}

Before turning to the proof, we begin with some notation; we first set
\[ P_I = |I| \cdot \Omega_I, \; \; \; P_{I,\delta} := |I| \cdot \Omega_{I,\delta} \]
for $I \in \mathcal{D}(I_0)$.

For each dyadic $\omega$ with $|\omega| |I| = 1$ so that
\[ \| P_I \|_{L^{\infty}(\omega)} \approx \delta |I|,\]
let $c_\omega$ denote the left end-point, and collect 
\[ \Sigma_I := \{ c_\omega : \| P_I \|_{L^{\infty}(\omega)}   \approx \delta |I| \}; \]
we let $\xi_\omega \in \omega$ denote the frequencies that realizes the supremum
\[ \| P_I \|_{L^{\infty}(\omega)} = |P_I(\xi_\omega)| \approx \delta |I|.\]
Note that
\[ (\delta |I|)^2 \cdot |\Sigma_I| \lesssim \sum_{\xi_\omega} |P(\xi_\omega)|^2 \lesssim |I| \cdot \| P \|_{L^2(\mathbb{T})}^2 \lesssim |I|^2 \]
by Lemma \ref{l:rsum}, so that
\[ |\Sigma_I| \lesssim \delta^{-2},\]
uniformly in $I$, see \eqref{e:specsize}.

An interval, $J$, will be said to be \emph{localized} (or $\rho$-localized, see \eqref{e:rho}) with respect to a collection of frequencies, $\Lambda$, if
\[ \sum_{|\omega| = |J|^{-1}, \ \omega \cap \Lambda = \emptyset} \| P_{J,\delta} \|_{L^2(\omega)}^2 \leq \rho^2 \cdot |J| \ll t^4 \cdot |J|;\]
 otherwise, $J$ will be called \emph{diffuse} with respect to $\Lambda$. Note that by Lemma \ref{l:H12}, if $J$ is localized to $\Lambda$, then for any $\chi_{\Lambda,J}$ satisfying \eqref{e:chi0} and \eqref{e:chi}, we may bound
\[ \| \Omega_{J,\delta} \cdot \chi_{\Lambda,J} \|_{A(\mathbb{T})} \lesssim \rho^{1/2} \ll t, \]
see \eqref{e:rho}.

So, initiate $\Lambda_0(I_0) := \Sigma_{I_0}$, and define
\[ X_1 := \{ r : \mathcal{I}_k(r) \text{ is diffuse with respect to } \Lambda_0(I_0), \ k \leq L \text{ is maximal} \} \]
 and let $T_1^{\text{max},l}$ denote the set of maximal dyadic $I \in \mathcal{D}(I_0) \cap \mathcal{Q}_l$ inside $X_1$, consolidating
 \[ T_1^{\text{max}} := \bigcup_{l \leq L} T_1^{\text{max},l},\]
so that
 \begin{itemize}
     \item We may uniquely decompose $X_1 = \bigcup_{T_1^{\text{max}}} I_1$; and
     \item If $r \in I_1 \in T_1^{\text{max},l_1}$, then for all $l_1 < k \leq L$, $\mathcal{I}_k(r)$ is localized with respect to $\Lambda_0(I_0)$, and for $k \leq l_1$, $\mathcal{I}_k(r) \subset I_1.$
 \end{itemize}

Now, for each $I_1 \in T_1^{\text{max}}$, set 
 \[ \Lambda_1(I_1) := \Lambda_0(I_0) \cup \Sigma_{I_1}.\]

We next localize the construction to each individual $I_1 \in T_1^{\text{max}}$ and iterate. Specifically, for
\[ I_1 \in T_1^{\text{max},l_1}, \; \; \; 1 \leq l_1 \leq L\]
we set  
  \[ X_2 \cap I_1:= \{ r \in I_1 : \mathcal{I}_k(r) \text{ is diffuse with respect to } \Lambda_1(I_1), \ k \leq l_1 \text{ is maximal} \} \]
 and
 \[ X_2 := \bigcup_{l \leq L} \bigcup_{I_1 \in T_1^{\text{max},l}} (X_2 \cap I_1). \]

We similarly define
\begin{align}
    T_2^{\text{max}} = \bigcup_{l\leq L} (T_2^{\text{max}} \cap \mathcal{Q}_l) := \bigcup_{l \leq L} T_2^{\text{max},l}.
\end{align}
where $T_2^{\text{max}}$ denotes the set of maximal dyadic $I \in \mathcal{D}(I_1)$, so that
 \begin{itemize}
     \item We may uniquely decompose $X_2 \cap I_1 = \bigcup_{T_2^{\text{max}}(I_1)} I_2$;
     \item If $I_1 \in T_1^{\text{max},l_1}$, then $I_2 \in \bigcup_{l \leq l_1} \mathcal{Q}_l$; 
     and
     \item If $r \in I_2 \in T_2^{\text{max}}(I_1) \cap \mathcal{Q}_{l_2}$, then for all $l_2 < k \leq l_1$, $\mathcal{I}_k(r)$ is localized with respect to $\Lambda_1(I_1)$, and for all $k \leq l_2$, $\mathcal{I}_k(r) \subset I_2$.
 \end{itemize}
Consolidate
\[ T_2^{\text{max},l} = \bigcup_{I_1 \in T_1^{\text{max}}} T_2^{\text{max},l}(I_1) \]
and
\[ T_2^{\text{max}} := \bigcup_{l \leq L} T_2^{\text{max},l}.\]
Continuing in this fashion, we will foliate
 \[ I_0 = \bigcup_{1 \leq j < V} (X_{j} \smallsetminus X_{j+1}) \cup X_V, \]
and let
\[ \mathcal{D}_j(I_0) := \bigcup_{l \leq L} \bigcup_{I_j \in T_j^{\text{max},l}} \{ \mathcal{I}_k(r) \subset I_j \text{ localized, } k \leq l \},  \]
with
\[ \mathcal{D}_j(I_j) = \mathcal{D}_j(I_0) \cap I_j \]
for $I_j \in T_j^{\text{max}}$.
 
Thus to establish Lemma \ref{l:forest}, it suffices to prove the exceptional set estimate,
\begin{align}\label{e:except} |X_V| := |\bigcup_{I \in \mathcal{D}_V(I_0)} I | \lesssim t^{-4} \delta^{-o(1)} V^{-1} |I_0|. \end{align}

\begin{proof}[Proof of Lemma \ref{l:forest}]
    Our job is to verify \eqref{e:except}; we accomplish this by using volume packing. So, suppose that for all $v \leq V$,
    \[ |\bigcup_{I \in \mathcal{D}_v(I_0)} I| = \sum_{T_v^{\text{max}}} |I_v| \geq L_0^{-1} |I_0|, \; \; \; L_0 \lesssim R^{O(1)}; \]
our goal will be to bound $V \lesssim t^{-4} \delta^{-o(1)} L_0$. 

But since $\{ I_v \}$ are diffuse
\begin{align}
    \sum_{T_{v-1}^{\text{max}}} \sum_{I_v \in T_{v}^{\text{max}}(I_{v-1})} |I_v| & \lesssim \rho^{-2} \sum_{T_{v-1}^{\text{max}}} \; \sum_{|\omega| |I_v| = 1,\ \omega \cap \Lambda_{v-1}(I_{v-1}) = \emptyset } \| P_{I_v} \|_{L^2(\omega)}^2 \\
    & \qquad \lesssim \rho^{-2} \sum_{T_{v-1}^{\text{max}}} \; \sum_{|\omega| |I_v| = 1, \ \omega \cap \Lambda_{v-1}(I_{v-1}) = \emptyset } \| P_{I_v} \|_{L^\infty(\omega)}^2 /|I_v| \\
    & \qquad \qquad \lesssim \rho^{-2} \sum_{T_{v-1}^{\text{max}}} \; \sum_{|\omega| |I_v| = 1,\ \omega \cap \Lambda_{v-1}(I_{v-1}) = \emptyset } |\langle \psi_{I_v \times \omega},g\rangle|^2 \\
    & \qquad  \qquad \qquad =: \rho^{-2} \sum_{I_v \in T_v^{\text{max}}} \sum_{s \in S(I_v)} |\langle \psi_s,g\rangle|^2,
\end{align} 
where
\[ \psi_{s}(n) := \frac{1}{|I_v|^{1/2}} \phi(\frac{n-c_I}{|I|}) e(n \xi_\omega), \; \; \; s = I_v \times \omega\]
so that
\[ |\langle \psi_{I_v \times \omega}, g \rangle | \approx \delta |I_v|^{1/2}.\]

The key claim is that
\[ \bigcup_{v \leq V} \bigcup_{I_v \in T_v^{\text{max}}} S(I_v) \]
is a disjoint collection of tiles: if 
\[ S(I_v) \cap S(J_w) \neq \emptyset\]
with $I_v \in T_{v}^{\text{max}}, \ J_w \in T_w^{\text{max}}$, $v  > w$, then necessarily $I_v \subsetneq J_w$, and no frequency interval in $S(I_v)$ intersects $\Lambda_{v-1}(I_{v-1})$, while the set of frequency intervals in $S(J_w)$ are contained in $\{ |\omega| = |J|^{-1} : \omega \cap \Lambda_{v-1}(I_{v-1}) \neq \emptyset \}$.

In particular, summing over all $v$ so that
\[ \sum_{T_v^{\text{max}}} |I_v| \geq L_0^{-1} |I_0|\]
we bound
\begin{align}
    V L_0^{-1} |I_0| &\lesssim \rho^{-2} \sum_{v \leq V} \sum_{I_v \in T_v^{\text{max}}} \sum_{s \in S(I_v)} |\langle \psi_s,g\rangle|^2 \\
    & \qquad \lesssim \rho^{-2} \delta^{-o(1)} |I_0|
\end{align} 
by Proposition \ref{p:tfa}, so there are at most 
\[ V \lesssim \rho^{-2} \delta^{-o(1)} L_0 \approx t^{-4} \delta^{-o(1)} L_0\]
many sets $\{ \mathcal{D}_j(I_0)\}$.
\end{proof}

\subsection{From Trees to Branches}
In this section, we ``prune" each tree in our $(L,V)$-forest into ``branches," which have the additional feature that the frequencies linked to each branch are ``widely separated" relative to scale. These ``branches" will be as in \eqref{e:BB}.

We begin with the following definition.

\begin{definition}
Regarding an $(L,V)$ forest as fixed, a $U$-\emph{branch} consists of a (possibly empty) disjoint collection of intervals $\mathcal{B}_s(I_j) \subset \mathcal{D}_j(I_j), \ I_j \in T_j^{\text{max}}, \ s \leq U$ and finite collections of frequencies $\Lambda_{j,s}(I_j) \subset \mathbb{T}$, so that
    \begin{itemize}
        \item There are not too many frequencies involved:
        \[ \Lambda_{j,s}(I_j) \subset \Lambda_j(I_j); \]
        \item The frequencies are widely separated relative to scale:
        \[  
\min_{\theta \neq \theta' \in \Lambda_{j,s}(I_j)} |\theta - \theta'| \geq 2^R \cdot \max_{J \in \mathcal{B}_s(I_j)} |J|^{-1}; \]
\item We may localize each $\{ \Omega_{J,\delta} : J \in \mathcal{B}_s(I_j) \}$ to a small neighborhood of $\Lambda_{j,s}(I_j)$
\[ \sup_{J \in \mathcal{B}_s(I_j)} \| \Omega_{J,\delta} \cdot \chi_{\Lambda_{j,s}(I_j),J}  \|_{A(\mathbb{T})} \ll t. \]
    \end{itemize}
\end{definition}

We isolate the main structural decomposition we will achieve in the following Lemma.

\begin{lemma}[Pruning a Tree]\label{l:branches}
    For each $U \gg \rho^{-2}$, for each $I_j \in T_j^{\text{max}}, \ j < V$, we may decompose
    \[ \mathcal{D}_j(I_j) = \bigcup_{s \leq U} \mathcal{B}_s(I_j) \cup \mathcal{B}_{\infty}(I_j) \]
where each $\mathcal{B}_s(I_j)$ is a (possibly empty) ``branch," and there are only a few scales represented in $\mathcal{B}_\infty(I_j)$, the ``boundary" branch:
\[ |\{ |I| : I \in \mathcal{B}_\infty(I_j)\}| \lesssim RU.\]
\end{lemma}
\begin{proof}
For each $\epsilon > 0$, let
\[ \text{Ent}(\Lambda_0,\epsilon) \]
denote the cardinality of the maximal $\epsilon$-separated subset of $\Lambda_0$, see \eqref{e:Ent}, where we abbreviate $\Lambda_0 := \Lambda_j(I_j)$. We select an increasing sequence of integers 
\[ m_0 < m_1 < \dots < m_{U'} \leq 0, \; \; \; U' \leq U\] 
so that $|I_j|^{-1} = 2^{m_0}$, and for each $s \geq 1$, $m_s$ is the maximal integer satisfying
\[ \text{Ent}(\Lambda_0,2^{m_s}) \geq \text{Ent}(\Lambda_0,2^{m_{s-1}}) - \delta^{-2}/U. \]
We now choose $\Lambda_1 \subset \Lambda_0$ a maximal $2^{m_1}$-separated subset, so that
\[ \Lambda_0 \subset \Lambda_1 + B(2^{m_1}),\]
and inductively select
\[ \Lambda_s \subset \Lambda_{s-1}\]
a maximal $2^{m_s}$-separated subset, so that
\[ \Lambda_{s-1} \subset \Lambda_s + B(2^{m_s}).\]

We collect
\[ \mathcal{B}_{\infty}(I_j) := \{ J \in \mathcal{D}_j(I_j) : |J|^{-1} = 2^{m_s \pm O(R)} \text{ for some } s \leq U \} \]
and
\[ \mathcal{B}_s(I_j) := \{ I \in \mathcal{D}_j(I_j) \smallsetminus \mathcal{B}_\infty(I_j) : 2^{m_{s-1}} \leq |I|^{-1} \leq 2^{m_s} \}. \]

By Lemma \ref{l:H12}, it remains to show that for all $J \in \mathcal{B}_s(I_j)$
\[ \| P_{J,\delta} \cdot \chi_{\Lambda_{j,s}(I_j);J} \|_{L^2(\mathbb{T})}^2 \ll t^4 \cdot |J|.
\]
Abbreviating 
\[ \Lambda_0 \supset \Lambda_{s-1} := \Lambda_{j,s-1}(I_j) \supset \Lambda_s := \Lambda_{j,s}(I_j), \]
we bound, for 
\begin{align*}
    \int_{\mathbb{T}} |\chi_{\Lambda_s;J}(\beta)|^2 \cdot |P_{J,\delta}(\beta)|^2 \ d\beta &\leq \int_{\big(\Lambda_0 + B(20/|J|) \big)^c} |P_{J,\delta}(\beta)|^2 \ d\beta \\
    & \qquad + \int_{ \big(\Lambda_0 + B(10/|J|) \big) \smallsetminus \big( \Lambda_s + B(20/|J|) \big)} |P_{J,\delta}(\beta)|^2 \ d\beta;
\end{align*}
since $J$ is localized with respect to $\Lambda_0$, we may bound
\[  \int_{\big(\Lambda_0 + B(10/|J|) \big)^c} |\Omega_{J,\delta}(\beta)|^2 \ d\beta  \leq \rho^2 \cdot |J| \ll t^4 \cdot |J|;\]
the second term is of a simpler nature: since $1/|J| \gg 2^{m_{s-1}}$, we may  contain
\begin{align*}
&\big(\Lambda_0 + B(10/|J|) \big) \smallsetminus \big( \Lambda_s + B(20/|J|) \big) \\
& \qquad \subset \big( \Lambda_{s-1} + B(2^{m_{s-1}}) + B(10/|J|) \big) \smallsetminus \big( \Lambda_s + B(20/|J|) \big) \\
& \qquad \qquad \subset \big( \Lambda_{s-1} + B(20/|J|) \big) \smallsetminus \big( \Lambda_s + B(20/|J|) \big) \\
& \qquad \qquad \qquad = \big( \Lambda_{s-1} \smallsetminus \Lambda_s \big) + B(20/|J|) , \end{align*}
which leads to the estimate
\begin{align}\label{e:aboutdelta}
    \int_{ \big(\Lambda_0 + B(10/|J|) \big) \smallsetminus \big( \Lambda_s + B(20/|J|) \big)} |P_{J,\delta}(\beta)|^2 \ d\beta &\lesssim (\delta |J|)^2 \cdot |\Lambda_{s-1} \smallsetminus \Lambda_s| \cdot 1/|J| \\
    & \qquad \leq U^{-1} \cdot |J| \ll t^4 |J| \notag;
\end{align}
we used the bound $|P_{J,\delta}(\beta)| \approx \delta |J|$ in the inequality \eqref{e:aboutdelta} above.
\end{proof}

Motivated by the tree/branch construction, we introduce the following classes of cut-off functions, which will be used to construct \emph{close} approximates to the bilinear averages
\[ \{ \frac{1}{N} \sum_{n \leq N} f(x-\lfloor \alpha n \rfloor ) g(a+n) : N \} \]
at each scale and location.
\begin{itemize}
    \item For $I \in \mathcal{B}_{s}(I_j)$ we let
    \begin{align}\label{e:Xi0} \Xi^0_I(\beta) := \sum_{\theta \in \Lambda_{j,s}(I_j)} \varphi(|I|(\beta - \theta))\end{align}
    where 
    \[ \mathbf{1}_{B(10)} \leq \varphi \leq \mathbf{1}_{B(20)} \] is smooth; note that since $I \notin \mathcal{B}_{\infty}(I_j)$, $|\Xi^0_I|, \ |\Xi^0_I(\alpha \cdot)| \leq 1$ as the bump functions are (very) disjointly supported; and
    \item For $I$ as above, we set
     \begin{align}\label{e:Xi} \Xi_I(\beta) := \sum_{\theta \in \Lambda_{j,s}(I_j)} \varphi(R |I|(\beta - \theta)),\end{align}
     cut-offs like $\{ \Xi_I^0 \}$ but much more Fourier localized. 
\end{itemize}

With these preliminaries in mind, we are prepared to turn to the proof proper.

\section{Oscillation and Maximal Inequalities: Setting the Stage}\label{s:oscmax}
Let $M_0 \ll M_1 \ll \dots \ll M_K$ be temporarily fixed. Let $I_M(a) := (a,a+M]$ so that, with $w \in \mathcal{C}_c^{\infty}([0,1])$ satisfying
\[ \mathbf{1}_{[1/L,1 -1/L]} \leq w \leq \mathbf{1}_{[0,1]} \]
as in Lemma \ref{l:w} with $L = \tau^{-O(1)}$, using the irrationality of $\alpha$ and equidistribution\footnote{The argument in the case where $\alpha \in \mathbb{Q}$ is simpler, and reduces to the case where $\alpha \in \mathbb{Z}$ by periodicity, at which point the floor function can be ignored}, we express
\begin{align}
    &\frac{1}{M} \sum_{m \leq M} f(x-\lfloor \alpha m \rfloor ) g(a+m) \\
    & \qquad = \frac{1}{M} \sum_{m \leq M} \sum_k w(\alpha m -k) f(x- k ) g(a+m) + O(\tau^A) \\
    & \qquad = \sum_{|\xi| \leq L/\tau} e(-\alpha x \xi) \cdot \int_{\mathbb{T}} \widehat{f_\xi}(\beta) \Big(\frac{1}{N} \sum_{n \leq N} g_\xi(a+n) e(-n \alpha \beta) \Big) e( \beta x ) \ d\beta + O(\tau^{A}) \\
    &  \qquad = \sum_{|\xi| \leq L/\tau} e(-\alpha x \xi) \cdot \int_{\mathbb{T}} \widehat{f_\xi}(\beta) \Big(\sum_{n} \phi_N(n) g_\xi(a+n) e(-n \alpha \beta) \Big) e( \beta x ) \ d\beta + O(\tau^{A}) \\
    & \qquad =\sum_{|\xi| \leq L/\tau} e(-\alpha x \xi)  \int_{\mathbb{T}} \widehat{f_\xi}(\beta) \Big(\sum_{n} \phi_N(n-a) g_\xi(n) e(-n \alpha \beta) \Big) e( \beta x + \alpha a \beta ) \ d\beta + O(\tau^A)
    \end{align} 
where $f_\xi,g_\xi$ are as in Lemma \ref{l:w}, and $\phi$ is as in \eqref{e:phitau}, with 
\[ \phi_N(t):= 1/N \cdot \phi(t/N).\] Since all of our arguments will allow us to concede polynomial factors in $\tau^{-O(1)}$, possibly after replacing $C_{\tau,B,H}$ with $C_{\tau^{B},2B,H}$ etc.\ we will focus on the case of a single $\xi$, and will accordingly suppress the subscripts. Further, we emphasize that every function $f$ introduced will have Fourier support inside $[0,1/D]$, see \eqref{e:Da}, so satisfies \eqref{e:psiD}. 

Thus, we are interested in understanding sets of the form
\begin{align*}
    &\{[H]: \max_{M_{k-1} \leq M \leq M_k} | \int \big(\Omega_{I_M(a)}(\alpha \beta) - \Omega_{I_{M_k}(a)}(\alpha \beta) \big) \hat{f}(\beta) e(\beta(x+\alpha a)) \ d\beta | \geq \tau \} \\
    & \qquad \subset 
\{[H]: \max_{M_{k-1} \leq M \leq M_k, \ (*)} | \int \big(\Omega_{I_M(a)}(\alpha \beta) - \Omega_{I_{M_K}(a)}(\alpha \beta) \big) \hat{f}(\beta) e( \beta (x+\alpha a)) \ d\beta | \geq \tau/2 \}
\end{align*}  
where $(*)$ denotes that we restrict only to times of the form
\[ M \in \{ 2^{b / B_\tau} : b \geq 1 \}, \; \; \; B_\tau = A_0 \cdot \tau^{-1}. \]
We further sparsify to $B_\tau \Delta = O(\tau^{-1} \Delta)$ many subsets of times of the form
\[ M \in \{ 2^{b/B_\tau + U + k \Delta} : k \} =: \mathcal{J}_{b,U} \]
so that if for each $1 \leq k \leq K$ with $K \lesssim_H 1$ maximal
\begin{align*}
    |\{[H]: \max_{M_{k-1} \leq M \leq M_k} | \int_{\mathbb{T}} \big(\Omega_{I_M(a)}(\alpha \beta)- \Omega_{I_{M_k}(a)}(\beta) \big) \hat{f}( \beta) e(\beta (x+\alpha a)) \ d\beta | \geq \tau \}| \geq \tau^B H 
    \end{align*}
then necessarily there exists some $\mathcal{J}_{b,U}$ so that for all $k \in O_{b,U,K} \subset [K]$, a subset with size $\geq \tau K/\Delta$,
\begin{align*}
    &|\{[H]: \max_{M_{k-1} \leq M \leq M_k, \ M \in \mathcal{J}_{b,U}} | \int_{\mathbb{T}} \big(\Omega_{I_M(a)}(\alpha \beta) - \Omega_{I_{M_k}(a)}(\alpha \beta) \big) \hat{f}(\beta) e(\beta (x+\alpha a)) \ d\beta | \geq \tau/2 \}| \\
    & \qquad \geq \tau^{B+1}/\Delta \cdot H;
    \end{align*}
before proceeding, we emphasize the reproducing identity
\begin{align}
&\int_{\mathbb{T}} \big(\Omega_{I_M(a)}(\alpha \beta) - \Omega_{I_{M_k}(a)}(\alpha \beta) \big) \hat{f}(\beta) e(\beta (x+\alpha a)) \ d\beta \\
& \qquad \equiv \int_{\mathbb{T}} \big(\Omega_{I_M(a)}(\alpha \beta) - \Omega_{I_{M_k}(a)}(\alpha \beta) \big) \widehat{\psi_D}(\beta) \cdot \hat{f}(\beta) e(\beta (x+  \alpha a )) \ d\beta
\end{align}
where $\psi_D$ is as in \eqref{e:psiD}; for notational ease, for intervals $I$ we set
\[ \Omega'_I(\alpha \beta) := \Omega_I(\alpha \beta) \widehat{\psi_D}(\beta) = \Omega_I(\alpha \beta) \hat{\psi}(D \beta). \]

In what follows, therefore, we will view $b,U$ as fixed, and will henceforth suppress the dependence on this set of times; for notational ease, we will restrict to $b, U =0$ as this will allow us to make use of the standard grid $\mathcal{D}$, but other choices of $b,U$ can be treated by working with other grids \[ \{ \mathcal{D}_U^{\Delta,L,b} : L\}.\] In particular, we will implicitly restrict all times $\{ M \} \subset \{ 2^{n \Delta} : n \geq 1 \}$

We now smoothly restrict to the level-sets of $\Omega'$, as per \eqref{e:level},
\[ \Omega_{I,\delta}'(\alpha \beta) := \Psi_\delta \big( \Omega_I'(\alpha \beta) \big) \equiv \Psi_\delta(\Omega_I(\alpha \beta)) \widehat{\psi_D}(\beta), \; \; \; \beta \in \text{supp } \hat{f} \]
and bound
\begin{align*}
    &|\{[H]: \max_{M_{k-1} \leq M \leq M_k} | \int  \big(\Omega_{I_M(a)}'(\alpha \beta) - \Omega_{I_{M_k}(a)}'(\alpha \beta) \big) \hat{f}(\beta) e(\beta (x+ \alpha a)) \ d\beta | \geq \tau/2 \}| \notag \\
    &  \leq \sum_{\log^{-1}(1/\delta) \geq \tau^4} |\{[H]: \max_{M_{k-1} \leq M \leq M_k} | \int \big(\Omega_{I_M(a),\delta}'(\alpha \beta) - \Omega_{I_{M_k}(a),\delta}'(\alpha \beta) \big) \hat{f}(\beta) e(\beta (x+ \alpha a)) \ d\beta | \gtrsim \tau^5 \}| \\
    & + \sum_{\log^{-1}(1/\delta) < \tau^4} |\{ [H] : \sup_M |\int \Omega_{I_M(a),\delta}'(\alpha \beta) \hat{f}(\beta) e( \beta (x+ \alpha a)) \ d\beta | \gtrsim \log^{-10}(1/\delta) \}|
\end{align*}

We will handle the contribution of each $\delta$-level set individually. Thus, we will prove the following propositions. Recall
\[
     2 \mathbb{N} + 1 \ni \Delta \approx t^{-10},
\]
see \eqref{e:Delta}.
\begin{proposition}\label{p:max00}[Maximal Estimates]
    The following estimates holds for $\log^{-1}(1/\delta) < \tau^4$ 
    and each $\mathcal{D}_U^{\Delta,L,b}$ (of which there are only $\log^{O(1)}(1/\delta)$ many choices):
\begin{align}
|\{  \overline{\mathcal{D}_U^{\Delta,L,b}} \cap [H] : \sup_M |\int \Omega_{I_M(x),\delta}'(\alpha \beta) \hat{f}(\beta) e( \beta (x +  \alpha x)) \ d\beta | \gtrsim \log^{-10}(1/\delta) \}| \lesssim \delta^{\alpha_0} \cdot H.
\end{align}
And, away from a set $E_\delta \subset \overline{\mathcal{D}_U^{\Delta,L,b}} \cap [N] $ of size $|E_\delta| = o_{N \to \infty}(\delta^{\alpha_0} N)$
\begin{align}
|\{ [H] : \sup_M |\int \Omega_{I_M(a),\delta}'(\alpha \beta) \hat{f}(\beta) e( \beta (x+ \alpha a)) \ d\beta | \gtrsim \log^{-10}(1/\delta) \}| \lesssim \delta^{\alpha_0} \cdot H.
\end{align}
\end{proposition}

To state our second proposition, designed for $\log^{-1}(1/\delta) \geq \tau^4$, we introduce the following notation

\begin{align*}
    C^{\Delta,L,U,b}_{\tau,\delta, B,H} &:= \sup \Big\{ K : \text{ there exist $1$-bounded } f,g , M_0 < M_1 < \dots < M_K \leq H/100, \text{ so that} \\
&  |\{ \overline{\mathcal{D}^{\Delta,L,b}_U} \cap [H]: \max_{M_{k-1} \leq M \leq M_k} | \int \big(\Omega_{I_M(x),\delta}'(\alpha \beta) - \Omega_{I_{M_k}(x),\delta}'(\alpha \beta) \big) \hat{f}(\beta) e(\beta x + \alpha \beta x) \ d\beta | \gtrsim \tau^5 \}| \\
& \qquad  \geq \tau^B H \Big\}
\end{align*} 
and for a $1$-bounded $g$ -- which we regard as fixed --
\begin{align*}
    C_{\tau,\delta,B,H}(a) &:= \sup \Big\{ K : \text{ there exist  $1$-bounded } f, \ M_0 < M_1 < \dots < M_K \leq H/100  \text{ so that }\\
&   |\{[H]: \max_{M_{k-1} \leq M \leq M_k} | \int \big(\Omega_{I_M(a),\delta}'(\alpha \beta) - \Omega_{I_{M_k}(a),\delta}'(\alpha \beta) \big) \hat{f}(\beta) e(\beta x+\alpha a \beta ) \ d\beta | \gtrsim \tau^5 \}| \\
& \qquad \geq \tau^B H \Big\}.
\end{align*}

\begin{proposition}\label{p:osc00}[Oscillation Estimates]
    The following estimates holds for $\log^{-1}(1/\delta) \geq \tau^4$ and each $\mathcal{D}_U^{\Delta,L,b}$
\begin{align}\label{e:oscDR}
C^{\Delta,L,U,b}_{\tau,\delta, B,H} \lesssim 2^{O_B(\tau^{-A_0})}
\end{align}
and away from a set $E_\tau \subset \overline{\mathcal{D}_U^{\Delta,L,b}} \cap [N]$ of size $|E_\tau| = o_{N \to \infty}(2^{-\tau^{-O(A_0)}} N)$
\begin{align}\label{e:oscRT}
C_{\tau,\delta,B,H}(a) \lesssim 2^{O_B(\tau^{-A_0})}.
\end{align}
\end{proposition}

In what follows, we will restrict to the case $b,L,U = 0$, as the other cases involve only notational changes.

\subsection{Reformulation}
To clean things up, we use the notation
\begin{align}\label{e:Q}
\mathcal{Q}_k := \{ I \in \mathcal{D} : M_{k-1} \leq |I| \leq M_k \} \end{align}
and
\[ \mathcal{F}^1(J;x,a) := \int \Omega_J'(\alpha \beta) \hat{f}(\beta) e(\beta(x+\alpha a)) \ d\beta \]
with
\[ \mathcal{F}^1(J;x) := \mathcal{F}^1(J;x,x),\]
and 
\[ \mathcal{F}^1_\delta(J;x,a) := \int \Omega_{J,\delta}'(\alpha \beta) \hat{f}(\beta) e(\beta(x+ \alpha a)) \ d\beta \]
with
\[ \mathcal{F}^1_\delta(J;x) := \mathcal{F}_\delta^1(J;x,x);\]
we will also need to keep track of differences of the $\{ \mathcal{F}_\delta^1 \}$; for $J \in \mathcal{Q}_k$ we define
\[ \widehat{\mathcal{F}}^1_\delta(J;x,a) := \mathcal{F}^1_\delta(J;x,a) - \mathcal{F}^1_\delta(\hat{J};x,a),\]
where $\hat{J} \supset J$ is the unique interval of length $M_k$.

With this language in mind, focusing on the representative (sparse) grid $\mathcal{D} := \mathcal{D}_0^{\Delta,0,0}$, by Lemma \ref{l:cont},
it suffices to prove the following.
\begin{proposition}\label{p:max0}[Maximal Estimates]
    The following estimates holds for $\log^{-1}(1/\delta) < \tau^4$:
\begin{align}\label{e:maxDR}
|\{ \overline{\mathcal{D}} \cap [H] : \sup_{x \in J \in \mathcal{D}([H])} |\mathcal{F}^1(J;x) | \gtrsim \log^{-10}(1/\delta) \}| \lesssim \delta^{\alpha_0} \cdot H
\end{align}
and, away from a set $E_\delta \subset \overline{\mathcal{D}} \cap [N] $ of size $|E_\delta| = o_{N \to \infty}(\delta^{\alpha_0} N)$
\begin{align}\label{e:maxRT}
|\{ [H] : \sup_{a \in J \in \mathcal{D}([N])} |\mathcal{F}^1(J;x,a) | \gtrsim \log^{-10}(1/\delta) \}| \lesssim \delta^{\alpha_0} \cdot H.
\end{align}
\end{proposition}

And, we may reformulate

\begin{align*}
    C'_{\tau,\delta, B,H} := \sup \Big\{ K &: \text{ there exists $1$-bounded } f,g , M_0 < M_1 < \dots < M_K \leq H/100, \text{ so that} \\
&  |\{ \overline{\mathcal{D}} \cap [H]: \max_{x \in J \in \mathcal{Q}_k \cap \mathcal{D}([H])} | \hat{\mathcal{F}}_\delta^1(J;x) | \gtrsim \tau^5 \}| \geq \tau^B H \Big\}
\end{align*} 
and for a $1$-bounded $g$ -- which we regard as fixed --
\begin{align*}
    C_{\tau,\delta,B,H}(a) := \sup \Big\{ K &: \text{ there exists $1$-bounded  } f, \ M_0 < M_1 < \dots < M_K \leq H/100  \text{ so that }\\
&   |\{[H]: \max_{a \in J \in \cap \mathcal{D}([N])} | \hat{F}_\delta^1(J;x,a) | \gtrsim \tau^5 \}| \geq \tau^B H \Big\}.
\end{align*}

Our strategy will be to carefully -- and generically -- replace
\[ \{ \mathcal{F}_\delta^1(J;x,a) \} \]
by two successive approximates, using our tree/branch foliation of 
\[ \mathcal{D}(I_0), \; \; \; I_0 = [H], [N], \]
with $U = W$ many branches, see \eqref{e:RW}.

Specifically, for 
\[ J \in \mathcal{B}_{s}(I_j), \ s \leq W, \ I_j \in T_j^{\text{max}}\]
we set
\[ \mathcal{F}_\delta^2(J;x,a) := \int \Omega_{J,\delta}'(\alpha \beta) \Xi^0_J(\alpha \beta) \hat{f}(\beta) e(\beta(x+\alpha a)) \ d\beta, \]
see \eqref{e:Xi0},
and
\[ \mathcal{F}_\delta^3(J;x,a) := \int \Omega_{J,\delta}'(\alpha \beta) \Xi_J(\alpha \beta) \hat{f}(\beta) e(\beta(x+\alpha a)) \ d\beta, \]
see \eqref{e:Xi}, with the obvious modifications for $\mathcal{F}_\delta^j(J;x)$; $\Xi_J^0, \ \Xi_J$ are as in Section \ref{ss:fol} above. We maintain the same notation for differences of $\{ \mathcal{F}_\delta^j \}$, thus for $J \in \mathcal{Q}_i$ we define
\[ \widehat{\mathcal{F}}^j_\delta(J;x,a) := \mathcal{F}^j_\delta(J;x,a) - \mathcal{F}^j_\delta(\hat{J};x,a),\]
and similarly 
\[ \widehat{\mathcal{F}}^j_\delta(J;x) := \widehat{\mathcal{F}}^j_\delta({J};x,x).\]

We will be concerned with the case where $a \in J$ (whether or not we specialize to $a=x$ as per the Double Recurrence formulation). 

We begin by recording the following lemma, which allows us to control the contribution from each individual
\[ \{ \mathcal{F}_\delta^j : 1 \leq j \leq 3\}. \]

\begin{lemma}\label{l:1scale}[Single Scale Estimate]
Under the hypotheses and notation of Lemma \ref{l:cont}, the following bounds hold for $N \in \{ 2^{k \Delta} : k \geq 1\}$, whenever 
\[ J_N(x) = \begin{cases} (x,x+N] & \text{ or } \\
J & \text{ where } x \in J \in \mathcal{D} \text{ has } |J| = N, \end{cases}\]
then
\begin{align}
    &|\{ x \in \overline{\mathcal{D}} \cap I: |\int \Omega_{J_N(x),\delta}'(\alpha \beta) \eta( \alpha \beta) \cdot \widehat{f}(\beta) e(\beta (x + \alpha a )) \ d\beta| \gg t \}| \lesssim t^{-10} U^4 \delta^2 \cdot |I|;
\end{align}
and
\begin{align}
&|\{ x \in \overline{\mathcal{D}} \cap I: |\int \Omega_{J_N(x),\delta}'(\alpha \beta) \eta( \alpha \beta) \cdot \widehat{f}(\beta) e(\beta (x + \alpha a)) \ d\beta| \gg t \}|  \\
& \qquad \lesssim t^{-6} \delta^2 U^2 \cdot \| \eta(\alpha \cdot) \widehat{f_{I'}} \|_{L^2(\mathbb{T})}^2, \; \; \;    I' := D^2 t^{-4} U^2 I,
\end{align}
see \eqref{e:Da}.
\end{lemma}
\begin{proof}
 First, by Lemma \ref{l:mult1}, there is no loss of generality in restricting 
\[ f \longrightarrow f_{I'}, \; \; \; I' := D^2 t^{-4} U^2 I\]
at the cost of adjusting the implicit constants in the super-level sets; more generally, for each dyadic interval, $P$, set $P' := D^2 t^{-4} U^2 P$. Moreover, by Lemma \ref{l:cont}, after a further adjustment of implicit constants, 
there is no loss of generality in replacing
\[ \Omega_{J_N(x),\delta}' \longrightarrow \Omega_{J(x),\delta}'\]
where $x \in J(x) \in \mathcal{D}$ with $|J| = N$ (i.e.\ we may reduce to the second alternative in the definition of $J_N(x)$).

Now, the key observation is that whenever $a \in J$, $\Omega_{J(a),\delta} \equiv \Omega_{J,\delta}$ is independent of $a$ -- and just becomes a Fourier multiplier.
    
So, we may bound
\begin{align*}
    |\{ x \in \overline{\mathcal{D}} \cap J: \int \Omega_{J,\delta}'(\alpha \beta) \eta(\alpha \beta) \widehat{f_{J'}}(\beta) e(\beta (x + \alpha a )) \ d\beta \gg t \}| \lesssim t^{-2} \delta^2 \| \widehat{f_{J'}} \|_{L^2(\mathbb{T})}^2
\end{align*}
so that
\begin{align*}
 &   |\{ x \in \overline{\mathcal{D}} \cap I : \int \Omega_{J_N(x),\delta}'(\alpha \beta) \eta(\alpha \beta) \hat{f}(\beta) e(\beta (x + \alpha a )) \ d\beta \gg t \}| \\
 & \qquad \lesssim t^{-2} \delta^2 \sum_{J \subset I, \ |J| = N} \| \widehat{f_{J'}} \|_{L^2(\mathbb{T})}^2 \\
    & \qquad \qquad \lesssim t^{-6} U^2 \delta^2 \cdot |I'| \lesssim t^{-10} U^4 \delta^2 \cdot  |I|.
\end{align*}
For the second point, after relabeling
\[ \eta \longrightarrow \eta(\alpha \cdot) \]
(recalling that $\alpha = O(1)$ is fixed), we bound
\begin{align*}
&|\{ x \in \overline{ \mathcal{D} } \cap I : \int \Omega_{J_N(x),\delta}(\alpha \beta) \eta(\beta) \hat{f}(\beta) e(2 \beta x) \ d\beta \geq t \}| \\
    & \qquad \leq \sum_{J \subset I, \ |J| =N} | \{ x \in \overline{\mathcal{D}} \cap J : \int \Omega_{J,\delta}(\alpha \beta) \widehat{ (\eta^{\vee} * f)_{J'}}(\beta) e(2 \beta x)  \ d\beta \geq t/2 \} | \\
    & \qquad \qquad \lesssim  t^{-2} \delta^2 \sum_{J \subset I, \ |J| =N} \| \eta^{\vee} *f \|_{\ell^2(J')}^2 \\
    & \qquad \qquad \qquad \lesssim t^{-6} \delta^2 U^2  \cdot \| \eta \hat{f} \|_{L^2(\mathbb{T})}^2 \equiv t^{-6} \delta^2 U^2  \cdot \| \eta \widehat{f_{I'}} \|_{L^2(\mathbb{T})}^2
\end{align*}
by a final application of Plancherel.
\end{proof}

\begin{corollary}\label{c:bdrybranch}
    Suppose that $\{ \mathcal{D}_j(I_0) : j < V \} \cup \mathcal{D}_V(I_0)$ is a $V$-forest, and that $\mathcal{B}_\infty(I_j), \ I_j \in T_j^{\text{max}}$ is a boundary branch as in Lemma \ref{l:branches}.
    
    Then for any $j < V$, 
\begin{align}
&\{ I_j \smallsetminus X_{j+1} : \sup_{x \in J \in \mathcal{B}_s(I_j)} |\mathcal{F}_\delta^1(J;x)| \gg t \} \subset \{ I_j \smallsetminus X_{j+1} : \sup_{x \in J \in \mathcal{B}_s(I_j)} |\mathcal{F}_\delta^2(J;x)| \gg t \} 
   \end{align}
and
    \[ |\{ I_j \smallsetminus X_{j+1} : \sup_{x \in J \in \mathcal{B}_\infty(I_j)} |\mathcal{F}_\delta^2(J;x)| \gg t \}| \lesssim RU t^{-10} \delta^2  \cdot |I_j|. \]
\end{corollary}

We also record the following Corollary, which will give us the flexibility to cheaply pass from 
\[ \mathcal{F}_\delta^2(J;x) \longrightarrow \mathcal{F}_\delta^3(J;x). \]

Below, we set 
\[ \chi_t(x) := x \cdot \chi(A_0 t^{-1} |x|) \cdot \varphi(x)\]
where
\[ \mathbf{1}_{|x| \geq 1/2} \leq \chi \leq \mathbf{1}_{|x| \geq 1/4}, \; \; \;  \mathbf{1}_{|x| \leq 2} \leq \varphi \leq \mathbf{1}_{|x| \leq 4} \]
are smooth, so that 
\[ \sup_{t > 0} \| \chi_t \|_{\mathcal{C}^1(\mathbb{R})} \lesssim 1.\]

\begin{corollary}\label{c:F2F3}
    For any $I_j \in T_j^{\text{max}} \, \cap \, \mathcal{D}(I_0)$, and any $\mathcal{B}_{s}(I_j), \ j <V , s \leq U$ with $|I_j| \geq 10 \cdot \max_{J \in \mathcal{B}_{s}(I_j)} |J|$, 
\begin{align*}
    \sum_{J \in \mathcal{B}_s(I_j)} \|\chi_t\Big( \mathcal{F}_\delta^2(J;x) - \mathcal{F}_\delta^3(J;x) \Big) \|_{\ell^2(J)}^2 \lesssim R^3 \delta^2  \cdot |I_j|
\end{align*}
and in fact
    \begin{align*}
        |\{x \in \overline{\mathcal{D}} \cap I_j : \sup_{x \in J \in \mathcal{B}_{s}(I_j)} |\mathcal{F}_\delta^2(J;x) - \mathcal{F}_\delta^3(J;x)| \geq t \}| \lesssim   R^3 \delta^2  \cdot |I_j|. 
    \end{align*}
Moreover
\begin{align*}
    \sum_{ a\in J \in \mathcal{B}_s(I_j)} \| \mathcal{F}_\delta^2(J;x,a) - \mathcal{F}_\delta^3(J;x,a)\|_{\ell^2_x(\mathbb{Z})}^2 \lesssim \delta^2 \cdot \log R \cdot H.
\end{align*} 
\end{corollary}
\begin{proof}
We will apply the third point in Lemma \ref{l:1scale}, with
\[ \eta \in \{ \Xi_J^0(\alpha \cdot) - \Xi_J(\alpha \cdot) \}; \]
the key point is the since each $J \in \mathcal{B}_{s}(I_j)$, the distinguished frequencies, $\Lambda_{j,s}(I_j)$, in the definition of 
\[ \{ \Xi_J^0, \ \Xi_J :J \in \mathcal{B}_{s}(I_j) \} \]
do not change, and thus we may express
\[ \Xi_J^0 \equiv \Xi_{[N]}^0, \; \; \; \Xi_J \equiv \Xi_{[N]} \; \; \; |J| = N.\]

Since
\[ \sup_{x \in J \in \mathcal{B}_s(I_j),\  j =2,3} | \mathcal{F}_\delta^j(J;x) | \ll t \]
when $f$ is supported off $D^2 R^2 t^{-4} I_j$, see Lemma \ref{l:mult1}, in what follows we will implicitly restrict 
\[ f \longrightarrow f \cdot \mathbf{1}_{D^2 R^2 t^{-4} I_j}; \]
with this in mind, for notational ease, define 
\[ f_J := \Big( \hat{f} \cdot \big(\Xi_{[N]}^0(\alpha \cdot) - \Xi_{[N]}(\alpha \cdot)\big) \Big)^{\vee} \cdot \mathbf{1}_{D^2 R^2 t^{-4} J} \]
so that
\[ \| \big( \sum_N |f_J|^2 \big)^{1/2} \|_{\ell^2(\mathbb{Z})}^2 \lesssim R^2 \log R \cdot t^{-8} \cdot |I_j|.\]

But now, for $x \in J$
\begin{align*}
    |\chi_t\big( \mathcal{F}_\delta^2(J;x) - \mathcal{F}_\delta^3(J;x) \big)|& \leq |\chi_{t/2}\Big( \int \widehat{f_ J}( \beta) \cdot \Omega_{J,\delta}'(\alpha \beta) e( \beta (x + \alpha x)) \ d\beta \Big)| \\
    & \qquad \qquad \leq | \int \widehat{f_ J}(\beta) \cdot \Omega_{J,\delta}'(\alpha \beta) e( \beta (x + \alpha x)) \ d\beta |
\end{align*}
and so
\begin{align*}
    \sum_{J \in \mathcal{B}_s(I_j)} \| \chi_t\big( \mathcal{F}_\delta^2(J;x) - \mathcal{F}_\delta^3(J;x) \big) \|_{\ell^2(J)}^2 \lesssim \delta^2 \cdot  \sum_{J \in \mathcal{B}_s(I_j)} \| f_J  \|_{\ell^2}^2  \lesssim \delta^2 \cdot R^2 \log R \cdot  t^{-8} |I_j|,
\end{align*}   
as desired.

The final point is simpler: again using that $\Lambda_{j,s}(I_j)$ is fixed, we now set
\[ f_J := \Big( \hat{f} \cdot \big(\Xi_{[N]}^0(\alpha \cdot) - \Xi_{[N]}(\alpha \cdot) \big) \Big)^{\vee}, \; \; \; |J| = N\]
noting that there exists a unique $J$ containing $a$ with $|J| = N$; we then bound
\[ |\mathcal{F}_\delta^2(J;x,a) - \mathcal{F}_\delta^3(J;x,a)| \leq |(\Omega_{J,\delta}'(\alpha \cdot ) \widehat{f_J})^{\vee}(x)| \]
and square sum, using Plancherel and the orthogonality condition
\[ \sum_N |\Xi_{[N]}^0 - \Xi_{[N]}|^2 \lesssim \log R.\]
\end{proof}

With these preliminaries in mind, we are prepared to turn to the main argument.

\section{Proof of the Maximal Estimate, Double Recurrence}\label{s:DRmax}
In this section we establish \eqref{e:maxDR}, the first statement in Proposition \ref{p:max0}. With $I_0 = [H]$, we extract a $V = \delta^{-2 \alpha_0}$-forest, with $U=W$ branches, see \eqref{e:RW}, to bound
\begin{align*}
    &|\{ \overline{\mathcal{D}} \cap I_0 : \sup_{J \in \mathcal{D}} |\mathcal{F}_\delta^1(J;x)| \gg t \}| \\
& \qquad \leq \sum_{v <V} |\{ \overline{\mathcal{D}} \cap (X_v \smallsetminus X_{v+1}) : \sup_{J \in \mathcal{D}} |\mathcal{F}_\delta^1(J;x)| \gg t \}| + O(\delta^{2\alpha_0 - o(1)} |I_0|) \\
& \qquad \qquad \leq \sum_{v < V} \sum_{I_v \in T_v^{\text{max}}} |\{ \overline{\mathcal{D}} \cap (I_v \smallsetminus X_{v+1}) : \sup_{J \in \mathcal{D}_v(I_v)} |\mathcal{F}_\delta^2(J;x)| \gg t \}| + O(\delta^{\alpha_0} |I_0|).
\end{align*}
We drop the error term as acceptably small, and focus on each individual $V$, where we will prove that for each $v < V$, there exists an absolute constant,
\begin{align} 0 < c < 1
\end{align}
independent of $A_0$, so that
\begin{align*}
    \sum_{I_v \in T_v^{\text{max}}} |\{ \overline{\mathcal{D}} \cap (I_v \smallsetminus X_{v+1}) : \sup_{J \in \mathcal{D}_v(I_v)} |\mathcal{F}_\delta^2(J;x)| \gg t \}| \lesssim \delta^{c} \cdot |I_0|;
\end{align*}
below, we may allow $c$ to vary from line to line, but every such constant introduced will be bounded below by $\frac{1}{100}$ (say).

To do so, we pass to $W$-branches, and bound, for each $I_v \in T_v^{\text{max}}$
\begin{align*}
&    |\{ \overline{\mathcal{D}} \cap (I_v \smallsetminus X_{v+1}) : \sup_{J \in \mathcal{D}_v(I_v)} |\mathcal{F}_\delta^2(J;x)| \gg t \}| \\
& \qquad \leq \sum_{s \leq W} |\{ \overline{\mathcal{D}} \cap (I_v \smallsetminus X_{v+1}) : \sup_{J \in \mathcal{B}_s(I_v)} |\mathcal{F}_\delta^2(J;x)| \gg t \}| + O(RW t^{-10} \delta^2 |I_v|) \\
& \qquad \qquad \leq \sum_{s \leq W} |\{ \overline{\mathcal{D}} \cap (I_v \smallsetminus X_{v+1}) : \sup_{J \in \mathcal{B}_s(I_v)} |\mathcal{F}_\delta^3(J;x)| \gg t \}| + O(RW t^{-10} \delta^2 |I_v| ) + O(R^3 W \delta^2 |I_v|)
\end{align*}
by Corollaries \ref{c:bdrybranch} and \ref{c:F2F3} (in turn). Since the error term sums over $I_v \in T_v^{\text{max}}, \ v <V$ to (say)
\[ O(\delta |I_0|) \]
we discard this contribution, and focus on proving that for each $s \leq W$, 
\begin{align}\label{e:maxgoalDR} \sum_{I_v \in T_v^{\text{max}}} |\{ \overline{\mathcal{D}} \cap (I_v \smallsetminus X_{v+1}) : \sup_{J \in \mathcal{B}_s(I_v)} |\mathcal{F}_\delta^3(J;x)| \gg t \}| \lesssim \delta^{c} \cdot |I_0|;
\end{align}
this will be our focus for the remainder of this section.

\subsection{Establishing \eqref{e:maxgoalDR}}\label{ss:replace}
The first remark is that on the support of each ``bubble" in 
\[ \Xi_{J}(\beta) = \sum_{\theta \in \Lambda_{v,s}(I_v)} \varphi(R |J|(\beta - \theta)), \]
$|\beta - \theta|$ is so small relative to scale that we may replace
\[ \Omega_{J,\delta}(\beta) e(\beta x) = \Psi_\delta\big( \Omega_J(\beta) e(\beta x)\big) \]
with
\[ \Omega_{J,\delta}(\theta) e(\theta x) = \Psi_\delta\big( \Omega_J(\theta) e(\theta x)\big) \]
on the support of $\varphi(R |J| (\cdot - \theta))$; below, we implicitly restrict to $x \in J$.

Indeed, it suffices to show that, with 
\[ \Lambda := \alpha^{-1} \Lambda_{v,s}(I_v) \subset [0,1/10],\]
we have the bound
\[ \big( \sum_{\theta \in \Lambda} \int_{|\beta - \theta| \leq R^{-1} |J|^{-1}} |\Psi_\delta(\Omega_J'(\alpha \beta) e(\alpha \beta x) ) - \Psi_\delta(\Omega_J'(\alpha \theta) e(\alpha \theta x) )|^2 \ d\beta \big)^{1/2} \ll t^2 R^{-1} |J|^{-1/2} \]

Since $\| \Psi_\delta \|_{\mathcal{C}^1} \lesssim 1 $ is uniformly in $\delta > 0$, we reduce to showing 
\[ \big( \sum_{\theta \in \Lambda_{v,s}(I_v)} \int_{|\beta - \theta| \leq R^{-1} |J|^{-1}} |\Omega_J(\beta) e(\beta x)  - \Omega_J(\theta) e(\theta x) )|^2 \ d\beta \big)^{1/2} \ll t^2 R^{-1} |J|^{-1/2}; \]
if we consolidate
\[ G_J(t) := \frac{1}{|J|} \sum_{m \in J} g(m) e(-(m-x)t), \]
a trigonometric polynomial of degree $O(|J|)$,
then we are interested in showing that 
\begin{align*}
    \sum_{\theta \in \Lambda_{v,s}(I_v)} \int_{|\beta - \theta| \leq R^{-1} |J|^{-1}} |G_J(\beta) - G_J(\theta)|^2 \ d\beta \ll t^4 R^{-2} |J|^{-1}.
\end{align*}
But, if $V_J$ is a degree $O(|J|)$ trigonometric polynomial, see \eqref{e:rep}, with 
\begin{align}\label{e:VJ} \mathbf{1}_{5|J|} \leq V_J^{\vee} \leq \mathbf{1}_{10|J|},\end{align}
then we can split
\begin{align*}
    |G_J(\beta) - G_J(\theta)| &\lesssim \| (\partial V_J) * G_J(t) \|_{L^{\infty}(\theta + B(1/R|J|))} \cdot \frac{1}{R|J|} \\
    & \qquad \leq \| (\partial V_J) * (G_J \cdot \mathbf{1}_{\theta + B(R^{1/2}/|J|)}) \|_{L^{\infty}(\theta + B(1/R|J|))} \cdot \frac{1}{R|J|} \\
    & \qquad \qquad + \| (\partial V_J) * (G_J \cdot \mathbf{1}_{(\theta + B(R^{1/2}/|J|) )^c}) \|_{L^{\infty}(\theta + B(1/R|J|))}  \cdot \frac{1}{R|J|};
\end{align*}
we will prove
\begin{align*}
    |G_J(\beta) - G_J(\theta)| \lesssim_A \frac{|J|^{1/2}}{R}  \cdot \| G_J \cdot \mathbf{1}_{\theta + B(R^{1/2}/|J|)} \|_{L^2(\mathbb{T})} + R^{-A}.
\end{align*}
The local contribution is the main term, which we estimate
\begin{align*}
   & \| (\partial V_J) * (G_J \cdot \mathbf{1}_{\theta + B(R^{1/2}/|J|)}) \|_{L^{\infty}(\theta + B(1/R|J|))}  \cdot \frac{1}{R|J|} \\
& \qquad \lesssim  \| \partial V_J \|_{L^2(\mathbb{T})} \cdot \| G_J \cdot \mathbf{1}_{\theta + B(R^{1/2}/|J|)} \|_{L^2(\mathbb{T})}  \cdot \frac{1}{R|J|} \\
    & \qquad \qquad \qquad \lesssim \frac{|J|^{1/2}}{R}  \cdot \| G_J \cdot \mathbf{1}_{\theta + B(R^{1/2}/|J|)} \|_{L^2(\mathbb{T})}.
\end{align*}
For the global contribution, by \eqref{e:VJ}, we may bound
\[ \| \partial_t^{\gamma} V_J^{\vee}(t) \|_{L^{1}(\mathbb{R})} \lesssim_\gamma |J|^{1-\gamma} \]
for sufficiently many $\gamma \geq 0$, so that 
\[ |V_J(\beta)| \lesssim_A |J| \cdot ( 1 + |J| \cdot  |\beta| )^{-A} \]
by the Riemann Lebesgue Lemma, and thus whenever $\beta \in \theta + B(1/R|J|)$, we may bound
\begin{align*}
    |\int \partial V_J(\beta - t) \cdot (G_J \cdot \mathbf{1}_{(\theta + B(R^{1/2}/|J|) )^c})(t) \ dt| &\lesssim \int_{(\theta + B(R^{1/2}/|J|) )^c} |\partial V_J(\beta - t)| \ dt  \\
    & \qquad \lesssim_A |J| \cdot R^{-A},
\end{align*}
so that
\[ \| (\partial V_J) * (G_J \cdot \mathbf{1}_{(\theta + B(R^{1/2}/|J|) )^c}) \|_{L^{\infty}(\theta + B(1/R|J|))} \cdot \frac{1}{R|J|} \lesssim_A R^{-A}.\]
The total contribution, is therefore
\begin{align*}
    &\sum_{\theta \in \Lambda_{v,s}(I_v)} \int_{|\beta - \theta| \leq R^{-1} |J|^{-1}} |G_J(\beta) - G_J(\theta)|^2 \ d\beta \\
    & \qquad \lesssim \frac{|J|}{R^2}  \cdot  \sum_{\theta \in \Lambda_{v,s}(I_v)} \int_{|\beta - \theta|\leq R^{-1} |J|^{-1}} \| G_J \|_{L^2(\theta + B(R^{1/2}/|J|))}^2 \ d\beta  + \sum_{\theta \in\Lambda_{v,s}(I_v)} \int_{|\beta - \theta| \leq R^{-1} |J|^{-1}} O_A(R^{-A}) \ d\beta\\
    & \qquad \lesssim \frac{1}{R^3} \cdot \sum_{\theta \in \Lambda} \| G_J \|_{L^2(\theta + B(R^{1/2}/|J|))}^2 +O_A(R^{-A}  \cdot |J|^{-1})  \\
    & \qquad \qquad \lesssim R^{-3} \cdot \| G_J \|_{L^2(\mathbb{T})}^2 +O_A(R^{-A} \cdot |J|^{-1}) \\
    & \qquad \qquad \qquad \lesssim R^{-3} \cdot |J|^{-1}
\end{align*}
since 
\[ \min_{\theta \neq \theta' \in \Lambda} |\theta - \theta'| \gg 2^R  \cdot |J|^{-1} \]
for $J \in \mathcal{B}_s(I_v)$.

Thus, we have reduced \eqref{e:maxgoalDR} to establishing the following Proposition.
\begin{proposition}
    For any $s \leq W$, with $\Lambda := \alpha^{-1} \Lambda_{v,s}(I_v)$,
    \begin{align*}
        \sum_{I_v \in T_v^{\text{max}}} |\{ I_v : \sup_{\mathcal{B}_s(I_v)} |\sum_{\Lambda} e(\theta x + \theta \alpha x) \rho_{J}*f_\theta(x) \cdot \Omega_{J,\delta}(\theta) | \gg t \}| \lesssim \delta^{c} \cdot H
    \end{align*}
    where
    \[ \widehat{f_\theta}(\beta) := \phi(2^{-m_s-R} \beta) \cdot \hat{f}(\beta + \alpha \theta) \]
    with $\rho_J$ defined via the Fourier transform
    \[ \widehat{ \rho_{J} }(\beta) := \varphi(\alpha R|J|\beta).\]
\end{proposition}

\begin{proof}
   We work locally, on intervals of length $2^{-m_s}$: indeed, whenever $x,x' \in D P$, an interval with $|P| = 2^{-m_s}$, 
\[ |\rho_Q*f_\theta(x) - \rho_Q*f_\theta(x')| \lesssim \frac{|P|}{|Q|R} \leq 2^{-R} \cdot (\frac{|P|}{|Q|})^{1/2} 
\]
so that 
\begin{align*}
&    \sum_{\theta \in \Lambda} e(\theta x + \theta \alpha x) \rho_Q*f_\theta(x) \cdot \Psi_\delta( \frac{1}{|Q|} \sum_{m \in Q} g(m) e(- m \alpha \theta)) \\
    & \qquad = 
\sum_{\theta \in \Lambda} e(\theta x + \theta \alpha x) \rho_Q*f_\theta(x_P) \cdot \Psi_\delta( \frac{1}{|Q|} \sum_{m \in Q} g(m) e(- m \alpha \theta)) + O(2^{-R} (\frac{|P|}{|Q|})^{1/2})
\end{align*}
for any $x_P \in D P$ that we wish; in particular, if, for any $\{x_P\}$, we expand
\begin{align*}
    &|\{ (I_j \smallsetminus X_{j+1} ) \cap \overline{\mathcal{D}} : \max_{x \in Q \in \mathcal{D}}   | \sum_{\theta \in \Lambda} \int \varphi(\alpha R|Q|(\beta - \theta)) \Omega_{Q,\delta}'(\alpha \theta) e(\theta \alpha x) \hat{f}(\beta) e(x  \beta) \ d\beta | \gg t \}| \\
    &  \leq \sum_{|P| = 2^{-m_s}, \ P \subset I_j} | \{ P \cap \overline{\mathcal{D}} : \max_{x _P \in Q \in \mathcal{D}}   | \sum_{\theta \in \Lambda} \int \varphi(R|Q| \alpha (\beta - \theta)) \Omega_{Q,\delta}'(\alpha \theta) e(\theta \alpha x) \hat{f}(\beta) e( x_P \beta) \ d\beta | \gg t \}|
\end{align*}
then for any choices of points $x_P \in D P$, we may bound the above
\begin{align*}
    \sum_{|P| = 2^{-m_s}, \ P \subset I_j} |\{ P \cap \overline{\mathcal{D}} : \max_{P \subset Q \in \mathcal{D}}   | \sum_{\theta \in \Lambda} \rho_Q*f_\theta(x_P) \Omega_{Q,\delta}(\alpha \theta) e(\theta x_P + \theta \alpha x) | \gg t \}|
\end{align*}
We now turn to $\ell^2$-methods to estimate, for each $P$
\begin{align*}
   & |\{ P \cap \overline{\mathcal{D}} : \max_{P \subset Q \in \mathcal{D}}   | \sum_{\theta \in \Lambda} \rho_Q*f_\theta(x_P) \Omega_{Q,\delta}'(\alpha \theta) e(\theta x_P + \theta \alpha x) | \gtrsim t \}| \\
   & \qquad \lesssim t^{-2} \cdot \| \eta_P(x) \max_{P \subset Q \in \mathcal{D}}    | \sum_{\theta \in \Lambda} \rho_Q*f_\theta(x_P) \Psi_\delta\big( \Omega_{Q}'(\alpha \theta) \big) e(\theta x_P + \theta \alpha x)\|_{\ell^2(\mathbb{Z})}^2,
\end{align*}
where 
\begin{align}\label{e:etaP} \mathbf{1}_{DP} \leq \eta_P \leq \mathbf{1}_{3DP} \end{align}
is smooth, with $|\partial^j \eta_P| \lesssim |P|^{-j}$ for sufficiently many $j$.

We rephrase the foregoing in the language of entropy:

For each $Q$, we can express
\begin{align*}
    \sum_{\theta \in \Lambda} \rho_Q*f_\theta(x_P) \Psi_\delta\big( \Omega_{Q}(\alpha \theta)  \big) e(\theta x_P + \alpha \theta x) =: \sum_{\theta \in \Lambda} a_{Q;x_P}(\theta) b_{Q}(\theta) \psi(\delta^{-1} |b_Q(\theta)|) e(\theta x_P + \alpha \theta x)
\end{align*}
for
\[ a_{Q;x}(\theta) := \rho_Q*f_\theta(x) = \sum_n \tilde{\rho_Q}(x-n) f(n) e(-n\alpha \theta) \]
with
\[ \widehat{ \tilde{\rho_Q} } := \widehat{\rho_Q} \cdot \phi(2^{-m_s-R} \cdot)\]
satisfies all the same estimates as does $\rho_Q$, e.g.\
\begin{align}\label{e:tildedecay} |\tilde{\rho_Q}(x)| \lesssim (R|Q|)^{-1} \cdot (1 + |x|/R|Q|)^{-100} \end{align}
so that
\begin{align}\label{e:estonaQ}
\| a_{Q;x} \|_{\ell^2(\Lambda)} &\leq \big( \sum_{\Lambda} |\sum_n \tilde{\rho_Q}(x-n) \cdot \mathbf{1}_{RQ}(x-n) f(n) e(-n\alpha \theta)|^2 \big)^{1/2} \\
& \qquad + \sum_{k \geq 1} \big( \sum_{\Lambda} |\sum_n \tilde{\rho_Q}(x-n) \cdot \mathbf{1}_{\text{dist}(\cdot, RQ) \approx 2^k R |Q|}(x-n) f(n) e(-n\alpha \theta)|^2 \big)^{1/2} \\
& \qquad \qquad \lesssim 1,
\end{align}
by Lemma \ref{l:rsum}, using the rapid decay of $\tilde{\rho_Q}$.
Similarly, 
\begin{align*} 
b_Q(\theta) &:= \frac{1}{|Q|} \sum_{m \in Q} g(m) e(-m \alpha \theta) \\
& \qquad = \frac{1}{|Q|} \sum_{m \leq |Q|} g(m + c_Q) e( - m \alpha \theta) e(- \alpha \theta c_Q)
\end{align*}
is a trigonometric polynomial in $\alpha \theta$ of degree $|Q| \gg 2^{-m_s + R}$, and $c_Q \in Q$ is the left end-point.

We abbreviate
\[ \vec{a_{Q;x}}, \ \vec{b_Q} \psi(\delta^{-1} |\vec{b_Q}|)\]
to denote
\[ \{ a_{Q;x}( \theta) : \theta \in \Lambda \}\]
and similarly
\[ \vec{b_Q} \psi(\delta^{-1} |\vec{b_Q}|) = \{ b_Q(\theta) \psi(\delta^{-1} |b_Q(\theta)|) : \theta \in \Lambda \}.\]

Note that
\[ \| \vec{a_{Q;x_P}} \eta_P \|_{\ell^2(\mathbb{Z})}^2 \lesssim |P| \cdot \max_{x \in 3P} \| \vec{a_{Q;x}} \|_{\ell^2(\Lambda)}^2 \lesssim |P| \]
by \eqref{e:estonaQ}, see \eqref{e:etaP}.
Similarly,
\begin{align}\label{e:bbd}
    \| \vec{b_Q} \|_{\ell^2(\Lambda)}^2 \lesssim |Q| \cdot \| \frac{1}{|Q|} \sum_{m \leq |Q|} g(m + c_Q) e( - m \beta) e(- c_Q \beta) \|_{L^2(\mathbb{T})}^2 \lesssim 1.
\end{align}

With this in mind, we bound, setting
\[ \Lambda := \alpha^{-1} \Lambda_s \]
\begin{align*}
&\| \eta_P(x) \max_{Q \in \mathcal{D}, \ Q \supset P}   | \sum_{\theta \in \Lambda} \rho_Q*f_\theta(x_P) \Psi_\delta\big( \Omega_{Q}'(\alpha \theta) \big) e(\theta x_P + \alpha \theta x)  | \|_{\ell^2(\mathbb{Z})}^2 \\
& \qquad \leq \| \eta_P(x) \max_{\vec{a_Q;x_P}, \vec{b_Q}} |\sum_{\theta \in \Lambda} a_{Q;x_P}(\theta) b_Q(\theta) \psi(\delta^{-1} |b_Q(\theta)|) e(\theta x_P + \alpha \theta x) | \|_{\ell^2(\mathbb{Z})}^2;
    \end{align*}
if we let
\[ D_\epsilon(P) := \text{Ent}_1( \vec{a_{Q;x_P}}(\theta) \vec{b_Q}(\theta) \psi(\delta^{-1} |b_Q(\theta)|) : Q \supset P ; \epsilon) \]
denote the $\ell^1(\Lambda)$-entropy at altitude $\epsilon > 0$, see \eqref{e:Ent}, then for any $Q$, we may bound
\begin{align*}
    &|\sum_{\theta \in \Lambda} a_{Q;x_P}( \theta) b_Q( \theta) \psi(\delta^{-1} |b_Q(\theta)|) e( \theta x_P  + \alpha \theta x)| \\
    & \qquad \leq O(\epsilon) + \big( \sum_{i \leq D_\epsilon} |\sum_{\theta \in \Lambda} a_{Q_i;x_P}(\theta) \cdot b_{Q_i}(\theta) \cdot  \psi(\delta^{-1} |b_{Q_i}(\theta)|) e( \theta x_P + \alpha \theta x) |^2 \big)^{1/2} 
\end{align*}
for $\{ Q_i \}$ which realize the $\ell^1(\Lambda)$-entropy at altitude $\epsilon$; see \eqref{e:ent00}.

In particular, we may estimate
\begin{align*}
    &\| \eta_P(x) \max_{\vec{a_{Q;x_P}}, \vec{b_Q}} |\sum_{\theta \in \Lambda} e( \theta x_P + \alpha \theta x) a_{Q;x_P}(\theta) \cdot b_Q(\theta) \cdot \psi(\delta^{-1} |b_Q(\theta)|)| \|_{\ell^2(\mathbb{Z})}^2 \\
    & \qquad \lesssim \epsilon^2 |P| + \delta^2 D_\epsilon(P) \cdot |P|
\end{align*}
as
\[ \max_{\vec{a_{Q;x_P}}, \vec{b_Q}} \| \eta_P(x) \cdot \sum_{\theta \in \Lambda} e(\theta x_P + \alpha \theta x) a_{Q_i;x_P}(\theta) b_{Q_i}(\theta) \psi(\delta^{-1} |b_{Q_i}(\theta)|) \|_{\ell^2(\mathbb{Z})}^2 \lesssim \delta^2 \cdot |P|
\]
by \eqref{e:estonaQ} and \eqref{e:bbd}.

But, if we let
\[ N_\epsilon( \vec{a_{Q;x_P}} : Q \supset P ), \; \; \;  N_\epsilon( \vec{b_Q} : Q \supset P ) \]
denote the $\ell^2(\Lambda)$-entropy of the pertaining vectors at altitude $\epsilon > 0$, then we may bound
\begin{align}
    D_\epsilon(P) \leq N_{\epsilon/10}( \vec{a_{Q;x_P}} : Q \supset P ) \cdot N_{\epsilon/10}( \vec{b_Q} : Q \supset P ) 
\end{align}
for a local contribution of
\begin{align*}
   & \| \eta_P(x) \max_{\vec{a_{Q;x_P}}, \vec{b_Q}} |\sum_{\theta \in \Lambda} e( \theta x_P+\alpha \theta x) a_{Q;x_P}(\theta) b_Q(\theta) \psi(\delta^{-1} |b_Q(\theta)|)| \|_{\ell^2(\mathbb{Z})}^2 \\
   & \qquad \lesssim \epsilon^2 \cdot |P| + \delta^2     \cdot N_{\epsilon/10 }( \vec{a_{Q;x_P}} : Q \supset P ) \cdot N_{\epsilon/10}( \vec{b_Q} : Q \supset P ) \cdot |P|,
\end{align*}
so summing over $\{|P| = 2^{-m_s}, \ P \subset I_j\}$ and $\{ I_v \in T_v^{\text{max}} \}$ yields a contribution of 
\begin{align*}
    \epsilon^2 \cdot H + \delta^2  \sum_{I_j \in T_j^{\text{max}}} \sum_{P \subset I_j} N_{\epsilon/10}( \vec{a_{Q;x_P}} : Q \supset P ) \cdot N_{\epsilon/10}( \vec{b_Q} : Q \supset P ).
\end{align*}
In particular, away from a union of dyadic intervals of side length $2^{-m_s}$, $Y_\epsilon \subset [H]$, of size
\[ |Y_\epsilon| \leq \epsilon^2 \cdot H \]
we may bound
\[ N_{\epsilon/10}( \vec{a_{Q;x_P}} : Q \supset P ) \lesssim  \epsilon^{-4} \]
by vector-valued jump-counting estimates, Corollary \ref{c:jump0}. We excise $Y_\epsilon$ from consideration (before passing to $\ell^2(\mathbb{Z})$-based estimates).

Specifically, we estimate
\begin{align*}
&  \sum_{I_j \in T_j^{\text{max}}} \sum_{|P| = 2^{-m_s}, \ P \subset I_j} |\{ P \cap \overline{\mathcal{D}} : \max_{P \subset Q \in \mathcal{D}}   | \sum_{\theta \in \Lambda} \rho_Q*f_\theta(x_P) \Omega_{Q,\delta}'(\alpha \theta) e(\theta x_P+\alpha \theta x) | \gtrsim t \}| \\
& \qquad \leq |Y_\epsilon| + \sum_{I_j \in T_j^{\text{max}}} \sum_{|P| = 2^{-m_s}, \ P \subset I_j \smallsetminus Y_\epsilon} |\{ P \cap \overline{\mathcal{D}} : \max_{P \subset Q \in \mathcal{D}}   | \sum_{\theta \in \Lambda} \rho_Q*f_\theta(x_P) \Omega_{Q,\delta}'(\alpha \theta) e(\theta x_P+\alpha \theta x) | \gtrsim t \}| \\
& \qquad \qquad \lesssim 
\epsilon^2 \cdot H + \delta^2  \cdot \epsilon^{-6} \cdot \sum_{I_j \in T_j^{\text{max}}} \sum_{P \subset I_j} \epsilon^2 \cdot  N_{\epsilon/10}( \vec{b_Q} : Q \supset P ) \cdot |P|
\end{align*}
But we may bound
\[ 
\sum_{P \subset I_j} \epsilon^2 \cdot  N_{\epsilon}( \vec{b_Q} : Q \supset P ) \cdot |P| \leq \sum_{x \in I_j} \epsilon^2 \cdot N_\epsilon( \big( \mathbb{E}_Q \text{Mod}_\theta g\big)_{\theta \in  \Lambda_s} : x \in Q, \ |Q| \gg 2^{-m_s + R} ) \]
so, upon summing over $I_v \in T_v^{\text{max}}$, we arrive at the bound
\[ \epsilon^2  \cdot H + \delta^2  \cdot  \epsilon^{-6}  \cdot \sum_{y \in [H]} \epsilon^2 N_{\epsilon} ( \big( \mathbb{E}_Q \text{Mod}_\theta g \big)_{\theta \in  \Lambda_s} : y \in Q, \ |Q| \gg 2^{-m_s + R} )\]
which is bounded by
\[ (\epsilon^2 + \delta^2 \epsilon^{-6} ) \cdot H \lesssim \delta^{1/10} \cdot H  \]
by Corollary \ref{c:sep}.
\end{proof}

\section{Proof of the Maximal Estimate, Return Times}\label{s:RTmax}
In this section we establish \eqref{e:maxRT}, the second statement in Proposition \ref{p:max0}. We maintain the numerology of $V = \delta^{-2 \alpha_0}, \epsilon = \delta^{1/10}$ etc.

So, with $I_0 = [N]$, we extract a $V$ forest, and discard
\[ |\bigcup_{\mathcal{D}_V(I_0)} I | \lesssim t^{-4} V^{-1} \delta^{-o(1)} N, \]
which is acceptably small.

We need to show that for each $v < V$
\begin{align}\label{e:step0}
    \Big|\{ a \in \overline{ \mathcal{D} } \cap I_v \smallsetminus X_{v+1} : |\{ [H] : \sup_{N} |\int \Omega_{J_N(a),\delta}'(\alpha \beta) \hat{f}(\beta) e(  \beta x + \alpha a \beta) \ d\beta | \gg t \}| \gg t^B H \}\Big| \ll \delta^{c} \cdot |I_v|,
\end{align} 
as this estimate is strong enough to absorb the loss of $V = \delta^{-2 \alpha_0}$; above
\[ X_v := \bigcup_{I \in \mathcal{D}_v(I_0)} I = \bigcup_{T_v^{\text{max}}} I_v \subset [N]. \]
Below, we will implicitly restrict all $a \in [N]$ to $\overline{\mathcal{D}}( [N] )$, with the understanding that all arguments are uniform over each choice of dyadic grid.

Now, by Lemma \ref{l:cont}, we reduce to \eqref{e:step0}
\begin{align}\label{e:step1}
    |\Big\{ a \in I_v \smallsetminus X_{v+1} : |\{ [H] : \sup_{a \in J \in \mathcal{D}_v(I_v)} |\int \Omega_{J,\delta}'(\alpha \beta) \hat{f}(\beta) e( (\alpha a+x) \beta) \ d\beta | \gg t \}| \gg t^B H \Big\}| \ll \delta^{c} |I_v|
\end{align} 
uniformly in $I_v \in T_v^{\text{max}} \subset \mathcal{D}([N])$.

We now reduce to $V$-branches, first discarding the boundary branch $\mathcal{B}_\infty(I_v)$; this estimate follows by the union bound:
if 
\[ E := \{ |I| : I \in \mathcal{B}_\infty(I_v) \} \subset 2^{k \Delta} \]
is a set of times of size $\lesssim RW$, then 
\[ |\{ [H] : \sup_{a \in J \in \mathcal{D}_v(I_v), \ |J| \in E} |\mathcal{F}_\delta^1(J;x,a) | \gg t \}| \lesssim \delta^2 RW  t^{-2} \cdot H \ll  \delta \cdot H; \]
in particular, we may replace \eqref{e:step1} with 
\begin{align}\label{e:step2}
    | \Big\{ a \in I_v \smallsetminus X_{v+1} : |\{ [H] : \sup_{a \in J \in \mathcal{B}_s(I_v)} |\mathcal{F}_\delta^1(J;x,a)  | \gg t \}| \gg t^B H \Big\}| \ll \delta^{c} \cdot |I_v|
\end{align} 
provided our bounds are uniform in $s \leq W$; and we may further replace
\[ \Omega_{J,\delta} \longrightarrow \Omega_{J,\delta} \Xi_J^0 \]
to reduce \eqref{e:step2} to
\begin{align}\label{e:step3}
    | \Big\{ a \in  I_v \smallsetminus X_{v+1} : |\{ [H] : \sup_{a \in J \in \mathcal{B}_s(I_v)} |\mathcal{F}_\delta^2(J;x,a) | \gg t \}| \gg t^B H \Big\}| \ll \delta^{c} |I_v|.
\end{align} 
Since
\[ \sum_{J : a \in J \in \mathcal{B}_s(I_v)} \|\mathcal{F}_\delta^2(J;x,a) - \mathcal{F}_\delta^3(J;x,a)\|_{\ell^2([H])}^2 \lesssim \delta^2 \log R \cdot H \ll \delta H \]
we may replace \eqref{e:step3} with 
\begin{align}\label{e:step4}
    | \Big\{ a \in I_v \smallsetminus X_{v+1} : |\{ [H] : \sup_{a \in J \in \mathcal{B}_s(I_v)} |\mathcal{F}_\delta^3(J;x,a) | \gg t \}| \gg t^B H \Big\}| \ll \delta^{c} |I_v|.
\end{align} 
By arguing as in Subsection \ref{ss:replace} -- formally substituting the variable $x \longrightarrow a$ but leaving the remainder of the argument untouched -- we further reduce to proving, with
\[ \Lambda := \alpha^{-1} \Lambda_{v,s}(I_v) \]
\begin{align}\label{e:step5}
    &| \Big\{ a \in I_v \smallsetminus X_{v+1} : |\{ [H] : \sup_{a \in J \in \mathcal{B}_s(I_v)} |\sum_{\theta \in \Lambda} e(x \theta) \cdot e( \alpha a \theta) \rho_{RJ}*f_\theta(x) \cdot \Omega_{J,\delta}'(\alpha \theta) | \gg t \}| \gg t^B H \Big\}|\\
    & \qquad \equiv  | \Big\{ a \in I_v \smallsetminus X_{v+1} : |\{ [H] : \sup_{a \in J \in \mathcal{B}_s(I_v)} |\sum_{\theta \in \Lambda} e(x \theta) \cdot e( \alpha a \theta) \rho_{RJ}*f_\theta(x) \cdot \Omega_{J,\delta}(\alpha \theta) | \gg t \}| \gg t^B H \Big\}|\\
    & \qquad \qquad \ll \delta^{c} \cdot |I_v|. \notag
\end{align} 
We now define the exceptional set
\begin{align}
    E_\delta := \bigcup_{v \leq V} \bigcup_{I_v \in T_v^{\text{max}}} \bigcup_{s \leq W} \{ a \in I_v : \epsilon^2 \cdot N_{\epsilon/20} ( \big( \mathbb{E}_Q \text{Mod}_\theta g\big)_{\theta \in \Lambda_{v,s}(I_v)}(a) : |Q| \gg 2^{-m_s + R} ) \geq \delta^{-1/2} \}
\end{align}
which has cardinality $\ll \delta^{1/10} \cdot N$. And, away from this exceptional set, we may bound
\begin{align*}
    &|\{ [H] : \sup_{a \in J \in \mathcal{B}_s(I_v)} |\sum_{\theta \in \Lambda} e(x \theta) \cdot e( \alpha a \theta) \rho_{RJ}*f_{\theta}(x) \cdot \Omega_{J,\delta}'(\alpha \theta) | \gg t \}| \\
    & \qquad \lesssim t^{-2} \sum_{P \subset [H] : |P| = 2^{-m_s}} \| \eta_P(x) \cdot \sup_{a \in J \in \mathcal{B}_s(I_v)} |\sum_{\theta \in \Lambda } e(x \theta) \cdot e(\alpha a \theta) \rho_{RJ}*f_{\theta}(x) \cdot \Omega_{J,\delta}'(\alpha \theta) | \|_{\ell^2(\mathbb{Z})}^2 \\
    & \qquad \qquad \lesssim t^{-2} \cdot \Big( \epsilon^2 \cdot H + \delta^2 \cdot \epsilon^{-4} \cdot \delta^{-1/2} \sum_{x \in [H]} \epsilon^2 N_{\epsilon/10}( \big( \rho_Q*f_{\theta}(x) \big)_{\theta \in \Lambda} : |Q| \gg 2^{-m_s+R} ) \Big) \\
    & \qquad \qquad \qquad \lesssim \delta^{1/10} \cdot H.
\end{align*}
In particular, we have shown that
  \begin{align*}
  &\{ a \in X_v \smallsetminus X_{v+1} : |\{ [H] : \sup_{N} |\int \Omega_{J_N(a),\delta}'(\alpha \beta) \hat{f}(\beta) e( (\alpha a+x) \beta) \ d\beta | \gg t \}| \gg t^B H \Big\} \\
    & \qquad \subset \bigcup_{I_v \in T_v^{\text{max}}} \bigcup_{s \leq W} E_s(I_v) 
    \\
    & \qquad \qquad \subset E_\delta,
    \end{align*}
where
\[ E_s(I_v) := \{ a \in I_v \smallsetminus X_{v+1} : |\{ [H] : \sup_{a \in J \in \mathcal{B}_s(I_v)} |\sum_{\theta \in \Lambda} e(x \theta)  e( \alpha a \theta) \rho_{RJ}*f_{\theta}(x)  \Omega_{J,\delta}(\alpha \theta) | \gg t \}| \gg t^B H \};\]
this completes the proof.

\section{Proof of the Oscillation Inequality, Double Recurrence}\label{s:DRosc}
The goal of this section is to prove the first statement in Proposition \ref{p:osc00}, namely \eqref{e:oscDR}.

So, let $K \lesssim_H 1$ be maximal subject to the constraint that there exist $M_0 < M_1 < \dots < M_K$ so that
\begin{align}
|\{ \overline{ \mathcal{D}} \cap [H] : \sup_{x \in J \in \mathcal{Q}_k} |\hat{\mathcal{F}}_\delta^1(J;x)| \gg t \}| \gg t^B H 
\end{align}
for each $k \leq K$, see \eqref{e:Q}. With $V = t^{-B^2}$, $L =K$, extract an $(L,V)$-forest, with tree tops
\[ \bigcup_{v < V} \bigcup_{T_v^{\text{max}}} I_v\]
and let $l(I_v)$ denote the unique set of intervals so that there exists $J \in Q_{l(I_v)} \cap \mathcal{B}_s(I_v)$ with
\[ J \subset I \subsetneq \hat{J} \]
for some $s < W$ (if it exists); in other words, $I_v \in T_v^{\text{max},l(I_v)}$.

Then observe that for each $s \in [W] \cup \{ \infty\}$ the function
\[ \sup_{x \in J \in (\mathcal{Q}_l \cap \mathcal{B}_s(I_v))} |\hat{\mathcal{F}}_\delta^j(J;x)| \leq 2 \cdot \mathbf{1}_{I_v} \]
is supported inside $I_v$, so that
\begin{align*}
    &\sum_{|l| \leq W}
    \sum_{v < V} \sum_{T_v^{\text{max}}} \sum_{s < W} \sup_{x \in J \in (\mathcal{Q}_{l(I_v) + l} \cap \mathcal{B}_s(I_v))} |\hat{\mathcal{F}}_\delta^j(J;x)|^2 \\
    & \qquad \lesssim W^2 \cdot \sum_{v < V} \sum_{T_v^{\text{max}}} \mathbf{1}_{I_v} \lesssim W^2 V \cdot \mathbf{1}_{I_0}.
\end{align*} 
In particular, for each $s \in [W] \cup \{ \infty\}$
\[ \sum_{x \in I_0} \Big( \sum_{|l| \leq W}
    \sum_{v < V} \sum_{T_v^{\text{max}}} \sum_{s < W} \sup_{x \in J \in (\mathcal{Q}_{l(I_v) + l} \cap \mathcal{B}_s(I_v))}|\hat{\mathcal{F}}_\delta^j(J;x)|^2 \Big) \lesssim W^2 V \cdot |I_0|, \]
independent of $K$.

As for the sum
\[ \sum_{v <V} \sum_{T_v^{\text{max}}} \sum_{|l-l(I_v)| \geq W} \sup_{x \in J \in \mathcal{Q}_l \cap \mathcal{B}_\infty(I_v)} |\hat{\mathcal{F}}^j_\delta(J;x)|^2, \]
we just bound
\[ \sum_{|l-l(I_v)| \geq W} \sup_{x \in J \in \mathcal{Q}_l \cap \mathcal{B}_\infty(I_v)} |\hat{\mathcal{F}}^j_\delta(J;x)|^2 \leq RW \cdot \sup_{x \in J \in \mathcal{B}_\infty(I_v)} |\hat{\mathcal{F}}^j_\delta(J;x)|^2 \leq RW \cdot \mathbf{1}_{I_v} \]
as there are only $O(RW)$ many scales that appear in $\mathcal{B}_\infty(I_j)$, and thus only $O(RW)$ many terms in the above sum that do not vanish identically.

In particular
\[ \sum_{v <V} \sum_{T_v^{\text{max}}} \sum_{|l-l(I_v)| \geq W} \sup_{x \in J \in \mathcal{Q}_l \cap \mathcal{B}_\infty(I_v)} |\hat{\mathcal{F}}^j_\delta(J;x)|^2 \leq RWV \cdot \mathbf{1}_{I_0},
\]
and since this last function has sum bounded by
\[ R WV |I_0|, \]
independent of $K$, we can discard this contribution.

In what follows, we will group scales as
\[ \{ \mathcal{Q}_l \cap \mathcal{B}_s(I_v) : I_v \in T_v^{\text{max}}, s \leq W, l \leq l(I_v) - W,\ v < V \}\]
so that when studying
\begin{align*}
    \{ \sup_{x \in J \in \mathcal{Q}_l \cap \mathcal{B}_s(I_v)} |\hat{\mathcal{F}}^j_\delta(J;x)| \gg t \}
\end{align*}
we may implicitly restrict
\begin{align}\label{e:rest} f \longrightarrow f \cdot \mathbf{1}_{3I_v}.
\end{align}

With $f$ so-restricted, we will prove that 
\begin{align*}
\sum_{x} \sum_{l \leq K, \ l(I_v) - W} \sup_{x \in J \in \mathcal{Q}_l \cap \mathcal{B}_s(I_v)} |\chi_t\big( \hat{\mathcal{F}}_\delta^2(J;x) \big)|^2 \lesssim K^{9/10} \cdot |I_v|. 
\end{align*}
By Corollary \ref{c:F2F3}, since
\[ \sum_x \sum_{x \in J \in \mathcal{B}_s(I_v)}|\chi_t\big( \mathcal{F}_\delta^2(J;x) - \mathcal{F}_\delta^3(J;x) \big)|^2 \lesssim R^2 \delta^2 |I_v| \]
is independent of $K$, it suffices to prove
\begin{align*}
\sum_{x} \sum_{l \leq K, \ l(I_v) - W} \sup_{x \in J \in \mathcal{Q}_l \cap \mathcal{B}_s(I_v)} |\hat{\mathcal{F}}_\delta^3(J;x)|^2 \lesssim K^{9/10}  \cdot |I_v| 
\end{align*}

There are two bounds that we will use: first
\[ \sum_{l \leq K, \ l(I_v) - W} \sup_{x \in J \in \mathcal{Q}_l \cap \mathcal{B}_s(I_v)} |\hat{\mathcal{F}}_\delta^3(J;x)|^2 \lesssim W K;\]
second, for $2 < r < \infty$, using H\"{o}lder's and then the triangle inequality to pull out the sum in 
\[ \{ \theta \in \Lambda := \alpha^{-1} \Lambda_{v,s}(I_v) \}, \]
we estimate
\begin{align*}
    & (\sum_{l \leq K, \ l(I_v) - W} \sup_{x \in J \in \mathcal{Q}_l \cap \mathcal{B}_s(I_v)} |\hat{\mathcal{F}}_\delta^3(J;x)|^2)^{1/2} \\
    & \qquad \lesssim K^{1/2-1/r} \sum_{\theta \in \Lambda} \mathcal{V}^r( \varphi_{R|J|}*f_{\theta}(x) \cdot \Omega_{J,\delta}(\alpha \theta) : x \in J \in \mathcal{Q}_l \cap \mathcal{B}_s(I_v)),
\end{align*} 
where $\widehat{f_\theta}(\beta) := \varphi(2^{m_{s}+R} \beta) \hat{f}(\beta+\alpha \theta)$.
Since $\Psi_\delta$ is uniformly $\mathcal{C}^1$, and thus uniformly Lipschitz, we may bound, for each $\theta$
\begin{align*}
    &\mathcal{V}^r( \varphi_{R|J|}*f_{\theta}(x) \cdot \Omega_{J,\delta} (\alpha \theta) : x \in J \in \mathcal{Q}_l \cap \mathcal{B}_s(I_v)) \\
    & \qquad \leq \mathcal{V}^r(\varphi_{R|J|}*f_{\theta}(x) : |J|) \cdot \mathcal{V}^r( \Omega_{J,\delta} (\alpha \theta) :  x \in J \in \mathcal{Q}_l \cap \mathcal{B}_s(I_v)) \\
& \qquad \qquad \lesssim \mathcal{V}^r(\varphi_{R|J|}*f_{\theta}(x) : |J|) \cdot \mathcal{V}^r( \Omega_{J}(\alpha \theta) :  x \in J \subset I_v) =: \mathcal{V}^r_{\theta,1}(x) \cdot \mathcal{V}^r_{\theta,2}(x),
\end{align*}
see \eqref{e:prod}.

By Lemma \ref{l:var}, we have the bound
\[ \sup_{\theta,i} \ \| \mathcal{V}^r_{\theta,i} \|_{\ell^{2,\infty}(\mathbb{Z})} \lesssim |I_v|^{1/2};
\]
in particular, if
\[ X_K := \bigcup_{\theta \in \Lambda_{v,s}(I_v), i } \{ x \in I_v : \mathcal{V}^r_{\theta,i} \geq K^{1/5} \} \] 
then
\[ |X_K| \lesssim K^{-1/5} \cdot |I_v| \]
for all $K$ sufficiently large; we here use the restriction \eqref{e:rest}. And, we just bound
\begin{align*}
    &\sum_{x} \sum_{l \leq K, \ l(I_v) - W} \sup_{x \in J \in \mathcal{Q}_l \cap \mathcal{B}_s(I_v)} |\hat{\mathcal{F}}_\delta^3(J;x)|^2 \\
    & \qquad \lesssim W K \cdot |X_K| + |\Lambda_{v,s}(I_v)|^2 \cdot (K^{1/2-1/r + 1/5})^2 \cdot |I_v| \\
    & \qquad \qquad \ll K^{9/10} \cdot |I_v|.
\end{align*}
provided that $2 < r \leq 3$ and $K \gg_\tau 1$ is sufficiently large.

\section{Proof of the Oscillation Inequality, Return Times}\label{s:RTosc}
We conclude the main body of the paper by proving the second statement in Proposition \ref{p:osc00}, namely \eqref{e:oscRT}

For $a \in [N]$, let $K \lesssim_H 1$ be maximal subject to the constraint that there exist $M_0 < M_1 < \dots < M_K$ so that, 
\begin{align*}
|\{ x \in [H] : \sup_{a \in J \in \mathcal{Q}_k} |\hat{\mathcal{F}}_\delta^1(J;x,a)| \gg t \}| \gg t^B H 
\end{align*}
for each $k \leq K$, see \eqref{e:Q}. With $V = R^{A_0}$, $L=K$, extract an $(L,V)$-forest, with tree tops
\[ \bigcup_{v < V} \bigcup_{T_v^{\text{max}}} I_v \subset [N]; \]
with $X_v := \bigcup_{T_v^{\text{max}}} I_v$ as above, it suffices to prove that
\begin{align*}
    |\{ (\overline{\mathcal{D}} \cap I_v) \smallsetminus X_{v+1} : C_{\tau,B,H}(a) \gg L_0 \}| = o_{L_0 \to \infty}( V^{-O(1)} |I_v| \Delta^{-O(1)});
\end{align*}
in this context, we may restrict $a \in J \in \mathcal{D}_v(I_v)$, and more precisely, to $a \in J \in \mathcal{B}_s(I_v)$ for some $s \leq W$, as
\begin{align*}
    K t^{B+2} H &\leq \sum_{k \leq K} t^2 \cdot |\{ x \in [H] : \sup_{a \in J \in \big( \mathcal{Q}_k \cap \mathcal{B}_{\infty}(I_v) \big) } |\hat{\mathcal{F}}_\delta^1(J;x,a)| \gg t \}| \\
    & \qquad \leq RW \cdot H
\end{align*}
which bounds $K = O_t(1)$ absolutely for the above contribution, as there are at most $O(RW)$ many scales $k \leq K$ so that
\[ \mathcal{Q}_k \cap \mathcal{B}_\infty(I_v) \neq \emptyset,\]
as above, and we just use the trivial estimate
\[ |\{ x \in [H] : \sup_{a \in J \in \big( \mathcal{Q}_k \cap \mathcal{B}_{\infty}(I_v) \big) } |\hat{\mathcal{F}}_\delta^1(J;x,a)| \gg t \}| \leq H.\]

So, if we define
\[ C_{\tau,B,H}^{v,s}(a) \]
to be the maximal $K$ so that there exist $M_0 < M_1 < \dots < M_K$ so that 
\begin{align*}
|\{ x \in [H] : \sup_{a \in J \in \big( \mathcal{Q}_k \cap \mathcal{B}_s(I_v) \big)} |\hat{\mathcal{F}}_\delta^1(J;x,a)| \gg t \}| \gg t^B H 
\end{align*}
we need to show that
\[ |\{ a \in (\overline{\mathcal{D}} \cap I_v) \smallsetminus X_{v+1} :  C_{\tau,B',J}^{v,s}(a) \geq L_0 \}| = o_{L_0 \to \infty}( (VW)^{-O(1)} \Delta^{-O(1)} |I_v|)
\]
where $B \ll B' = B'(B,t) = O_{\tau,B}(1)$ is an absolute constant.

We now define the following exceptional set
\begin{align*}
    X(I_v) := \{ a \in I_v : \text{ there exists } \theta \in \Lambda_{v}(I_v) :  \mathcal{V}^r(\Omega_{J}(\theta) : a \in J \subset I_v) \geq V^{2 A_0} \},
\end{align*}
which has size $\ll V^{- A_0} \cdot |I_v|$ as we may of course restrict the above-specified (and fixed) $g$ in the definition of $\Omega_{J}$ to be supported on $I_v$. In particular, since 
\[ |X(I_v)| = o( V^{-A_0} |I_v|),\]
in what follows, we will assume that 
\[ a \in (\overline{\mathcal{D}} \cap I_v) \smallsetminus X(I_v).\]

Now, after adjusting the implicit constant in the super-level set ($\gg t$), we may replace $\hat{\mathcal{F}}_\delta^1$ with $\hat{\mathcal{F}}_\delta^2$ in the definition of $C_{\tau,B',J}^{v,s}$, see Corollary \ref{c:bdrybranch}.

To replace $\hat{\mathcal{F}}_\delta^2 \longrightarrow \hat{\mathcal{F}}_\delta^3$, we just bound
\begin{align*}
    &\sum_{k \leq K} \sum_{a \in J \in \big( \mathcal{Q}_k \cap \mathcal{B}_s(I_v) \big)} \\
    & \qquad \|\int \sum_{\Lambda } (\varphi(\alpha |J|/R(\beta - \theta)) - \varphi(\alpha R|J|(\beta - \theta))) \Omega_{J,\delta}'(\alpha \beta) \hat{f}( \beta) e(\beta (x+\alpha a)) \ d\beta\|_{\ell^2([H])}^2 \\
    & \qquad \qquad \lesssim \delta^2 \log R \cdot H
\end{align*}
which forces an absolute upper bound on $K$ deriving from this contribution; above, we maintain the notation
\[\Lambda := \alpha^{-1} \Lambda_{v,s}(I_v).\]

So, possibly after excising a total of $O(W^{10})$ (say) many scales $k \in [K]$ which could possibly support the lower bound
\[ |\{ x \in [H] : \max_{M_{k-1} \leq M \leq M_k} |{A}_M(f,g)(x;a) - {A}_{M_k}(f,g)(x;a)| \geq \tau \}| \geq \tau^B H\]
and again adjusting implicit constants, we may replace $\hat{\mathcal{F}}_\delta^2$ with $\hat{\mathcal{F}}_\delta^3$ in the definition of $C_{\tau,B',H}^{v,s}$ and reduce to the case where $\mathcal{Q}_k$ satisfies $k \leq l(I_v) - W$; let $E_K = E_K(a) \subset [K]$ have $|E_K| = K - O(W^{10})$ be the remaining scales.

At last, after once again adjusting implicit constants, we may replace $\hat{\mathcal{F}}_\delta^3$ with
\begin{align}\label{e:supRT}
\sup_{a \in J \in \big( \mathcal{Q}_k \cap \mathcal{B}_s(I_v) \big)} &|\sum_{\theta \in \Lambda} e(\alpha \theta a) \cdot \int \big( \varphi(\alpha R|J|(\beta - \theta)) \cdot \Omega_{J,\delta} (\alpha \theta) - \varphi(\alpha R|\hat{J}|(\beta - \theta)) \cdot \Omega_{\hat{J},\delta} (\alpha \theta) \big) \\
& \qquad \qquad \times \hat{f}(\beta) e(\beta x) \ d\beta|,
\end{align}
and we bound
\begin{align}\label{e:prodVr0} 
&\big( \sum_{k \in E_K} \eqref{e:supRT}^2 \big)^{1/2} \\
& \qquad \lesssim K^{1/2- 1/r} \cdot \sum_{\theta \in \Lambda} \mathcal{V}^r( \varphi_{R|J|}*f_{\theta} : a \in J \subset I_v) \cdot \mathcal{V}^r(\Omega_{J}(\alpha \theta) : a \in J),
\end{align}
using \eqref{e:prod}, and setting
\[ \varphi_{R|J|} := (\alpha R |J|)^{-1} \varphi^{\vee}(\alpha R |J|^{-1} \cdot).\]
Square summing
\begin{align*}
&K t^{B'} H \lesssim \| \eqref{e:prodVr0} \|_{\ell^2([H])}^2 \\
& \qquad \lesssim K^{1-2/r} \cdot (V \delta^{-2}) \cdot  \sum_{\theta \in \Lambda} \sum_{[H]} \mathcal{V}^r( \varphi_{R|J|}*f_{\theta} : a \in J \subset I_v)^2 \cdot \mathcal{V}^r(\Omega_{J}(\alpha \theta) : a \in J)^2 \\
& \qquad \qquad \lesssim K^{1 - 2/r} \cdot V^{4A_0} \cdot (V \delta^{-2}) \sum_{\theta \in \Lambda} \sum_{[H]} \mathcal{V}^r( \varphi_{R|J|}*f_\theta : a \in J \subset I_v)^2 \\
& \qquad \qquad \qquad \lesssim K^{1-2/r} \cdot V^{4A_0} \cdot (V \delta^{-2})^2 \cdot H,
\end{align*}
which forces an upper bound $K = C_{\tau,B,H}^{v,s}(a) = O_t(1)$ for all $a \in I_v \smallsetminus X(I_v)$; above, the initial prefactor of $V \delta^{-2}$ comes from Cauchy-Schwartz and the upper bound
\[ |\Lambda| = |\Lambda_{v,s}(I_v)| \leq |\Lambda_{v}(I_v)| \lesssim V \delta^{-2}.\]
The proof is (at last) complete.


\section{Appendix: Measurability Issues} \label{s:appendix}
The goal of this section is to prove provide a proof that 
\begin{align}\label{e:bad0} &\{ \omega \in Y : \sup_{(X,\mu,T)} C_{\tau,B,H}^{(X,\mu,T)}(\omega) = + \infty \} \\
& \qquad = \bigcap_{L \geq 1} \{ \omega \in Y : \sup_{(X,\mu,T)} C_{\tau,B,H}^{(X,\mu,T)}(\omega) \geq L \} \notag
\end{align}
is $\nu$-measurable; we will also reduce to the case where $(Y,\nu,S)$ is an \emph{ergodic} Lebesgue space. We do so be containing
\[ \{ \omega \in Y : \sup_{(X,\mu,T)} C_{\tau,B,H}^{(X,\mu,T)}(\omega) \geq L \} 
 \subset E_{\tau,B,H,L} \subset Y, \]
where $E_{\tau,B,H,L}$ is $\nu$-measurable with
\[ \nu(E_{\tau,B,H,L}) = o_{L \to \infty;\tau,B}(1)
\]
so that we can apply downwards continuity of measure (recall that $Y$ is a probability space).
We do so in steps; these arguments are similar to those appearing in \cite[\S 4]{D}. We first reduce to the case of indicator functions, arguing as in \cite[Lemma 1.7]{BOOK}, decomposing a general $1$-bounded function into a linear combination of (at most) four $1$-bounded non-negative functions, and then each $1$-bounded non-negative function into a geometrically decaying sum of indicators,
\[ f(x) = \sum_{k =0}^{\infty} 2^{-k} \mathbf{1}_{F_k}(x)\]
where 
\[ F_k := \{ x \in X : f(x) = \sum_{k \geq 0} \frac{\omega_k(x)}{2^k} \text{ has } \omega_k(x) = 1 \}, \; \; \; \omega_k(x) \in \{0,1\} \]
and we stipulate $\liminf \omega_k(x) = 0$. And, similarly for each $g$ considered. So, in what follows we will assume that all $f,g$ introduced are indicators.

\begin{lemma}[Measurability of \eqref{e:bad0}, I]\label{l:I}
For each $(X,\mu,T)$ let $\Sigma(X)$ denote a countable subset of indicator functions that is dense in the class of indicator functions with respect to the $L^1(X)$ norm,\footnote{It is here that we require $(X,\mu)$ to have a countably generated $\sigma$-algebra} and let
\[ \Sigma_{\tau,B,H}^{(X,\mu,T)} \]
denote the restriction of $C_{\tau,B,H}^{(X,\mu,T)}$ to only indicators with $F \in \Sigma(X)$. Then
\[ C_{\tau,B,H}^{(X,\mu,T)} \leq \Sigma_{\tau/2,B,H}^{(X,\mu,T)}. \]
\end{lemma}
\begin{proof}
    For any $f,g$,
    \[ \limsup_M |\Phi_M(f,g)(x;\omega)| \leq M_Tf(x), \]
    where 
    \[ M_T f(x) := \sup_{N \geq 1} \, \frac{1}{N} \sum_{n \leq N} T^n |f|\]
    denotes the (one-sided) ergodic maximal function. In particular, for any $\phi = \mathbf{1}_V \in \Sigma(X)$, we are free to compare
    \[ |\Phi_M(f,g)(x;\omega) - \Phi_M(\phi,g)(x;\omega)| \leq M_T(f- \phi)(x),\]
so for each $k \leq K$ and each $\phi$, we have the containment
\begin{align*}
    &\{ x \in X : \max_{M_{k-1} \leq M \leq M_k} |\Phi_M(f,g)(x;\omega) - \Phi_{M_k}(f,g)(x;\omega)| \gg \tau \} \\
    & \qquad \subset \{ M_T(f - \phi ) \geq \tau/2 \} \cup \{ x \in X : \max_{M_{k-1} \leq M \leq M_k} |\Phi_M(\phi,g)(x;\omega) - \Phi_{M_k}(\phi,g)(x;\omega)| \geq \tau/2 \},
\end{align*}
    which leads to the estimate
\begin{align*}
    &\mu(\{ x \in X : \max_{M_{k-1} \leq M \leq M_k} |\Phi_M(f,g)(x;\omega) - \Phi_{M_k}(f,g)(x;\omega)| \gg \tau \}) \\
    & \qquad \leq O(\tau^{-1} \| f - \phi\|_{L^1(X)} ) + \mu(\{ x \in X : \max_{M_{k-1} \leq M \leq M_k} |\Phi_M(\phi,g)(x;\omega) - \Phi_{M_k}(\phi,g)(x;\omega)| \geq \tau/2 \}),
\end{align*}
by the weak-type $(1,1)$ boundedness of the ergodic maximal function.

In particular, if there exists an indicator function $f$, so that
\[ \mu( \{ x \in X : \max_{M_{k-1} \leq M \leq M_k} |\Phi_M(f,g)(x;\omega) - \Phi_{M_k}(f,g)(x;\omega)| \geq \tau \}) \geq \tau^B,\]
for each $\{ M_k = M_k(\omega) \}$, then by density there exists $\phi = \phi_{\tau,f} \in \Sigma(X)$ so that
\[ \mu( \{ x \in X : \max_{M_{k-1} \leq M \leq M_k} |\Phi_M(\phi,g)(x;\omega) - \Phi_{M_k}(\phi,g)(x;\omega)| \geq \tau/2 \}) \geq (\tau/2)^B\]
for each $\{ M_k = M_k(\omega) \}$.
\end{proof}

So, our task reduces to showing that
\begin{align}\label{e:RTgoal}
    &\{ \omega \in Y: \sup_{(X,\mu,T)} \Sigma_{\tau,B,H}^{(X,\mu,T)}(\omega) = +\infty \}) \\
    & \qquad = \bigcap_{L \geq 1} \{ \omega \in Y: \sup_{(X,\mu,T)} \Sigma_{\tau,B,H}^{(X,\mu,T)}(\omega) \geq L \} \notag
\end{align}
is $\nu$-null.

The next lemma allows us to restrict our attention from
\[ \{ \omega \in Y: \sup_{(X,\mu,T)} \Sigma_{\tau,B,H}^{(X,\mu,T)}(\omega) \geq L\} \]
to any given ergodic measure-preserving system $(X_0,\mu_0,T_0)$,
\[ \{ \omega \in Y: \Sigma_{\tau,B,H}^{(X_0,\mu_0,T_0)}(\omega) \geq L\},\]
which is a measureable set.
\begin{lemma}[Measurability of \eqref{e:bad0}, II]\label{l:II}
Suppose that $(X_0,\mu_0,T_0)$ is an ergodic measure-preserving system. Then for any other measure-preserving system $(X,\mu,T)$
\[ \Sigma_{\tau,B,H}^{(X,\mu,T)} \leq \Sigma_{\tau/2,B,H}^{(X_0,\mu_0,T_0)}
\]
and thus to prove \eqref{e:RTgoal}, it suffices to prove that the following measurable set
\[ \nu(\{ \omega \in Y : \Sigma_{\tau,B,H}^{(X_0,\mu_0,T_0)}(\omega) \geq L\}) = o_{L \to \infty;\tau,B}(1) \]
has small $\nu$-measure (possibly depending on $\tau,B$, but independent of $H$) as $L \to \infty$.
\end{lemma}
\begin{proof}
Let $(X,\mu,T)$ be arbitrary but fixed, and let $(X_0,\mu_0,T_0)$ be ergodic; our job is to prove
\[ \Sigma_{\tau,B,H}^{(X,\mu,T)}(\omega) \leq \Sigma_{\tau/2,B,H}^{(X_0,\mu_0,T_0)}(\omega).
\]
    If we let 
    \[ \mathbf{M}(X_0,X) := \{ \sigma : X_0 \to X \text{ invertible}: \mu\circ \sigma = \mu_0 \} \]
denote the set of invertible, measure-preserving transformations which send the measure $\mu_0$ to $\mu$,
then by Halmos' Theorem, \cite[p. 77-78]{H},
\[ \{ \sigma T_0 \sigma^{-1} : \sigma \in \mathbf{M}(X_0,X) \}\]
is weakly dense inside the class of measure-preserving transformations $X$. 

In particular, for each measurable subset $E$, and each $\tau \gg \epsilon >0$ sufficiently small, there exists $\sigma_0 = \sigma_{E,\epsilon} \in \mathbf{M}(X_0,X)$ so that
\[ \mu\big( (\sigma_0 T_0 \sigma_0^{-1} E_j) \triangle (T E_j) \big) \leq \epsilon \cdot 2^{- H}, \; \; \; E \triangle F := (E \smallsetminus F) \cup (F \smallsetminus E) \]
for each $0 \leq j \leq H$, where
\[ E_j := T^j E.\]
By induction, for any $m \leq H$,
\[ \max_{0 \leq j \leq J} \mu\big( (\sigma_0 T_0^m \sigma_0^{-1} E_j) \triangle (\sigma^m E_j) \big) \leq \epsilon \cdot m \cdot 2^{- H}  \leq \epsilon \cdot H^{-100} \]

But now for any $E \subset X$, if $\sigma_0 = \sigma_{E,\epsilon}$ is as above
\begin{align*} &\mu(\{ x \in X : \sup_{m \leq H} |\mathbf{1}_E(\sigma_0 T_0^m \sigma_0^{-1} x) - \mathbf{1}_E(T^m x)| \geq \tau/100 \}) \\
& \qquad \leq H \cdot \sup_{m \leq H} \mu(\{ x \in X : |\mathbf{1}_E(\sigma_0 T_0^m \sigma_0^{-1}x) - \mathbf{1}_E(T^mx)| \geq \tau/100 \}) \\
& \qquad \qquad \leq \tau^{-1} \cdot \epsilon \cdot H^{-10}.
\end{align*}

Turning to the problem at hand, let $(X,\mu,T)$ be arbitrary, and choose an appropriate 
\[ f = \mathbf{1}_F:X \to \{0,1\}\]
and sequence $M_0(\omega)\ll M_1(\omega) \ll \dots \ll M_K(\omega) \leq H/100$ so that for each $1 \leq k \leq K$ (with $K \lesssim_H 1$ maximal)
\[ \mu( \{ \sup_{M_{k-1}(\omega) \leq M \leq M_{k}(\omega)} |\Phi_M(f,g)(x;\omega) - \Phi_{M_k}(f,g)(x;\omega)| \geq \tau \}) > \tau^B \]
Now, let $\epsilon = o(\tau^{B+1})$ be a small parameter, and for $\sigma_0=\sigma_{F,\epsilon}$ as above, define
\[ \Phi_M'(f,g)(x;\omega) := \frac{1}{M} \sum_{m \leq M} (\sigma_0 T_0^m \sigma_0^{-1}) f(x) \cdot S^m g(\omega) \]
and observe that for each $1 \leq k \leq K$
\begin{align*} &\{ \sup_{M_{k-1}(\omega) \leq M \leq M_{k}(\omega)} |\Phi_M(f,g)(x;\omega) - \Phi_{M_k}(f,g)(x;\omega)| \geq \tau \} \\
& \qquad \subset 
\{ \sup_{M_{k-1}(\omega) \leq M \leq M_{k}(\omega)} |\Phi_M'(f,g)(x;\omega) - \Phi'_{M_k}(f,g)(x;\omega)| \geq \tau/2 \} \cup E  
\end{align*}
where
\[ E \subset \{ x \in X : \sup_{m \leq H} |\mathbf{1}_E(\sigma_0 T_0^m \sigma_0^{-1} x) - \mathbf{1}_E(\sigma^m x)| \geq \tau/100 \} \]
and so has $\mu(E) \leq \tau^{-1} \cdot \epsilon \cdot H^{-10}$, and thus for each $1 \leq k \leq K$
\begin{align*}
     \mu( \{ \sup_{M_{k-1}(\omega) \leq M \leq M_{k}(\omega)} |\Phi_M'(f,g)(x;\omega) - \Phi'_{M_k}(f,g)(x;\omega)| \geq \tau/2 \} ) &\geq \tau^B - o_{H \to \infty}(\tau^{-1} \epsilon) \\
     & \qquad = \tau^B - o_{H \to \infty}(\tau^B) \geq (\tau/2)^B.
\end{align*}

But now
\begin{align*}
&\mu( \{ X: \sup_{M_{k-1}(\omega) \leq M \leq M_{k}(\omega)} |\Phi_M'(f,g)(x;\omega) - \Phi'_{M_k}(f,g)(x;\omega)| \geq \tau/2 \} )  \\
& \qquad = 
\mu_0( \{ X_0 : \sup_{M_{k-1}(\omega) \leq M \leq M_{k}(\omega)} |\Phi_M^0(f \circ \sigma_0,g)(x;\omega) - \Phi^0_{M_k}(f \circ \sigma_0,g)(x;\omega)| \geq \tau/2 \} )  \end{align*}
where
\[ \Phi_M^0(f,g)(x;\omega) := \frac{1}{M} \sum_{m \leq M} T_0^m f(x) \cdot S^m g(\omega), \]
completing the proof.
\end{proof}

We introduce one final reduction, which will be crucial in allowing us to transfer the problem to a discrete harmonic-analytic estimate; the details are similar to the proof of Lemma \ref{l:II}.

\begin{lemma}[Reduction to Ergodic $(Y,\nu,S)$]\label{l:serg}
    It suffices to assume that $(Y,\nu,S)$ is ergodic.
\end{lemma}
\begin{proof}[Sketch]
One uses the weak density of conjugates of ergodic actions on $Y$ to approximate $S$ by conjugates of an ergodic action $S_0:Y\to Y$ so
\[ \max_{j,k \leq H} \mu\big( (\sigma_0 S_0^k \sigma_0^{-1} G_j ) \triangle (S^k G_j) \big) \leq \epsilon \cdot H^{-100}, \; \; \; G_j := S^j G\]
for an appropriate $\sigma_0 : Y \to Y$.

The argument is then concluded as above; for $(X_0,\mu_0,T_0)$ a fixed ergodic measure-preserving system, the $\nu$-measure of the set where
\[ \Sigma_{\tau,B,H}^{(X_0,\mu_0,T_0)}(\omega) \text{ with respect to $S$} \; \; \; > \; \; \;   \Sigma_{\tau/2,B,H}^{(X_0,\mu_0,T_0)}(\omega) \text{ with respect to $\sigma_0 S_0 \sigma_0^{-1}$ }
\]
is bounded above by
\begin{align*}
    &\nu(\{ \omega \in Y : \sup_{M \leq H/100, \ f: [H] \to \{0,1\}} |\frac{1}{M} \sum_{m \in [M]} f(m) \cdot \big( (\sigma_0 S_0^m \sigma_0^{-1}) - S^m\big) g(\omega)| \geq \tau/100 \}) \\
    & \qquad \leq \nu(\{ \omega \in Y : \sup_{m \leq H} |\big( (\sigma_0 S_0^m \sigma_0^{-1}) - S^m\big) g(\omega)| \geq \tau/100 \}) \\
    & \qquad = o_{H \to \infty}(\tau^{-1} \epsilon)
\end{align*} 
so that 
\begin{align*} 
&\nu(\{ \omega \in Y : \Sigma_{\tau,B,H}^{(X_0,\mu_0,T_0)}(\omega) \text{ with respect to $S$ } \geq L \})\\
& \qquad \leq \nu(\{ \omega \in Y : \Sigma_{\tau/2,B,H}^{(X_0,\mu_0,T_0)}(\omega) \text{ with respect to $\sigma_0 S_0 \sigma_0^{-1}$ } \geq L \}) + o_{H \to \infty}(\tau^{B}),
\end{align*}
at which point the argument is concluded using composition rules involving $\sigma_0$, namely, replacing $g$ with $g \circ \sigma_0$ and using the $\sigma_0$-invariance of $\nu$, as above.
\end{proof}

\end{document}